\documentclass[
a4paper 
,11pt 
,leqno 
,english
]{scrartcl} 
\usepackage{style/my-preamble-all-latex-files} 
\usepackage{style/my-preamble-article} 
\usepackage{style/my-preamble-cleverref-en}
\usepackage{style/my-generic-macros} 
\usepackage{style/my-specific-macros} 
\title{Generic transversality of travelling fronts, standing fronts, and standing pulses for parabolic gradient systems}
\author{Romain \textsc{Joly} and Emmanuel \textsc{Risler}}
\begin{document}
\hypersetup{pageanchor=false} 
\maketitle
\nnfootnote{%
\emph{2020 Mathematics Subject Classification:} 35K57, 37C20, 37C29, 37J46.\\%
\emph{Key words and phrases:} parabolic gradient systems, travelling fronts, standing fronts and pulses, homoclinic and heteroclinic orbits of Hamiltonian systems, generic transversality, Morse--Smale theorem.
}
\begin{abstract}
For nonlinear parabolic systems of the form
\[
\partial_t w(x,t) = \partial_{x}^2 w(x,t) - \nabla V\bigl(w(x,t)\bigr)
\,,
\]
the following conclusions are proved to hold generically with respect to the potential $V$: every travelling front invading a minimum point of $V$ is bistable, there is no standing front, every standing pulse is stable at infinity, the profiles of these fronts and pulses approach their limits at $\pm\infty$ tangentially to the eigenspaces corresponding to the smallest eigenvalues of $D^2V$ at these points, these fronts and pulses are robust with respect to small perturbations of the potential, and the set of their profiles is discrete. These conclusions are obtained as consequences of generic transversality results for heteroclinic and homoclinic solutions of the differential systems governing the profiles of such fronts and pulses. Among these results, it is proved that, for a generic Hamiltonian system of the form
\[
\ddot u=\nabla V(u)
\,,
\]
every asymmetric homoclinic orbit is transverse and every symmetric homoclinic orbit is elementary.
\end{abstract}
\thispagestyle{empty}
\pagestyle{empty}
\pagebreak
\tableofcontents
\pagebreak
\hypersetup{pageanchor=true} 
\pagestyle{plain}
\setcounter{page}{1}
\section{Introduction}
\label{sec:intro}
The purpose of this paper is to address the generic properties of travelling fronts and standing fronts/pulses of nonlinear parabolic systems of the form
\begin{equation}
\label{partial_differential_system}
\partial_t w(x,t) = \partial_{x}^2 w(x,t) - \nabla V\bigl(w(x,t)\bigr)
\,,
\end{equation}
where time variable $t$ and space variable $x$ are real, the spatial domain is the whole real line, the function $(x,t)\mapsto w(x,t)$ takes its values in $\rr^d$ with $d$ a positive integer, and the nonlinearity is the gradient of a \emph{potential} function $V:\rr^d\to\rr$, which is assumed to be regular (of class at least $\ccc^2$). Travelling fronts and standing fronts/pulses are the solutions of system \cref{partial_differential_system} of the form $w(x,t)=u(x-c t)$ that are stationary in a travelling ($c>0$) or standing ($c=0$) frame and that approach critical points of $V$ at the two ends of space. An insight into the main results of this paper (\cref{thm:main}, completed with \cref{thm:generic_asympt_behaviour}) is provided by the following corollary, illustrated by \cref{fig:spaghetti}. Its terms are precisely defined in the following \namecrefs{subsec:trav_fronts_stand_fronts_stand_pulses}. 
\begin{corollary}
\label{cor:main}
For a generic potential $V$ the following conclusions hold:
\begin{enumerate}
\item every travelling front invading a minimum point of $V$ is bistable;
\label{item:cor_main_bistable_front}
\item there is no standing front, and every standing pulse is stable at infinity;
\label{item:cor_main_bistable_pulse}
\item the set of all bistable travelling fronts and all standing pulses is discrete;
\label{item:cor_main_countable}
\item every travelling front and every standing pulse (considered individually) is robust with respect to small perturbations of $V$;
\label{item:cor_main_robust}
\item the profile of every bistable travelling front or standing pulse stable at infinity approaches its limit at $+\infty$ ($-\infty$) tangentially to the eigenspace corresponding to the smallest eigenvalue of $D^2V$ at this point.
\label{item:cor_main_direction_of_approach_of_limits_at_ends_of_R}
\end{enumerate}
\end{corollary}
\begin{figure}[htbp]
\centering
\resizebox{\textwidth}{!}{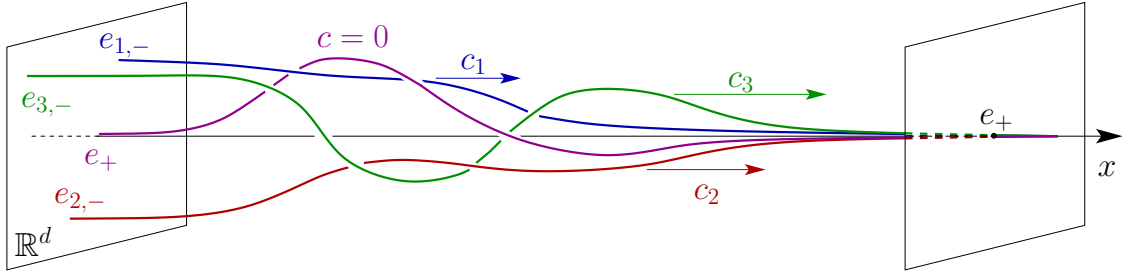}
\caption{Illustration of \cref{cor:main}. The coloured lines represent the profiles of travelling fronts or standing fronts/pulses $w_i(x,t)=u_i(x-c_i t)$ approaching a minimum point $e_+$ of a given potential at the right end of space. If this potential is generic, the critical point $e_{i,-}$ approached at the left end of space by every such profile must be a minimum point, and for the standing profile (speed $c=0$) this minimum point must be $e_+$. In addition, these profiles (up to translation of the argument) are isolated from each other, so that the set of such profiles (up to translation of the argument) is countable with respect to both speed and profile, and robust with respect to small perturbations of the potential. Additionally (conclusion \cref{item:cor_main_direction_of_approach_of_limits_at_ends_of_R}), these profiles approach their limits $e_+$ ($e_{i,-}$) tangentially to the eigenspace corresponding to the smallest eigenvalue of $D^2V(e_+)$ ($D^2V(e_{i,-})$), but this last feature is not displayed on the figure.}
\label{fig:spaghetti}
\end{figure}
\subsection{Travelling fronts, standing fronts and standing pulses}
\label{subsec:trav_fronts_stand_fronts_stand_pulses}
Let $c$ be a real quantity. A function 
\[
u:\rr\to\rr^d, \quad \xi\mapsto u(\xi)
\]
is the profile of a wave travelling at speed $c$ (if $c$ is nonzero), respectively a stationary solution (if $c$ equals $0$), for system \cref{partial_differential_system} if the function $w:(x,t)\mapsto u(x-ct)$ is a solution of this system, that is if $u$ is a solution of the second order differential system
\begin{equation}
\label{trav_wave_system_order_2}
\ddot u = -c\dot u + \nabla V(u) 
\,,
\end{equation}
where $\dot u$ and $\ddot u$ denote the first and second derivatives of $u$. 
Up to applying the transform $(\xi,c)\mapsto(-\xi,-c)$, which leaves system \cref{trav_wave_system_order_2} unchanged, we may assume that that the speed $c$ is nonnegative (and will always do so). Let us recall that a critical point of the potential $V$ is a point $e$ of $\rr^d$ such that $\nabla V(e)=0$, and that a non-degenerate local minimum point of $V$ is a critical point $m$ of $V$ such that $D^2V(m)$ is positive definite. If $e_-$ and $e_+$ are two critical points of $V$, and if $u$ is a \emph{non-constant} global solution of system \cref{trav_wave_system_order_2} such that the following limits hold
\begin{equation}
\label{approach_to_crit_points_at_ends_of_R}
u(\xi)\xrightarrow[\xi\to -\infty]{} e_-
\quad\text{and}\quad
u(\xi)\xrightarrow[\xi\to -\infty]{} e_+
\,,
\end{equation}
then the solution $(x,t)\mapsto u(x-ct)$ of \cref{partial_differential_system} is said to \emph{connect $e_-$ to $e_+$} and is called:
\begin{itemize}
\item a \emph{travelling front} if $c\neq 0$ and $e_-\neq e_+$,
\item a \emph{standing front} if $c=0$ and $e_-\neq e_+$,
\item a \emph{standing pulse} if $c=0$ and $e_-=e_+$. 
\end{itemize}
In addition, a travelling or standing front connecting a critical point $e_-$ to a critical point $e_+$ is said to be \emph{bistable} if both these critical points are non-degenerate (local or global) minimum points of $V$. Accordingly, a standing pulse connecting a critical point to itself is said to be \emph{stable at infinity} if this critical point is a non-degenerate (local or global) minimum point of $V$. Among standing pulses, it is relevant to distinguish \emph{symmetric pulses}, which are even with respect to some time (the solution goes back and forth following the same path), from \emph{asymmetric pulses} which are not.

Travelling fronts and standing fronts and pulses can be interpreted in terms of energy as follows. Let us denote by $\tilde{V}$ the opposite potential $-V$. Then, in system \cref{trav_wave_system_order_2} (where the argument $\xi$ plays the role of a time), the speed plays the role of a damping coefficient, and the nonlinear conservative force derives from the potential $\tilde{V}$. In other words, the system governs the motion of a ball rolling on the graph of $\tilde{V}$, submitted to the gravitational force and to a friction force $-c\dot u$. Its \emph{Hamiltonian energy} is the function $H_V$ defined as:
\begin{equation}
\label{Hamiltonian}
H_V: \rr^{2d}\to \rr
\,,\qquad
(u,v) \mapsto \frac{1}{2}|v|^2 - V(u)
= \frac{1}{2}|v|^2 + \tilde{V}(u)
\,,
\end{equation}
and, for every solution $\xi\mapsto u(\xi)$ of this system and every time $\xi$ where this solution is defined, the time derivative of $H_V$ along this solution reads 
\begin{equation}
\label{time_derivative_Hamiltonian}
\frac{d}{d\xi}H_V\bigl(u(\xi),\dot u(\xi)\bigr) = -c |\dot u(\xi)|^2
\,.
\end{equation}
As a consequence, if such a solution satisfies the limits \cref{approach_to_crit_points_at_ends_of_R},  
\begin{itemize}
\item if $c$ is positive then $e_-$ and $e_+$ must differ and $V(e_-)$ must be smaller than $V(e_+)$; then,
\begin{itemize}
\item from the point of view of the parabolic system \cref{partial_differential_system}, the travelling front will be said to \emph{invade} the ``higher'' (with respect to $V$) critical point $e_+$ (which is ``replaced'' with the ``lower'' one $e_-$); 
\item from the point of view of the Hamiltonian system \cref{trav_wave_system_order_2} the damping ``absorbs'' the positive lag $\tilde{V}(e_-)-\tilde{V}(e_+)$ (the ``higher'' critical point with respect to $\tilde{V}$ is $e_-$ and the ``lower'' one is $e_+$);
\end{itemize}
\item and if $c$ is zero then $e_-$ and $e_+$ must belong to the same level set of $V$.
\end{itemize}
In addition, as explained on \cref{fig:two_dim_heterocline,fig:two_dim_sym_homocline,fig:two_dim_asym_homocline}, the mechanical interpretation provides an intuitive explanation of \cref{cor:main}. 
\begin{figure}[htbp]
\centering
\resizebox{0.6\textwidth}{!}{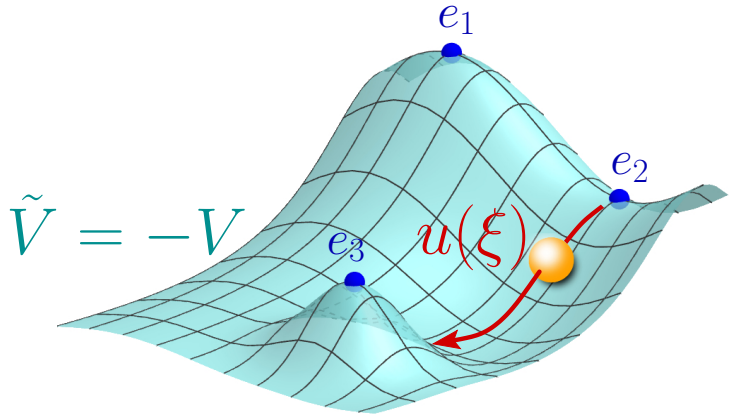}
\caption{Heteroclinic connections between critical points belonging to different level sets of $V$ for system \cref{trav_wave_system_order_2} (dimension $d$ equals $2$). This system governs the motion of a ball rolling on the surface $u\mapsto \tilde{V}(u)=-V(u)$, submitted to the gravitational force and to a friction force $-c\dot u$. The minimum points $e_1$ and $e_3$ of $V$ are maximum points for $-V$, whereas $e_2$ denotes a saddle point. A travelling front connecting $e_1$ or $e_2$ to $e_3$ corresponds to the ball leaving $e_1$ or $e_2$ with speed zero at time $-\infty$, and rolling towards $e_3$ with the suitable damping $c$ such that is reaches $e_3$ at rest when time goes to $+\infty$. Roughly speaking, this asymptotic behaviour in the future requires two conditions: the right direction (towards $e_3$) and the right damping (to reach $e_3$ and stop there). As can be intuitively seen on the figure, starting from $e_1$ provides two degrees of freedom (direction and damping), whereas starting from $e_2$ provides only one (damping). For that reason, connections between $e_1$ and $e_3$ (bistable travelling fronts invading $e_3$) are expected to occur generically and to be a robust feature, by contrast with connections between $e_2$ and $e_3$ (non bistable travelling fronts invading $e_3$), which should occur only for rare potentials. Conclusions \cref{item:cor_main_bistable_front,item:cor_main_countable} of \cref{cor:main} above and \Cref{thm:main} below formally confirm these expectations.}
\label{fig:two_dim_heterocline}
\end{figure}
\begin{figure}[p]
\centering
\resizebox{0.44\textwidth}{!}{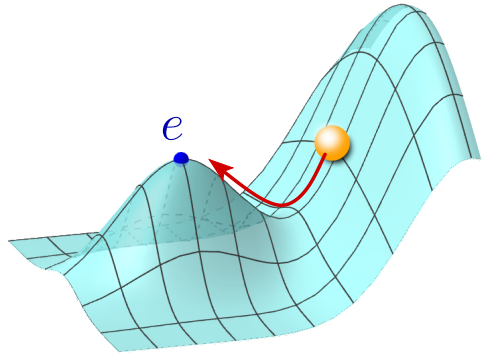}\\[2mm]
\caption{A symmetric standing pulse. A ball is dropped with speed zero at the same level of $V$ as the critical point $e$, and there is no damping. Because the Hamiltonian (energy) is conserved, reaching $e$ as time goes to $+\infty$ only requires to adjust the ``direction'' towards $e$. If $e$ is a minimum point of $V$ (a maximum point of $-V$) as on the figure, this condition can be fulfilled by choosing the adequate dropping point on the one-dimensional level set $V^{-1}\bigl(\{e\}\bigr)$. If by contrast $e$ was a saddle point, the dropping point should \emph{also} lie on the one-dimensional stable manifold of $e$, adding an additional condition. For that reason, symmetric standing pulses stable at infinity are expected to be a generic and robust feature, whereas those not stable at infinity should not occur but for rare potentials. Conclusions \cref{item:cor_main_bistable_pulse,item:cor_main_countable} of \cref{cor:main} above and \cref{thm:main} below confirm these expectations.}
\label{fig:two_dim_sym_homocline}
\end{figure}
\begin{figure}[p]
\centering
\resizebox{0.48\textwidth}{!}{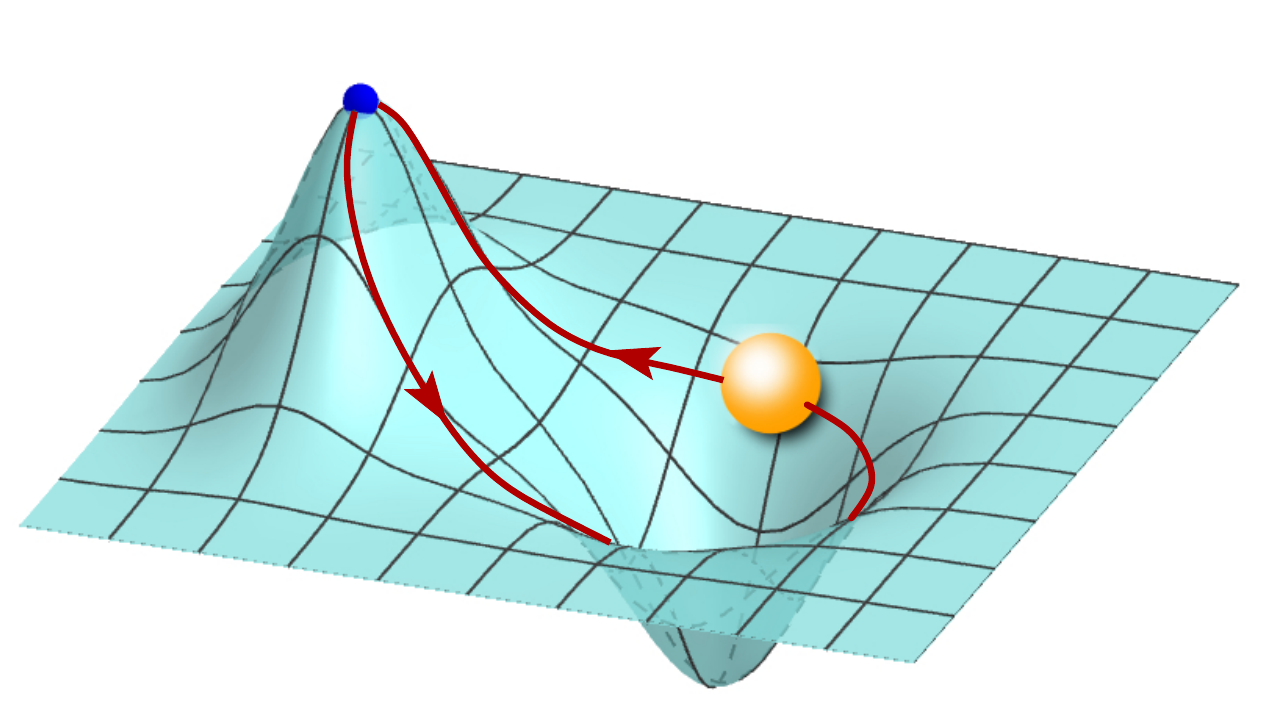}\\[2mm]
\caption{An asymmetric standing pulse. A ball ``leaves'' the critical point $e$ with speed zero at time $-\infty$, and there is no damping. Because the Hamiltonian (energy) is conserved, going back to $e$ as time goes to $+\infty$ only requires to adjust the ``direction'' towards $e$. If $e$ is a minimum point of $V$ (a maximum point of $-V$) as on the figure, this condition can be fulfilled by leaving $e$ in the adequate direction. If by contrast $e$ was a saddle point, there would be no choice for the direction of leaving (and in addition, going back to $e$ would require to do so through a particular direction). For that reason, asymmetric standing pulses stable at infinity are expected to be a generic and robust feature, whereas those not stable at infinity should not occur but for rare potentials. Conclusions \cref{item:cor_main_bistable_pulse,item:cor_main_countable} of \cref{cor:main} above and \cref{thm:main} below confirm these expectations.}
\label{fig:two_dim_asym_homocline}
\end{figure}
\subsection{Differential system governing the profiles of fronts and pulses}
Keeping the previous notation, let us consider the vector field
\begin{equation}
\label{vector_field}
F_{c,V}: \rr^{2d}\to\rr^{2d}, 
\qquad
\begin{pmatrix}
u \\ v 
\end{pmatrix}
\mapsto 
\begin{pmatrix}
v \\ \nabla V(u)-cv 
\end{pmatrix}
\,.
\end{equation}
The second order differential system \cref{trav_wave_system_order_2} is equivalent to the first order differential system
\begin{equation}
\label{trav_wave_system_order_1}
\left\{
\begin{aligned}
\dot u &= v \\
\dot v &= \nabla V(u)-cv
\end{aligned}
\right.
\quad\text{or equivalently}\quad
\dot U = F_{c,V}(U) \,\text{ with }\, U=(u,v)\in \rr^{2d}\,.
\end{equation}
A point $E$ of $\rr^{2d}$ is an equilibrium point of system \cref{trav_wave_system_order_1} if and only if there exists a critical point $e$ of $V$ such that $E$ equals $(e,0)$. Assume that $e$ is \emph{non-degenerate}, or in other words that $0$ is not in the spectrum of the symmetric matrix $D^2V(e)$. Let $\Ws_{c,V}(E)$ and $\Wu_{c,V}(E)$ denote the \emph{stable} and \emph{unstable manifolds} of $E$ for the differential system \cref{trav_wave_system_order_1}. 
Recall that each of these manifolds is defined as the union of the images of the solutions $\xi\mapsto U(\xi)$ that converge to $E$ at an exponential rate as $\xi$ goes to $+\infty$/$-\infty$, tangentially to the stable/unstable linear space of this equilibrium (see \cref{sec:dynamics}). The following statement, proved in \cref{subsec:proof_prop_equivalence_fronts_pulses_connections}, formalizes the correspondence between the profiles of travelling fronts and standing fronts/pulses and the intersections between such manifolds.
\begin{proposition}
\label{prop:profiles_of_fronts_pulses_as_intersections_between_Wu_Ws}
Let $e_-$ and $e_+$ be two (possibly equal) non-degenerate critical points of $V$, let $E_-$ and $E_+$ denote the corresponding equilibria for system \cref{trav_wave_system_order_1}, and let $c$ denote a real (zero or nonzero) quantity. 
For every profile $\xi\mapsto u(\xi)$ of a front/pulse connecting $e_-$ to $e_+$ and travelling at speed $c$ (or standing if $c$ equals zero), the image of the corresponding solution $\xi\mapsto \bigl(u(\xi),\dot u(\xi)\bigr)$ of system \cref{trav_wave_system_order_1} belongs to $\Wu_{c,V}(E_-)\cap\Ws_{c,V}(E_+)$. 
\end{proposition}
The meaning of this proposition is twofold. First, it states that the convergence of $u(\xi)$ towards $e_\pm$ at $\pm\infty$ yields the convergence of $\bigl(u(\xi),\dot u(\xi)\bigr)$ towards $(e_\pm,0)$. In other words, every profile of a travelling or standing front of the partial differential system \cref{partial_differential_system} corresponds to a heteroclinic orbit of system \cref{trav_wave_system_order_1}, and every profile of a standing pulse corresponds to a homoclinic orbit of this system. Second, those convergences occur at an exponential rate, thus not along a centre manifold (which exists for a non-degenerate critical point which is not a minimum point when $c$ vanishes, see \cref{subsec:eigenspaces_and_dimensions}).
\subsection{Transversality of fronts and pulses}
Usually, the transversality of a heteroclinic orbit connecting two equilibria $E_-$ and $E_+$ is defined as the transversality of the intersection between the unstable manifold of $E_-$ and the stable manifold of $E_+$. For travelling fronts, however, the freedom of moving the speed $c$ must be taken into account, and leads to the following definition.
\begin{definition}[transversality of a travelling front]
\label{def:transverse_travelling_front}
Let $e_-$ and $e_+$ be two non-degenerate critical points of $V$ and let $E_-$ and $E_+$ denote the corresponding equilibria for system \cref{trav_wave_system_order_1}. A front with profile $\xi\mapsto u(\xi)$ travelling at a positive speed $c$ and connecting $e_-$ to $e_+$ is said to be \emph{transverse} if the intersection
\[
\left(\bigcup_{c'>0} \{c'\}\times \Wu_{c',V}(E_-) \right) 
\cap 
\left(\bigcup_{c'>0} \{c'\}\times \Ws_{c',V}(E_+) \right) 
\]
is transverse, in $\rr^{2d+1}$, along the set $\{c\}\times U(\rr)$.
\end{definition}
For a standing pulse (connecting a critical point $e$ of $V$ to itself) the speed $c$ equals $0$, so that system \cref{trav_wave_system_order_1} preserves the Hamiltonian $H_V$ defined in \cref{Hamiltonian}. As a consequence, the stable and unstable manifolds of the equilibrium $E$ corresponding to $e$ belong to the same level set of $H_V$, so that the transversality between those manifolds cannot hold in $\rr^{2d}$, but only in this level set (which is a $2d-1$-manifold of class $\ccc^{k+1}$ outside of the set of equilibria). This leads to the following definition.
\begin{definition}[transversality of a standing pulse]
\label{def:transverse_stand_pulse}
Let $e$ denote a non-degenerate critical point of $V$ and let $E=(e,0)$. A standing pulse with profile $\xi\mapsto u(\xi)$ and connecting $e$ to itself is said to be \emph{transverse} if the intersection
\[
\Wu_{0,V}(E)\cap\Ws_{0,V}(E)
\]
is transverse, inside the level set $H_V^{-1}\bigl(-V(e)\bigr)$ deprived of $E$, along the trajectory $U(\rr)$.
\end{definition}
As mentioned above, standing pulses divide into two classes (symmetric and asymmetric, see \cref{fig:two_dim_sym_homocline,fig:two_dim_asym_homocline}), which will require separate treatments in the proofs. Here is a more precise definition. 
\begin{definition}[symmetric standing pulse, turning time]
\label{def:symmetric_pulse}
Let $e$ denote a non-degenerate critical point of $V$. A standing pulse with profile $\xi\mapsto u(\xi)$ connecting $e$ to itself is said to be \emph{symmetric} if there exists a time $\xiTurn$, called the \emph{turning time} of the pulse, such that $\dot u(\xiTurn)$ vanishes; or equivalently, such that $U(\xiTurn)$ belongs to $\rr^d\times\{0_{\rr^d}\}$. This subspace $\rr^d\times\{0_{\rr^d}\}$, often called the \emph{reversibility} or \emph{symmetry} subspace, will be denoted by $\sssSym$. 
\end{definition}
If such a turning time exists then it is unique and the profile of the pulse is indeed symmetric with respect to this turning time, see \cref{lem:equivalent_def_symm_hom_orbit}. Note that in the scalar case $d=1$ every standing pulse is symmetric (the derivative $\dot u$ must vanish if the solution approaches the same limits at both ends of $\rr$). 
For symmetric standing pulses (for any value of the dimension $d$), instead of the transversality defined in \cref{def:transverse_stand_pulse}, the following weaker transversality property (\cite{Devaney_reversibleDiffeoFlows_1976,FiedlerVanderbauwhede_homoclinicPeriodBlowUpRevConSyst_1992,Harterich_cascadesReversibleHomoclinicOrbitsSaddleFocus_1998}) will be required. 
\begin{definition}[elementary symmetric standing pulse]
\label{def:elementary_symm_stand_pulse}
Assume that the standing pulse $\xi\mapsto u(\xi)$ is symmetric with turning time $\xiTurn$. This pulse is said to be \emph{elementary} if the intersection
\[
\Wu_{0,V}(E)\cap\sssSym
\]
is transverse, in $\rr^{2d}$, at the point $U(\xiTurn)$. The feature of being elementary, for a standing pulse, will be called \emph{elementarity}.
\end{definition} 
Note that every transverse symmetric standing pulse is elementary: due to the time reversibility when $c$ is zero, a non-transverse intersection between $\Wu_{0,V}(E)$ and $\sssSym$ induces a non-transverse intersection between $\Wu_{0,V}(E)$ and $\Ws_{0,V}(E)$. But the converse is false: for a symmetric standing pulse, the intersection $\Wu_{0,V}(E)\cap\Ws_{0,V}(E)$ may be non-transverse in the sense of \cref{def:transverse_stand_pulse} while this intersection still crosses transversally the reversibility subspace $\sssSym$. This may occur, for instance, if a symmetric standing pulse is the limit of a one-parameter family of asymmetric ones. 
\subsection{The space of potentials}
\label{subsec:space_of_potentials}
For the remaining of the paper, let us take and fix an integer $k\ge1$. Let us denote by $\CkbFull$ the space of functions $\rr^d\to\rr$ of class $\ccc^{k+1}$ which are bounded, as well as their derivatives of order not larger than $k+1$, equipped with the norm
\[
\norm{W}_{\CkbShort} = \max_{\alpha\text{ multi-index, }|\alpha|\leq k+1} \|\partial^{|\alpha|}_{u_\alpha} W\|_{L^\infty(\rr^d,\rr)}
\,. 
\]
Let us embed the larger space $\CkFull$ with the following topology: for $V$ in this space, a basis of neighbourhoods of $V$ is given by the sets $V+\ooo$, where $\ooo$ is an open subset of $\CkbFull$ embedded with the topology defined by $\norm{\cdot}_{\CkbShort}$. This topology (which can be viewed as the one of an extended metric) is convenient, since local properties may be studied in Banach spaces of the form 
\[
V + \left(\CkbFull,\normCkb\right)
\,,
\]
with $\normCkb$ viewed as a classical norm. In this paper, the space $\CkFull$ will always be embedded with this topology (if a topology is needed) and $\Big(\CkFull,\normCkb\Big)$ will be denoted simply by $\CkFull$. 

Let us recall that a subset $A$ of a topological set $B$ is said to be a \emph{generic subset of $B$} if it contains a countable intersection of dense open subsets of $B$; accordingly, a property is said to \emph{hold for a generic potential} if it holds for every potential in a generic subset of $\CkFull$. 
It is important to notice that $\CkFull$ is a Baire space because it is locally equal to the Baire space $\CkbFull$. 
Thus, the notion of genericity provides relevant definitions of ``large'' subsets and ``almost everywhere satisfied'' properties. Other definitions exist and the results stated in this paper presumably still hold for those (the interested reader may consider \cite{Aronszajn_differentiabilityLipschitzianMappingsBanachSpaces_1976,OttYorke_prevalence_2005,Joly_adaptationGenericPDEResultsPrevalence_2007,BernardMandorino_someRemarksOnThomTransversalityTheorem_2015}).

Actually, the results stated in this paper also hold with other natural topologies, such as Whitney's topology. However the space $\CkFull$ is not locally a metric space for Whitney's topology (which is not characterized by sequences) and this leads to technical difficulties. The framework chosen above is thus convenient to state the main arguments while avoiding unessential technicalities, but the choice of the topology is not a key issue.

To finish, let us recall that a function having only non-degenerate critical points is commonly called a \emph{Morse function}. According to a classical result (see for instance \cite{Hirsch_differentialTopology_1976}), the set of Morse functions is a generic subset of $\CkFull$. Since the intersection of two generic subsets is still a generic subset, and since our purpose is to state results which hold generically, assuming that the potential $V$ under consideration is a Morse function does not present any inconvenience. As a consequence, only nondegenerate critical points will be considered in the following, and the potential $V$ will often be assumed to be a Morse function. 
\subsection{Main results}
The following generic transversality statement is the main result of this paper. 
\begin{theorem}[generic transversality of fronts and pulses]
\label{thm:main}$~$\\
There exists a generic subset of $\CkFull$ such that, for every potential function $V$ in this subset, $V$ is a Morse function and the following conclusions hold for the fronts and pulses defined by $V$: 
\begin{enumerate}
\item every travelling front invading a minimum point of $V$ is transverse;
\label{item_thm:main_travelling_front}
\item every symmetric standing pulse is elementary;
\label{item_thm:main_symmetric_standing_pulse}
\item every asymmetric standing pulse is transverse; 
\label{item_thm:main_asymmetric_standing_pulse}
\item there is no standing front.
\label{item_thm:main_standing_front}
\end{enumerate}
\end{theorem}
The core of this paper (\cref{sec:generic_transversality_travelling_fronts,sec:generic_elementarity_sym_stand_pulses,sec:generic_tranversality_asym_stand_pulses,sec:gen_non_existence_standing_fronts}) is devoted to the proof of this result among potentials which are quadratic past some radius (see their definition in \cref{notation_vvvQuad_of_R}). For such potentials, conclusion \cref{item_thm:main_travelling_front} is proved by \cref{prop:global_generic_transversality_travelling_fronts}, conclusion \cref{item_thm:main_symmetric_standing_pulse}, by \cref{prop:global_generic_elementarity_sym_stand_pulses}, conclusion \cref{item_thm:main_asymmetric_standing_pulse} by \cref{prop:global_generic_transversality_asym_stand_pulses}, and conclusion \cref{item_thm:main_standing_front} by \cref{prop:gen_non_existence_standing_fronts}. \Cref{sec:generic_transversality_travelling_fronts,sec:generic_elementarity_sym_stand_pulses} are devoted, respectively, to the proofs of these propositions. In \cref{sec:proof_main}, the proof of \cref{thm:main} is completed by extending these conclusions to general potentials of $\CkFull$ (not necessarily quadratic past some radius), and the qualitative conclusions \cref{item:cor_main_bistable_front,item:cor_main_bistable_pulse,item:cor_main_countable,item:cor_main_robust} of \cref{cor:main} are derived from \cref{thm:main}. 

Using the same techniques, the following extension of \cref{thm:main} (and of conclusions \cref{item:cor_main_bistable_front,item:cor_main_bistable_pulse,item:cor_main_countable,item:cor_main_robust} of \cref{cor:main}) is proved in \cref{sec:generic_asympt_behaviour}. The second conclusion of this extension is nothing but the last conclusion \cref{item:cor_main_direction_of_approach_of_limits_at_ends_of_R} of \cref{cor:main}. 
\begin{theorem}
\label{thm:generic_asympt_behaviour}
There exists a generic subset of $\CkFull$ such that, for every potential function $V$ in this subset, in addition to the conclusions of \cref{thm:main} (and to the conclusions \cref{item:cor_main_bistable_front,item:cor_main_bistable_pulse,item:cor_main_countable,item:cor_main_robust} of \cref{cor:main}), the following two additional conclusions hold:
\begin{enumerate}
\item for every minimum point of $V$, the smallest eigenvalue of the Hessian $D^2V$ at this minimum point is simple; 
\label{item:thm_generic_asympt_behaviour_simple_smallest_eig}
\item the profile of every bistable travelling front or standing pulse stable at infinity approaches its limit at $+\infty$ ($-\infty$) tangentially to the eigenspace corresponding to the smallest eigenvalue of $D^2V$ at this point. 
\label{item:thm_generic_asympt_behaviour_approach_limits}
\end{enumerate}
\end{theorem}
As explained in \cref{subsec:trav_fronts_stand_fronts_stand_pulses}, conclusions \cref{item_thm:main_symmetric_standing_pulse,item_thm:main_asymmetric_standing_pulse,item_thm:main_standing_front} of \cref{thm:main} can be interpreted in terms of homoclinic and heteroclinic orbits of the Hamiltonian system
\begin{equation}
\label{equation_hamiltonienne}
\left\{\begin{aligned} 
\dot u &= \partial_v \mathcal{H}(u,v) \\ 
\dot v &= -\partial_u \mathcal{H}(u,v) 
\end{aligned}\right. 
\quad \text{where}\quad 
\mathcal{H}(u,v)=\frac12 |v|^2 + \tilde V(u)
\quad \text{and}\quad 
\tilde{V} = -V
\,.
\end{equation}
The following statement explicitly provides this interpretation (for conclusions \cref{item_thm:main_symmetric_standing_pulse,item_thm:main_asymmetric_standing_pulse} only, since conclusion \cref{item_thm:main_standing_front} is actually elementary and well known, see \cref{sec:gen_non_existence_standing_fronts}). No proof is given since it is exactly a translation of these conclusions, with obvious meanings for (a)symmetry and elementarity of homoclinic orbits.
\begin{theorem}[the Hamiltonian point of view]
\label{thm:main2}
There exists a generic subset of $\CkFull$ such that, for every potential function $\tilde V$ in this subset,
\begin{enumerate}
\item every asymmetric homoclinic orbit of the Hamiltonian system \cref{equation_hamiltonienne} is transverse;
\label{item:thm_Ham_asym_hom}
\item every symmetric homoclinic orbit of the Hamiltonian system \cref{equation_hamiltonienne} is elementary.
\label{item:thm_Ham_sym_hom}
\end{enumerate}
\end{theorem}
\subsection{Short historical review}
\Cref{thm:main} and its proof rely on transversality theorems, also known as Sard--Smale or Thom's theorems, and are closely related to classical transversality results for differential systems, see for instance \cite{AbrahamRobbin_transversalMappingsFlows_1967,Kupka_contributionTheorieChampsGeneriques_1963,Peixoto_onApprroximationTheoremKupkaSmale_1967,Robbin_algebraicKupkaSmaleTheory_1981,Smale_stableManifolsDiffEquDiffeo_2011}. Significant differences with respect to previous works still deserve to be mentioned. 

First, genericity in \cref{thm:main} holds with respect to the sole potential function $V$, \emph{not} general vector fields in $\rr^{2d}$. Thus, perturbations of a given potential only provide a partial control on the dynamics (in other words, differential systems of the form \cref{trav_wave_system_order_1} do not generate all possible flows in $\rr^{2d}$). This constraint is balanced by the peculiarities of the systems considered, which will have to be taken into account. To our best knowledge, the first genericity result about the dynamics of a special class of differential systems goes back to \cite{Robbin_algebraicKupkaSmaleTheory_1981}, and deals with polynomial flows. 

Concerning Hamiltonian flows, homoclinic orbits play an important role, both from theoretical and physical points of view, see for instance the reviews \cite{Devaney_homoclinicOrbitsHamiltonianSystems_1976,Champneys_homoclinicOrbitsReversibleSystemsApplications_1998} and articles \cite{Devaney_blueSkyCatastrophesReversibleHamiltonianSystems_1977,Vanderbauwhede_bifurcationDegenerateHomoclinics_1992,FiedlerVanderbauwhede_homoclinicPeriodBlowUpRevConSyst_1992,Harterich_cascadesReversibleHomoclinicOrbitsSaddleFocus_1998,BessaFerreiraRochaVarandas_genericHamiltonianDynamics_2017,Motreanu_genericExistenceNondegenerateHomSol_2017}. The transversality/elementarity of such orbits has important dynamical consequences, as the presence of Smale horseshoes associated to complex dynamics. In \cite{Knobloch_bifurcationDegHomocOrbitsReversibleConservativeSyst_1997,Vanderbauwhede_bifurcationDegenerateHomoclinics_1992}, the genericity of these properties is considered in a general abstract framework, and obtained only under sufficient conditions corresponding to the assumptions of the transversality \cref{thm:Sard_Smale}. In \cite{Motreanu_genericExistenceNondegenerateHomSol_2017}, this genericity is proved, but in the case of non-autonomous systems.
Other references dealing with the generic transversality of connecting orbits include \cite{Robinson_genericPropertiesConservativeSystems_1970,Oliveira_genPropLagrangianOnSurfKupkaSmale_2008,Rifford_closingGeodesicsC1Topology_2012,RiffordRuggiero_genPropClosedOrbHamFlowsManeViewpoint_2012} and references therein. In \cite{Robinson_genericPropertiesConservativeSystems_1970}, genericity holds with respect to all Hamiltonian flows, and not only second order conservative systems as \cref{equation_hamiltonienne}. In \cite{Oliveira_genPropLagrangianOnSurfKupkaSmale_2008,Rifford_closingGeodesicsC1Topology_2012,RiffordRuggiero_genPropClosedOrbHamFlowsManeViewpoint_2012}, genericity holds with respect to the potential $\tilde V$ only, in a more general setting where the ``kinetic energy'' $\abs{v}^2/2$ of the Hamiltonian in \cref{equation_hamiltonienne} is replaced by a more general expression. But the transversality of homoclinic orbits is not considered in these papers. In \cite{Oliveira_genPropLagrangianOnSurfKupkaSmale_2008}, the transversality of heteroclinic orbits is derived from a perturbation result of \cite{ContrerasPaternain_genericityGeodesicFlowsPositiveTopologicalEntropyS2_2002}. The others results of \cite{Oliveira_genPropLagrangianOnSurfKupkaSmale_2008,Rifford_closingGeodesicsC1Topology_2012,RiffordRuggiero_genPropClosedOrbHamFlowsManeViewpoint_2012} are concerned with closed orbits. Thus, to the best of our knowledge, even \cref{thm:main2} (the results concerning standing pulses, in the language of Hamiltonian systems) is new.

Concerning the nonzero dissipation case (conclusion \cref{item_thm:main_travelling_front} of \cref{thm:main}), the statement differs from usual genericity properties. If $c$ is fixed (and nonzero), heteroclinic connections corresponding to travelling fronts invading a minimum point of $V$ do generically not exist for the flow of system \cref{trav_wave_system_order_1}. But the freedom provided by the parameter $c$ ensures the generic existence, transversality, and robustness of heteroclinic connections corresponding to bistable travelling fronts. This parameter $c$ will thus have to be taken into account in the setting where transversality theorems will be applied, a significant difference with classical genericity results about the flows of differential systems. 

The initial motivation for this paper actually relates to parabolic systems of the form \cref{partial_differential_system}. For such systems, the global dynamics of \emph{bistable solutions}, that is solutions close at both ends of space to local minimum points of the potential $V$, has been described under rather general (assumed to be generic) hypotheses on $V$ by the second author in \cite{Risler_globCVTravFronts_2008,Risler_globalRelaxation_2016,Risler_globalBehaviour_2016}. Every such solutions must approach, as time goes to $+\infty$, far to the left in space a stacked family of bistable fronts travelling to the left, far to the right in space a stacked family of bistable fronts travelling to the right, and in between a pattern of standing pulses/fronts going slowly away from one another (this extends to gradient systems the program initiated in the late seventies by Fife and McLeod for scalar equations \cite{FifeMcLeod_approachTravFront_1977,Fife_longTimeBistable_1979,FifeMcLeod_phasePlaneDisc_1981}). The present paper provides a rigorous proof of the genericity of the hypotheses made on the potential $V$ in \cite{Risler_globCVTravFronts_2008,Risler_globalRelaxation_2016,Risler_globalBehaviour_2016}. The same hypotheses yield similar conclusions for hyperbolic gradient systems \cite{Risler_globalBehaviourHyperbolicGradient_2017} and for radially symmetric solutions of parabolic gradient systems in higher space dimension \cite{Risler_globalBehaviourRadiallySymmetric_2017}. The results obtained in this last reference rely on an additional hypothesis, which is the higher space dimension analogue of conclusion \cref{item_thm:main_symmetric_standing_pulse} of \cref{thm:main} (elementarity of symmetric standing pulses). The genericity of this hypothesis is proved in the companion paper \cite{Risler_genericTransvRadSymStatSol_2023}, using the same approach as in the present paper. 

The extension \cref{thm:generic_asympt_behaviour} of \cref{thm:main} (comprising the last conclusion \cref{item:cor_main_direction_of_approach_of_limits_at_ends_of_R} of \cref{cor:main}) is motivated by the study of the long-range interaction between fronts and pulses of the parabolic system \cref{partial_differential_system}. The long-range interaction between such ``localized structures'' is the object of a large body of literature, both in Mathematics and Physics, see for instance \cite{KawasakiOtha_kinkDynamicsOneDimNonlinSyst_1982,CoulletElphick_topologicalDefectsDynamicsMelnikovTheory_1987,CarrPego_metastablePatternsSolutions_1989,Ei_motionPulses_2002,MielkeZelik_multiPulseEvolutionAndSpaceTimeChaosDissipativeSystems_2009,BethuelSmets_motionLawFrontsScalarRDEquEqualDepthMultWellPot_2017} among many other possible references. The conclusions of \cref{thm:generic_asympt_behaviour} are especially relevant in conjunction with this topic, for the following reason. Consider a solution of the parabolic system \cref{partial_differential_system} close to, say, two standing fronts or pulses or two fronts travelling at the same speed, far away from one another. Let us denote by $\uLeft(\cdot)$ and $\uRight(\cdot)$ their profiles, so that the solution is close to a translate of $\uLeft$ on $\rr_-$ and to a translate of $\uRight$ on $\rr_+$. Then, the (large) distance between these two translates is expected to vary slowly, according to a (long-range) interaction law that can be computed at first order, and which is related to the asymptotics of $\uLeft$ at $+\infty$ and of $\uRight$ at $-\infty$. Basically, when (as in the present context) the tails of $\uLeft$ and $\uRight$ are \emph{not} oscillating, this first order long-range interaction can be attractive or repulsive or neutral, depending on the sign of a scalar product between the (oriented) directions through which $\uLeft$ and $\uRight$ approach their (common) limit (at $+\infty$ and at $-\infty$ respectively), see for instance the conjecture at the bottom of p. 59 of \cite{BethuelOrlandi_slowMotion_2011}, or expressions (2.12) and (2.13) in Theorem 2.3 of \cite{Ei_motionPulses_2002}. In the present context, according to the conclusions of \cref{thm:generic_asympt_behaviour} and for a generic potential, these two oriented directions are aligned with the one-dimensional eigenspace associated with the smallest eigenvalue of the Hessian $D^2V$ of the potential at the minimum point which is the common limit mentioned above. Among the consequences, the first order long-range interaction is thus either attractive or repulsive, but not neutral. 
\section{Stable and unstable manifolds of equilibria}
\label{sec:dynamics}
Throughout all this section $V$ denotes a potential function in $\CkFull$ and $c$ denotes a non-negative quantity (speed). As stated in \cref{prop:profiles_of_fronts_pulses_as_intersections_between_Wu_Ws}, the travelling fronts and standing fronts/pulses of the parabolic equation \cref{partial_differential_system} correspond to heteroclinic and homoclinic connections for the flow in $\rr^{2d}$ generated by the first order differential system \cref{trav_wave_system_order_1}. 
Let $\Omega$ be the maximal (open) subset of $\rr\times\rr^{2d}$ where this flow is defined and let us consider its flow $S_{c,V}$ defined as
\begin{equation}
\label{flow}
S_{c,V}:\Omega\to\rr^{2d}
\,,\quad 
(\xi,U_0) \mapsto U(\xi) 
\,,
\end{equation}
where $U(\xi)$ is the solution of \cref{trav_wave_system_order_1} with $U(0)=U_0$.
By definition, for every $(\xi,U_0)$ in $\Omega$, 
\[
\frac{\partial}{\partial\xi}S_{c,V}(\xi,U_0) = F_{c,V}\bigl(S_{c,V}(\xi,U_0)\bigr)
\quad\text{ where }\quad F_{c,V}:\begin{pmatrix}
u \\ v \end{pmatrix}\mapsto\begin{pmatrix} v \\ \nabla V(u)-cv \end{pmatrix}\,.
\]
Although the variable $\xi$ denotes primarily the space variable in a frame travelling at speed $c$ for the initial partial differential system \cref{partial_differential_system}, this variable also plays the role of a time in the differential systems \cref{trav_wave_system_order_2,trav_wave_system_order_1} prescribing the profiles of travelling and standing waves. In the following, this variable will thus often be referred to as a ``time''.
\subsection{Linearization around an equilibrium point}
\label{subsec:eigenspaces_and_dimensions}
Let $e$ denote a non-degenerate critical point of $V$. Let $(u_1,\dots,u_d)$ denote an orthonormal basis of $\rr^d$ made of eigenvectors of the Hessian $D^2V(e)$ and let $\mu_1,\dots,\mu_d$ denote the corresponding (real) eigenvalues.
\begin{definition}\label{def:Morse_index}
Let us call \emph{Morse index of $e$}, denoted by $m(e)$, the number of negative eigenvalues of $D^2V(e)$, counted with their algebraic multiplicity. 
\end{definition}
Since the critical point $e$ is assumed to be non-degenerate, it is: a minimum point if $m(e)$ equals $0$, a saddle point if $m(e)$ is between $1$ and $d-1$, and a maximum point if $m(e)$ equals $d$. In addition, none of the quantities $\mu_1,\dots,\mu_d$ vanishes, and we may assume that
\[
\begin{aligned}
\mu_1 \le\dots\le\mu_{m(e)}<0<\mu_{m(e)+1}\le\dots\le\mu_d
\quad&\text{if}\quad
m(e)>0 \,,\\
\text{and}\quad
0<\mu_1 \le\dots\le\mu_d
\quad&\text{if}\quad
m(e)=0 
\,.
\end{aligned}
\]
Now, let us consider the equilibrium point $E=(e,0_{\rr^d})$ of $S_{c,V}$ corresponding to $e$. The linearized differential system \cref{trav_wave_system_order_1} at $E$ reads:
\begin{equation}
\label{trav_wave_system_order_1_linearized}
\dot U = DF_{c,V}(E) U
\,,
\quad\text{or equivalently}\quad
\left\{
\begin{aligned}
\dot u &= v \\
\dot v &= D^2V(e) u -cv
\end{aligned}
\right.
\,.
\end{equation}
Observe that a complex quantity $\lambda$ is an eigenvalue for the linear system \cref{trav_wave_system_order_1_linearized} if and only if the quantity $\lambda(\lambda+c)$ is an eigenvalue for the Hessian $D^2V(e)$, that is if $\lambda(\lambda+c)$ is equal to one of the quantities $\mu_1,\dots,\mu_d$. 
For $j$ in $\{1,\dots, d\}$, let 
\begin{equation}
\label{notation_eigenvalues_Rd_times_Rd}
\lambda_{j,+} = -\frac{c}{2}+\sqrt{\frac{c^2}{4}+\mu_j}
\quad\text{and}\quad
\lambda_{j,-} = -\frac{c}{2}-\sqrt{\frac{c^2}{4}+\mu_j}
\end{equation}
denote the two (real or complex) eigenvalues of the linear system \cref{trav_wave_system_order_1_linearized} corresponding to $\mu_j$, and let
\begin{equation}
\label{notation_eigenvectors_in_dim_2d}
U_{j,+} = \begin{pmatrix}
u_j \\ \lambda_{j,+} u_j
\end{pmatrix}
\quad\text{and}\quad
U_{j,-} = \begin{pmatrix}
u_j \\ \lambda_{j,-} u_j
\end{pmatrix}
\end{equation}
denote the corresponding eigenvectors. Let 
\begin{equation}
\label{notation_stable_centre_unstable_subspaces}
\Es_{c,V}(E)
\quad\text{and}\quad
\Ec_{c,V}(E)
\quad\text{and}\quad
\Eu_{c,V}(E)
\end{equation}
denote the stable, centre, and unstable subspaces of $\rr^{2d}$ for the linear operator $DF_{c,V}$ defined in \cref{trav_wave_system_order_1_linearized}, that is the eigenspaces corresponding to eigenvalues with negative, zero and positive real parts respectively. The dimensions of those spaces and of the corresponding invariant manifolds (defined below) derive from expressions \cref{notation_eigenvalues_Rd_times_Rd}, and are as summarized in \cref{table:dim_stable_unstable_centre}.
\begin{table}[htbp]
\centering
\begin{tabular}{|c|c|c|}
\hline
 & $c=0$ & $c>0$ \\ \hline
Dimension of $\Eu_{c,V}(E)$ and $\Wu_{c,V}(E)$ & $d-m(e)$ & $d-m(e)$ \\ \hline
Dimension of $\Es_{c,V}(E)$ and $\Ws_{c,V}(E)$ & $d-m(e)$ & $d+m(e)$  \\ \hline
Dimension of $\Ec_{c,V}(E)$ and $\WcloccV(E)$ & $2m(e)$ & $0$ \\ \hline
\end{tabular}
\caption{Dimensions of stable, unstable, and centre manifolds for an equilibrium point $E=(e,0)$ of the differential system \cref{trav_wave_system_order_1}, corresponding to a critical point $e$ of the potential with Morse index $m(e)$.}
\label{table:dim_stable_unstable_centre}
\end{table}
The case of a negative speed $c$ can be derived by the transformation $(c,\xi)\mapsto (-c,-\xi)$ which leaves the systems \cref{trav_wave_system_order_2,trav_wave_system_order_1} unchanged (and exchanges the stable and unstable dimensions).

The dimension of $\Eu_{c,V}(E)$ is also commonly called the Morse index of $E$. To avoid any confusion, the denomination \emph{Morse index} will only be used for critical points of the potential, not for the corresponding equilibria in $\rr^{2d}$.
\subsection{Local stable and unstable manifolds when the speed \texorpdfstring{$c$}{c} is positive}
\label{subsec:local_stable_unstable_manifolds_positive_speed}
The construction of the local stable (unstable) manifold of an equilibrium of a differential system is classical. A historical reference is Kelley's article \cite{Kelley_stableCenterUnstableMan_1967}, comprising the construction and the dependence on the parameters, however with a slightly non-optimal regularity. A complete construction can be found in many textbooks, for example Theorem 3.2.1 of \cite{GuckenHeimerHolmes_nonlinOscDynSystBif_1983} or Theorem 9.4 of \cite{Teschl_odeDynSyst_2012}. The goal of this \namecref{subsec:local_stable_unstable_manifolds_positive_speed} and of \cref{subsec:local_stable_unstable_manifolds_zero_speed} below is to provide precise statements (\cref{prop:loc_stab_unstab_man_c_positive} below and \cref{prop:loc_stab_unstab_man_c_equals_zero} in \cref{subsec:local_stable_unstable_manifolds_zero_speed} when the speed $c$ equals $0$) concerning these manifolds (for the differential system \cref{trav_wave_system_order_1}), and the associated notation (without proofs); those statements and notation will be called upon in the sequel. 

Take $V_0$ in $\CkFull$, let $e_0$ denote a non-degenerate critical point of $V_0$, and let $c_0$ denote a positive quantity. According to \cref{table:dim_stable_unstable_centre}, the point $(e_0,0)$, which will be denoted by $E_0$, is a hyperbolic equilibrium point and the subspaces $\Eu_{c_0,V_0}(E_0)$ and $\Es_{c_0,V_0}(E_0)$ introduced in \cref{notation_stable_centre_unstable_subspaces} generate the whole space $\rr^{2d}$ (or in other words the central part $\Ec_{c_0,V_0}(E_0)$ reduces to $\{0_{\rr^{2d}}\}$). Let 
\begin{equation}
\label{betas}
\begin{aligned}
\betau &= \min\left\{\ree(\lambda): \lambda \text{ eigenvalue of } DF_{c_0,V_0}(E_0)\text{ with }\ree(\lambda)>0 \right\} 
 \\ \text{and}\quad
\betas &= \max\left\{\ree(\lambda): \lambda \text{ eigenvalue of } DF_{c_0,V_0}(E_0)\text{ with }\ree(\lambda) <0\right\}
\,.
\end{aligned}
\end{equation}
There exist norms $\normu{\cdot}$ on the unstable subspace $\Eu_{c_0,V_0}(E_0)$ and $\norms{\cdot}$ on the stable subspace $\Es_{c_0,V_0}(E_0)$ such that, for every non negative quantity $\xi$, 
\begin{equation}
\label{normu_norms_characterization}
\begin{aligned}
\normu{\exp\Bigl[-\xi DF_{c_0,V_0}(E_0)_{|\Eu_{c_0,V_0}(E_-)}\Bigr]} &\le \exp\left(-\frac{\betau}{2}\xi\right) \,, \\
\text{and}\quad
\norms{\exp\Bigl[\xi DF_{c_0,V_0}(E_0)_{|\Es_{c_0,V_0}(E_+)}\Bigr]} &\le \exp\left(\frac{\betas}{2}\xi\right)
\,.
\end{aligned}
\end{equation}
For every positive quantity $r$, let 
\begin{equation}
\label{def_Bu_Bs_B_around_E0}
\begin{aligned}
\Bu_{E_0}(r) &= \bigl\{\Uu \in \Eu_{c_0,V_0}(E_0) : \normu{\Uu}\le r\bigr\}\,, \\
\text{and}\quad
\Bs_{E_0}(r) &= \bigl\{\Us \in \Es_{c_0,V_0}(E_0) : \norms{\Us}\le r\bigr\}\,, \\
\text{and}\quad
\widebar{B}_{E_0}(r) &= \bigl\{\Uu+\Us : \Uu\in\Bu_{E_0}(r) \text{ and } \Us\in \Bs_{E_0}(r) \bigr\}
\,.
\end{aligned}
\end{equation}
\begin{proposition}[local stable and unstable manifolds]
\label{prop:loc_stab_unstab_man_c_positive}
There exist a neighbourhood $\nu$ of $V_0$ in $\CkFull$, a neighbourhood $\ccc$ of $c_0$ in $(0,+\infty)$ and a positive quantity $r$ such that, for every $(c,V)$ in $\ccc\times\nu$, the following statements hold.
\begin{enumerate}
\item There exists a unique critical point $e(V)$ of $V$ such that $E(V)=(e(V),0)$ belongs to $E_0+\widebar{B}_{E_0}(r)$. In addition, $e(V)$ has the same Morse index as $e_0$ and the map $\nu\to\rr^d$, $V\mapsto e(V)$ is of class $\ccc^k$.
\item There exist $\ccc^k$-functions
\[
\wuloc{c}{V}:\Bu_{E_0}(r)\to \Bs_{E_0}(r)
\quad\text{and}\quad
\wsloc{c}{V}:\Bs_{E_0}(r)\to \Bu_{E_0}(r)
\]
such that, if we consider the sets
\[
\begin{aligned}
\Wuloc{c}{V}\bigl(E(V)\bigr) &= \left\{E(V) + \Uu + \wuloc{c}{V}(\Uu) : \Uu\in \Bu_{E_0}(r)\right\} \\ 
\text{and}\quad
\Wsloc{c}{V}\bigl(E(V)\bigr) &= \left\{E(V) + \Us + \wsloc{c}{V}(\Us) : \Us\in \Bs_{E_0}(r)\right\}
\,,
\end{aligned}
\]
then, for every $U$ in $\widebar{B}_{E_0}(r)$ the following two assertions are equivalent:
\begin{enumerate}
\item $U$ is in $\Wuloc{c}{V}\bigl(E(V)\bigr)$;
\label{item:U_is_in_Wuloc}
\item $S_{c,V}(\xi,U)-E(V)$ remains in $\widebar{B}_{E_0}(r)$ for all $\xi$ in $(-\infty,0]$ and $S_{c,V}(\xi,U)\to E(V)$ as $\xi\to -\infty$;
\label{item:U_goes_to_E_at_minus_infty}
\end{enumerate}
and for every $U$ in $\widebar{B}_{E_0}(r)$ the following two assertions are equivalent:
\begin{enumerate}
\setcounter{enumii}{2}
\item $U\in\Wsloc{c}{V}\bigl(E(V)\bigr)$;
\label{item:U_is_in_Wsloc}
\item $S_{c,V}(\xi,U)-E(V)$ remains in $\widebar{B}_{E_0}(r)$ for all $\xi$ in $[0,+\infty)$ and $S_{c,V}(\xi,U)\to E(V)$ as $\xi\to +\infty$.
\label{item:U_goes_to_E_at_plus_infty}
\end{enumerate}
\item Both differentials $D\wuloc{c_0}{V_0}(0)$ and $D\wsloc{c_0}{V_0}(0)$ vanish, and both maps 
\[
\begin{aligned}
\ccc\times\nu\times \Bu_{E_0}(r)\to \Bs_{E_0}(r), &\quad (c,V,\Uu)\mapsto \wuloc{c}{V}(\Uu) \\
\text{and}\quad
\ccc\times\nu\times \Bs_{E_0}(r)\to \Bu_{E_0}(r), &\quad (c,V,\Us)\mapsto \wsloc{c}{V}(\Us)
\end{aligned}
\]
are of class $\ccc^k$. 
\end{enumerate}
\end{proposition}
With the notation provided by \cref{prop:loc_stab_unstab_man_c_positive}, for every $(c,V)$ in $\ccc\times\nu$, let us introduce the maps
\[
\begin{aligned}
\hatwuloc{c}{V}:\Bu_{E_0}(r)\to\rr^{2d}, &\quad \Uu\mapsto E(V) + \Uu + \wuloc{c}{V}(\Uu) \,,\\
\text{and}\quad
\hatwsloc{c}{V}:\Bs_{E_0}(r)\to\rr^{2d}, &\quad \Us\mapsto E(V) + \Us + \wsloc{c}{V}(\Us)
\,.
\end{aligned}
\]
Local unstable and stable manifolds of $E(V)$ can be defined as
\begin{equation}
\label{def_local_stab_unstab_man_as_images_non_zero_speed}
\begin{aligned}
\Wuloc{c}{V}\bigl(E(V)\bigr) &= \hatwuloc{c}{V}\bigl(\Bu_{E_0}(r)\bigr) \,, \\
\text{and}\quad
\Wsloc{c}{V}\bigl(E(V)\bigr) &= \hatwsloc{c}{V}\bigl(\Bs_{E_0}(r)\bigr)
\,.
\end{aligned}
\end{equation}
Those manifolds depend smoothly of $c$ and $V$. The global unstable and stable manifolds
\[
\begin{aligned}
\Wu_{c,V}\bigl(E(V)\bigr)&=\{U\in\rr^{2d}:S_{c,V}(\xi,U)\rightarrow E(V)\text{ when }\xi\rightarrow - \infty\} \\
\text{and}\quad \Ws_{c,V}\bigl(E(V)\bigr)&=\{U\in\rr^{2d}:S_{c,V}(\xi,U)\rightarrow E(V)\text{ when }\xi\rightarrow + \infty\}
\end{aligned}
\]
can then be derived from those local manifolds through the flow $S_{c,V}$ as follows: 
\[
\begin{aligned}
\Wu_{c,V}\bigl(E(V)\bigr)&=S_{c,V}\left(\rr\times \Wuloc{c}{V} \bigl(E(V)\bigr) \right)\,, \\
\text{and}\quad \Ws_{c,V}\bigl(E(V)\bigr)&=S_{c,V}\left(\rr\times \Wsloc{c}{V} \bigl(E(V)\bigr) \right)\,.
\end{aligned}
\]
\begin{remark}
Here are two observations that will turn out to play some role in the forthcoming proofs. 
\begin{itemize}
\item 
According to the characterization provided by this proposition (equivalence between \cref{item:U_is_in_Wuloc} and \cref{item:U_goes_to_E_at_minus_infty} and between \cref{item:U_is_in_Wsloc} and \cref{item:U_goes_to_E_at_plus_infty}), for every solution $\xi\mapsto U(\xi)$ of system \cref{trav_wave_system_order_1}, if this solution belongs to the stable (unstable) manifold of $E(V)$ then it crosses exactly once the border $\partial\Wsloc{c}{V}\bigl(E(V)\bigr)$ of the local stable manifold (the border $\partial\Wuloc{c}{V}\bigl(E(V)\bigr)$ of the local unstable manifold) of $E(V)$. In addition, according to the the conditions \cref{normu_norms_characterization} satisfied by the norms $\normu{\cdot}$ and $\norms{\cdot}$, up to replacing the radius $r$ by a smaller quantity, this intersection between the trajectory of $\xi\mapsto U(\xi)$ and the border of the local stable (unstable) manifold of $E(V)$ is \emph{transverse} inside the full stable (unstable) manifold. Although the transversality  of this intersection is not formally required in the following, assuming that it holds helps figuring out the broad scheme, see for instance \cref{fig:phi_map}.
\item The functions $\wuloc{c}{V}$ and $\wsloc{c}{V}$ are uniquely defined by the characterization provided by \cref{prop:loc_stab_unstab_man_c_positive} once the radius $r$ and the departure sets of these two functions are chosen. As a consequence, those two functions and the local stable and unstable manifolds $\Wuloc{c}{V}\bigl(E(V)\bigr)$ and $\Wsloc{c}{V}\bigl(E(V)\bigr)$ remain unchanged if the potential function $V$ is modified outside a neighbourhood of the set
\[
\piPos\Bigl[\Wuloc{c}{V}\bigl(E(V)\bigr)\cup\Wsloc{c}{V}\bigl(E(V)\bigr)\Bigr]
\,,
\]
where $\piPos$ is the projection onto the position coordinates:
\begin{equation}
\label{def_pi_pos}
\piPos:\rr^{2d}\to\rr^d, \quad (u,v)\mapsto u
\,.
\end{equation}
\end{itemize}
\end{remark}
\subsection{Local stable and unstable manifolds when the speed \texorpdfstring{$c$ equals $0$}{c equals 0}}
\label{subsec:local_stable_unstable_manifolds_zero_speed}
As \cref{table:dim_stable_unstable_centre} shows, an equilibrium $E$ is hyperbolic except if $c$ vanishes and $m(e)$ is positive. In this case, there exists, in addition to the stable and unstable manifolds of $E$, a centre manifold with dimension $2m(e)$ (corresponding to the central part of the spectrum of the linear system \cref{trav_wave_system_order_1_linearized} at $E$). However, as shown by the following lemma, a solution $\xi\mapsto U(\xi)$ of system \cref{trav_wave_system_order_1} cannot asymptotically approach $E$ through such a centre manifold. The statement of \cref{prop:loc_stab_unstab_man_c_equals_zero} below and the proof of \cref{prop:profiles_of_fronts_pulses_as_intersections_between_Wu_Ws}, provided in \cref{subsec:proof_prop_equivalence_fronts_pulses_connections}, rely on this lemma. 
\begin{lemma}[approach of critical points through stable/unstable manifolds]
\label{lem:approach_through_stable_unstable_manifold}
Assume that $c$ equals $0$. For every critical point $e$ of $V$ such that the Morse index $m(e)$ is positive, and for every (maximal) solution $\xi\mapsto U(\xi)$ of the differential system \cref{trav_wave_system_order_1}, if $E$ denotes the point $(e,0)$, the following conclusions hold: 
\begin{enumerate}
\item if $U(\xi)$ goes to $E$ as $\xi$ goes to $+\infty$, then the trajectory of $\xi\mapsto U(\xi)$ converges to $E$ tangentially to the stable space $\Es_V(E)$.
\label{item_lem:approach_through_stable_unstable_manifold_stable}
\item if $U(\xi)$ goes to $E$ as $\xi$ goes to $-\infty$, then the trajectory of $\xi\mapsto U(\xi)$ converges to $E$ tangentially to the unstable space $\Eu_V(E)$,
\label{item_lem:approach_through_stable_unstable_manifold_unstable}
\end{enumerate}
\end{lemma}
\begin{proof}
Let $\xi\mapsto U(\xi)=(u,v)(\xi)$ denote a solution of the differential system \cref{trav_wave_system_order_1} for a speed $c$ equal to $0$, and let us assume that $U(\xi)$ goes to $E$ as $\xi$ goes to $+\infty$. It follows from the invariance of the Hamiltonian function $H_V$ (defined in \cref{Hamiltonian}) along $U(\cdot)$ that $H_V(U(\xi))=H_V(E)$, or in other words that
\begin{equation}
\label{HV_of_U_of_xi_equals_HV_of_E}
\frac{1}{2}\abs{v(\xi)}^2 - V\bigl(u(\xi)\bigr) =-V(e) \,.
\end{equation}
Let us proceed by contradiction and assume that this solution does \emph{not} belong to the stable manifold of $E$. With the notation of \cref{subsec:eigenspaces_and_dimensions}, it follows that, as $\xi$ goes to $+\infty$, the component of $U(\xi)-E$ along the centre subspace $\Ec_{V}(E)$ is dominant compared to its component along the hyperbolic subspace $\Es_{V}(E)+\Eu_{V}(E)$; with symbols, if $\piCent$ denotes the projection along $\Es_{V}(E)+\Eu_{V}(E)$ onto $\Ec_{V}(E)$ in $\rr^{2d}$,
\begin{equation}
\label{centre_component_dominates}
U(\xi)-E = \piCent\bigl(U(\xi)-E\bigr) + o_{\xi\to+\infty}\Bigl(\piCent\bigl(U(\xi)-E\bigr)\Bigr)
\,.
\end{equation}
It follows from the expressions \cref{notation_eigenvalues_Rd_times_Rd,notation_eigenvectors_in_dim_2d} of the eigenvalues and eigenvectors of $DF_{0,V}(E)$ that
\[
\Ec_{V}(E) = \spanset\bigl\{U_{1,+},U_{1,-},\dots,U_{m(e),+},U_{m(e),-}\bigr\}
\,.
\]
As a consequence, applying the projection $\piPos$ (projection onto the position coordinates, defined in \cref{def_pi_pos}) to equality \cref{centre_component_dominates}, it follows that, if we denote by $\pi_{m(e)}$ the orthogonal projection onto $\spanset\{u_1,\dots,u_{m(e)}\}$ in $\rr^d$, 
\[
u(\xi)-e = \pi_{m(e)}\bigl(u(\xi)-e\bigr) + o_{\xi\to+\infty}\Bigl(\pi_{m(e)}\bigl(u(\xi)-e\bigr)\Bigr)
\,.
\]
Since the restriction of $D^2V(e)$ to the image of $\pi_{m(e)}$ is negative definite, it follows that, for $\xi$ sufficiently large, $V\bigl(u(\xi)\bigr)$ is smaller than $V(e)$, thus $-V\bigl(u(\xi)\bigr)$ is larger than $-V(e)$, contradicting equality \cref{HV_of_U_of_xi_equals_HV_of_E}. \Cref{lem:approach_through_stable_unstable_manifold} is proved. 
\end{proof}
As for \cref{prop:loc_stab_unstab_man_c_positive} in the case $c>0$, the aim of the next \cref{prop:loc_stab_unstab_man_c_equals_zero} is to provide (in the case $c=0$) a precise statement and the associated notation concerning the local stable and unstable manifolds of an equilibrium for the differential system \cref{trav_wave_system_order_2}. In this case $c=0$, the conclusions of \cref{lem:approach_through_stable_unstable_manifold} show that centre manifolds are not relevant for homoclinic and heteroclinic solutions; for that reason, those centre manifolds are ignored in \cref{prop:loc_stab_unstab_man_c_equals_zero}. Concerning the construction and properties of the local stable and unstable manifolds, there is no difference with respect to the positive speed case considered in \cref{prop:loc_stab_unstab_man_c_positive}, see again \cite{Kelley_stableCenterUnstableMan_1967,GuckenHeimerHolmes_nonlinOscDynSystBif_1983,Teschl_odeDynSyst_2012}. Observe that, by contrast with the statements that can be found in textbooks, the characterization of these local stable and unstable manifolds does not require an exponential rate of convergence towards $E$, again due to the conclusions of \cref{lem:approach_through_stable_unstable_manifold} (see the equivalence between assertions \cref{item:U_is_in_WulocZero} and \cref{item:U_goes_to_E_at_minus_infty_Zero} and between assertions \cref{item:U_is_in_WslocZero} and \cref{item:U_goes_to_E_at_plus_infty_Zero} in \cref{prop:loc_stab_unstab_man_c_equals_zero} below).
\begin{notation}
For the remaining of this paper, when the speed $c$ vanishes, it will be omitted in the notation. Thus, concerning the previously introduced notation, 
\[
\begin{aligned}
&F_{V} 
&\quad\quad
&S_{V}
&\quad\quad
&\Es_{V}
&\quad\quad
&\Ec_{V}
&\quad\quad
&\Eu_{V}
&\quad\quad
&\Ws_{V}
&\quad\quad
&\Wu_{V} \\
\text{stand for:}\quad
&F_{0,V} 
&\quad\quad
&S_{0,V}
&\quad\quad
&\Es_{0,V}
&\quad\quad
&\Ec_{0,V}
&\quad\quad
&\Eu_{0,V}
&\quad\quad
&\Ws_{0,V}
&\quad\quad
&\Wu_{0,V}
\,.
\end{aligned}
\]
\end{notation}
Take $V_0$ in $\CkFull$ and let $e_0$ denote a non-degenerate critical point of $V_0$ and let $E_0=(e_0,0)$ (which  is not necessarily hyperbolic). Let $\betau$ and $\betas$ be as in \cref{betas}. As in the case $c>0$, there exist norms $\normu{\cdot}$ on the unstable subspace $\Eu_{V_0}(E_0)$ and $\norms{\cdot}$ on the stable subspace $\Es_{V_0}(E_0)$ such that inequalities \cref{normu_norms_characterization} hold for every non negative quantity $\xi$. Let $\normc{\cdot}$ denote any norm on the centre subspace $\Ec_{V_0}(E_0)$ (for instance the euclidean norm). For every positive quantity $r$, let 
\[
\begin{aligned}
\Bu_{E_0}(r) &= \{\Uu \in \Eu_{V_0}(E_0) : \normu{\Uu}\le r\}\,, \\
\Bs_{E_0}(r) &= \{\Us \in \Es_{V_0}(E_0) : \norms{\Us}\le r\}\,, \\
\Bc_{E_0}(r) &= \{\Uc \in \Ec_{V_0}(E_0) : \normc{\Uc}\le r\}\,, \\
\text{and}\quad
\widebar{B}_{E_0}(r) &= \{\Uu+\Us+\Uc : \Uu\in\Bu_{E_0}(r) \text{ and } \Us\in \Bs_{E_0}(r) \text{ and } \Uc\in \Bc_{E_0}(r)\}
\,.
\end{aligned}
\]
\begin{proposition}[local stable and unstable manifolds]
\label{prop:loc_stab_unstab_man_c_equals_zero}
There exist a neighbourhood $\nu$ of $V_0$ in $\CkFull$
and a positive quantity $r$ such that, for every $V$ in $\nu$, the following statements hold.
\begin{enumerate}
\item There exists a unique critical point $e(V)$ of $V$ such that $E(V)=(e(V),0)$ belongs to $E_0+\widebar{B}_{E_0}(r)$. In addition, $e(V)$ has the same Morse index as $e_0$ and the map $\nu\to\rr^d$, $V\mapsto e(V)$ is of class $\ccc^k$.
\item There exist $C^k$-functions
\[
\wulocZero{V}:\Bu_{E_0}(r)\to \Bs_{E_0}(r)+\Bc_{E_0}(r)
\quad\text{and}\quad
\wslocZero{V}:\Bs_{E_0}(r)\to \Bu_{E_0}(r)+\Bc_{E_0}(r)
\]
such that, if we consider the sets
\[
\begin{aligned}
\WulocZero{V}\bigl(E(V)\bigr) &= \left\{E(V) + \Uu + \wulocZero{V}(\Uu) : \Uu\in \Bu_{E_0}(r)\right\}  \\ 
\text{and}\quad
\WslocZero{V}\bigl(E(V)\bigr) &= \left\{E(V) + \Us + \wslocZero{V}(\Us) : \Us\in \Bs_{E_0}(r)\right\}
\,,
\end{aligned}
\]
then, for every $U$ in $\widebar{B}_{E_0}(r)$, the following two assertions are equivalent:
\begin{enumerate}
\item $U$ is in $\WulocZero{V}\bigl(E(V)\bigr)$;
\label{item:U_is_in_WulocZero}
\item $S_{V}(\xi,U)-E(V)$ remains in $\widebar{B}_{E_0}(r)$ for all $\xi$ in $(-\infty,0]$ and $S_{V}(\xi,U)\to E(V)$ as $\xi\to -\infty$,
\label{item:U_goes_to_E_at_minus_infty_Zero}
\end{enumerate}
and for every $U$ in $\widebar{B}_{E_0}(r)$, the following two assertions are equivalent:
\begin{enumerate}
\setcounter{enumii}{2}
\item $U\in\WslocZero{V}\bigl(E(V)\bigr)$;
\label{item:U_is_in_WslocZero}
\item $S_{V}(\xi,U)-E(V)$ remains in $\widebar{B}_{E_0}(r)$ for all $\xi$ in $[0,+\infty)$ and $S_{V}(\xi,U)\to E(V)$ as $\xi\to +\infty$.
\label{item:U_goes_to_E_at_plus_infty_Zero}
\end{enumerate}
\item Both differentials $D\wulocZero{V_0}(0)$ and $D\wslocZero{V_0}(0)$ vanish, and both maps 
\[
\begin{aligned}
\nu\times \Bu_{E_0}(r)\to \Bs_{E_0}(r), &\quad (V,\Uu)\mapsto \WulocZero{V}(\Uu) \\
\text{and}\quad
\nu\times \Bs_{E_0}(r)\to \Bu_{E_0}(r), &\quad (V,\Us)\mapsto \WslocZero{V}(\Us)
\end{aligned}
\]
are of class $\ccc^k$. 
\end{enumerate}
\end{proposition}
With the notation provided by \cref{prop:loc_stab_unstab_man_c_equals_zero}, for every $V$ in $\nu$, let us introduce the maps
\[
\begin{aligned}
\hatwulocZero{V}:\Bu_{E_0}(r)\to\rr^{2d}, &\quad \Uu\mapsto E(V) + \Uu + \wulocZero{V}(\Uu) \,,\\
\text{and}\quad
\hatwslocZero{V}:\Bs_{E_0}(r)\to\rr^{2d}, &\quad \Us\mapsto E(V) + \Us + \wslocZero{V}(\Us)
\,.
\end{aligned}
\]
Local unstable and stable manifolds of $E(V)$ can be defined as
\begin{equation}
\label{def_local_stab_unstab_man_as_images_zero_speed}
\begin{aligned}
\WulocZero{V}\bigl(E(V)\bigr) &= \hatwulocZero{V}\bigl(\Bu_{E_0}(r)\bigr) \,, \\
\text{and}\quad
\WslocZero{V}\bigl(E(V)\bigr) &= \hatwslocZero{V}\bigl(\Bs_{E_0}(r)\bigr)
\,.
\end{aligned}
\end{equation}
Those manifolds depend smoothly of $V$. As in \cref{subsec:local_stable_unstable_manifolds_positive_speed}, the global unstable/stable manifolds of $E(V)$, denoted by $\Wu_{V}(E(V))$ and $\Ws_{V}(E(V))$ can be expressed in terms of those local manifolds and of the flow $S_V$. Both observations made in the remark ending the previous \namecref{subsec:local_stable_unstable_manifolds_positive_speed} are still valid in the present case of zero speed and potential existence of a centre manifold. 
\section{Preliminary properties of travelling fronts and standing fronts and pulses}
\label{sec:fronts_pulses}
Let us take and fix, for this whole \namecref{sec:fronts_pulses}, a potential function $V$ in $\CkFull$. 
\subsection{Proof of \texorpdfstring{\cref{prop:profiles_of_fronts_pulses_as_intersections_between_Wu_Ws}}{Proposition \ref{prop:profiles_of_fronts_pulses_as_intersections_between_Wu_Ws}}}
\label{subsec:proof_prop_equivalence_fronts_pulses_connections}
Let $e_-$ and $e_+$ be two (possibly equal) non-degenerate critical points of $V$, let $c$ denote a non negative quantity (speed), and let $\xi\mapsto u(\xi)$ denote the profile of a front or pulse connecting $e_-$ to $e_+$ and travelling at speed $c$ (or standing if $c$ equals zero) for the potential $V$. 
\begin{lemma}
\label{lem:derivative_goes_to_zero_at_both_ends}
The derivative $\dot u(\xi)$ goes to $0$ as $\xi$ goes to $\pm\infty$. 
\end{lemma}
\begin{proof}
If the speed $c$ is positive, then $\xi\mapsto u(\xi)$ is the profile of a travelling front. It follows from integrating \cref{time_derivative_Hamiltonian} that
\begin{equation}
\label{potential_lag}
\lim_{\xi\rightarrow +\infty} H_V(u(\xi)) - \lim_{\xi\rightarrow -\infty} H_V(u(\xi)) = - c\int_\rr |\dot u(\xi)|^2 \, d\xi
\end{equation}
and thus that $\dot u(\cdot)$ is in $L^2(\rr)$. Thus $0$ is an adherent value of the kinetic part of the Hamiltonian function $\xi\mapsto H_V\bigl(U(\xi)\bigr)$ as $\xi$ goes to $\pm\infty$, meaning that $V(e_\pm)$ are adherent values of $H_V(U(\xi))$. Since according to \cref{time_derivative_Hamiltonian} this last function decreases with $\xi$, it follows that $H_V\bigl(U(\xi)\bigr)$ goes to $V(e_\pm)$ as $\xi$ goes to $\pm\infty$, and the intended conclusion follows. 

If the speed $c$ equals $0$, it follows from the differential system \cref{trav_wave_system_order_2} and the convergence of $u(\cdot)$ to critical points that $\ddot u(\xi)$ goes to $0$ as $\xi$ goes to $\pm\infty$. Thus $\dot u(\cdot)$ is uniformly continuous and the convergence of $u$ yields the intended conclusion. 
\end{proof}

\begin{proof}[Proof of \cref{prop:profiles_of_fronts_pulses_as_intersections_between_Wu_Ws}]
Let us use the notation of \cref{prop:profiles_of_fronts_pulses_as_intersections_between_Wu_Ws}. 
If $c$ is non zero or if $c$ equals $0$ and both Morse indices $m(e_-)$ and $m(e_+)$ of $e_-$ and $e_+$ vanish, then $E_-$ and $E_+$ are hyperbolic equilibria of the differential system \cref{trav_wave_system_order_1} and the conclusions of \cref{prop:profiles_of_fronts_pulses_as_intersections_between_Wu_Ws} follow from \cref{lem:derivative_goes_to_zero_at_both_ends}. 

If $c$ equals $0$ and the Morse indices $m(e_-)$ and $m(e_+)$ are any, then the equilibria $E_-$ and $E_+$ are not necessarily hyperbolic, but again in this case it follows from \cref{lem:derivative_goes_to_zero_at_both_ends} that $U(\xi)$ goes to $E_\pm$ as $\xi\to\pm\infty$; and it follows from \cref{lem:approach_through_stable_unstable_manifold} that the values of $\xi\mapsto U(\xi)$ belong to the unstable manifold of $E_-$ and to the stable manifold of $E_+$.
\end{proof}
\subsection{Equivalent definitions of a symmetric standing pulse}
Let $e$ denote a non-degenerate critical point of $V$, and let $\xi\mapsto u(\xi)$ denote the profile of a standing pulse connecting $e$ to itself. In \cref{def:symmetric_pulse}, the symmetry of such a pulse was defined by the existence of a ``turning time'' where $\dot u$ vanishes. The following standard result (see for instance \cite{Devaney_reversibleDiffeoFlows_1976}) completes this definition. 
\begin{lemma}[equivalent definitions of a symmetric standing pulse]
\label{lem:equivalent_def_symm_hom_orbit}
For every real quantity $\xiTurn$, the following properties are equivalent: 
\begin{enumerate}
\item $\xiTurn$ is a turning time in the sense of \cref{def:symmetric_pulse}, that is $\dot u(\xiTurn)=0$;
\label{item_lem:equivalent_def_symm_hom_orbit_vanishing_speed_turning_time}
\item for every $\xi$ in $\rr$, 
\label{item_lem:equivalent_def_symm_hom_orbit_symmetry_all}
\begin{equation}
\label{reversibility_symmetry_position}
u(\xiTurn-\xi) = u(\xiTurn+ \xi)
\,;
\end{equation}
\item there exists $\xi$ in $\rr$ such that
\begin{equation}
\label{reversibility_symmetry_position_speed}
u(\xiTurn-\xi) = u(\xiTurn+ \xi)
\quad\text{and}\quad
\dot u(\xiTurn-\xi) = -\dot u(\xiTurn + \xi)
\,.
\end{equation}
\label{item_lem:equivalent_def_symm_hom_orbit_symmetry_once}
\end{enumerate}
In addition, these statements hold for at most one real quantity $\xiTurn$. 
\end{lemma}
\begin{proof}
Differentiating equality \cref{reversibility_symmetry_position} with respect to $\xi$ yields equalities \cref{reversibility_symmetry_position_speed} for all $\xi$, so that property \cref{item_lem:equivalent_def_symm_hom_orbit_symmetry_all} implies property \cref{item_lem:equivalent_def_symm_hom_orbit_symmetry_once}, and property \cref{item_lem:equivalent_def_symm_hom_orbit_symmetry_once} for $\xi$ equal to $0$ is equivalent to property \cref{item_lem:equivalent_def_symm_hom_orbit_vanishing_speed_turning_time}, so that property \cref{item_lem:equivalent_def_symm_hom_orbit_symmetry_all} implies property \cref{item_lem:equivalent_def_symm_hom_orbit_vanishing_speed_turning_time} and property \cref{item_lem:equivalent_def_symm_hom_orbit_vanishing_speed_turning_time} implies property \cref{item_lem:equivalent_def_symm_hom_orbit_symmetry_once}. 

It remains to prove that property \cref{item_lem:equivalent_def_symm_hom_orbit_symmetry_once} implies property \cref{item_lem:equivalent_def_symm_hom_orbit_symmetry_all}. Assume that property \cref{item_lem:equivalent_def_symm_hom_orbit_symmetry_once} holds, and, for every real quantity $\xi$, let us write
\[
u_1(\xi) = u(\xiTurn + \xi)
\quad\text{and}\quad
u_2(\xi)= u(\xiTurn - \xi)
\,.
\]
Since $\xi\mapsto u(\xi)$ is a solution of the second order differential system \cref{trav_wave_system_order_2} with $c$ equal to zero, both $\xi\mapsto U_1(\xi)$ and $\xi\mapsto U_2(\xi)$ are solutions of the first order differential system \cref{trav_wave_system_order_1} (again with $c$ equal to zero). According to property \cref{item_lem:equivalent_def_symm_hom_orbit_symmetry_once}, there exists $\xi$ such that $U_1(\xi)$ is equal to $U_2(\xi)$. Thus $U_1(\xi)$ must be equal to $U_2(\xi)$ for every real time $\xi$, and property \cref{item_lem:equivalent_def_symm_hom_orbit_symmetry_all} follows.
Thus the three properties of \cref{lem:equivalent_def_symm_hom_orbit} are equivalent. 

In addition, if property \cref{item_lem:equivalent_def_symm_hom_orbit_symmetry_all} holds for two different turning times $\xiTurn$ and $\xiTurn'$, then $\xi\mapsto u(\xi)$ is periodic with a period equal to $2(\xiTurn'-\xiTurn)$, a contradiction with the assumption that $u$ is a standing pulse connecting $e$ to itself. \Cref{lem:equivalent_def_symm_hom_orbit} is proved. 
\end{proof}
\subsection{Values reached only once by profiles of travelling fronts / standing pulses}
\label{subsub:values_reached_only_once}
The proofs carried on in the  \cref{sec:generic_transversality_travelling_fronts,sec:generic_elementarity_sym_stand_pulses,sec:generic_tranversality_asym_stand_pulses} below rely on the construction of suitable perturbations of the potential $V$. Whereas the uniqueness of the solutions of differential system \cref{trav_wave_system_order_1} ensures that the function $\xi\mapsto \bigl(u(\xi),\dot u(\xi)\bigr)$ defined by such a solution is one-to-one, this is not necessarily true for the function $\xi\mapsto u(\xi)$ (as shown by \cref{fig:one-to-one}). As a consequence, a perturbation of the potential $V$ may affect this solution at different times. The goal of the following proposition is to avoid this inconvenience, by providing in each case under consideration a time interval where $u(\xi)$ is reached only once.
\begin{proposition}
\label{prop:values_reached_only_once}
\hfill
\begin{enumerate}
\item \label{item_prop:values_reached_only_once_travelling_front}
For every profile $\xi\mapsto u(\xi)$ of a front travelling at a positive speed $c$ and connecting two non-degenerate critical points, there exists a time $\xiOnce$ such that, for all times $\xi^*$ in $(-\infty,\xiOnce]$ and $\xi$ in $\rr$, 
\begin{equation}\label{only_once}
u(\xi) = u(\xi^*) \implies \xi = \xi^*\,.
\end{equation}
\item For every profile $\xi\mapsto u(\xi)$ of an asymmetric standing pulse and for every nonempty open interval $I$ of $\rr$, there exists a nonempty open interval $\IOnce$, included in $I$, such that, for all times $\xi^*$ in $\IOnce$ and $\xi$ in $\rr$, implication \cref{only_once} holds. 
\label{item_prop:values_reached_only_once_asymmetric_standing_pulse}
\item For every profile $\xi\mapsto u(\xi)$ of a symmetric standing pulse, if $\xiTurn$ denotes the turning time of this pulse (see \cref{lem:equivalent_def_symm_hom_orbit}), then, for every nonempty open interval $I$ included in $(-\infty,\xiTurn]$, there exists a nonempty open interval $\IOnce$, included in $I$, such that, for all times $\xi^*$ in $\IOnce$ and $\xi$ in $(-\infty,\xiTurn]$, implication \cref{only_once} holds.
\label{item_prop:values_reached_only_once_symmetric_standing_pulse}
\end{enumerate}
\end{proposition}
\begin{figure}[htbp]
\centering
\resizebox{0.4\textwidth}{!}{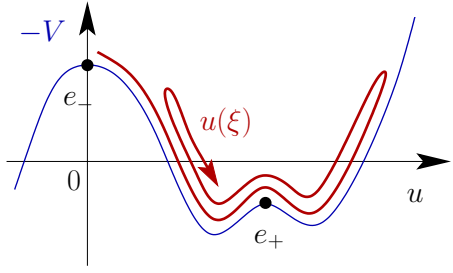}
\caption{The one-dimensional example of this figure shows that property \cref{item_prop:values_reached_only_once_travelling_front} of \cref{prop:values_reached_only_once} may not hold outside a small neighbourhood of the critical point $e_-$.}
\label{fig:one-to-one}
\end{figure}
\begin{proof}[Proof of statement \cref{item_prop:values_reached_only_once_travelling_front} of \cref{prop:values_reached_only_once}]
Let $\xi\mapsto u(\xi)$ denote the profile of a front travelling at a positive speed $c$ for the potential $V$, and let $e_-$ denote the critical point, assumed to be non-degenerate, approached by $u(\xi)$ as $\xi$ goes to $-\infty$. Since all eigenvalues of $DF_{c,V}(E_-)_{|\Eu_{c,V}(E_-)}$ are real and positive (see \cref{subsec:eigenspaces_and_dimensions}), the corresponding solution $U(\xi)$ of system \cref{trav_wave_system_order_1} must approach $E_-$ tangentially to some (real, unstable) eigenvector $\Ueig$ of $DF_{c,V}(E_-)$ as $\xi$ goes to $-\infty$. If $\lambda$ denotes the corresponding (positive) eigenvalue, then $\Ueig$ is of the form $(\ueig,\lambda\ueig)$, where $\ueig$ is an eigenvector of $D^2V(e_-)$, see expression \cref{notation_eigenvectors_in_dim_2d}. Thus there must exist a nonzero scalar function $\xi\mapsto\alpha(\xi)$ so that, as $\xi$ goes to $-\infty$, 
\[
U(\xi)=E_- + \alpha(\xi) \Ueig + o\bigl(\alpha(\xi)\bigr)
\,,
\quad\text{that is}\quad \left\{
\begin{aligned}
u(\xi) &= e_- + \alpha(\xi)\ueig + o\bigl(\alpha(\xi)\bigr)\,,\\
\dot u(\xi) &= \alpha(\xi)\lambda\ueig + o\bigl(\alpha(\xi)\bigr)\,.
\end{aligned}
\right.
\]
It follows that there exists a large negative time $\xi_0$ such that, for every $\xi$ in $(-\infty,\xi_0]$,
\[
\frac{d}{{d}\xi} |u(\xi)-e_-|^2\,=\,2(u(\xi)-e_-)\cdot \dot u(\xi)\,>\,0
\,.
\]
In particular, the function
\begin{equation}
\label{distance_to_e_minus_for_large_negative_xi}
(-\infty,\xi_0]\to\rr^d, \quad \xi\mapsto u(\xi)
\end{equation}
is a $\ccc^1$-diffeomorphism onto its image. According to the decrease \cref{time_derivative_Hamiltonian} of the Hamiltonian, the quantity $H_V\bigl(U(\xi_0)\bigr)$ is smaller than $-V(e_-)$. As a consequence, there exists a time $\xiOnce$ in $(-\infty,\xi_0)$ such that, for every $\xi^*$ in $(-\infty,\xiOnce]$, 
\begin{equation}
\label{minus_V_larger_than_H_at_xiExit_for_large_negative_xi}
H_V\bigl(U(\xi_0)\bigr)<-V\bigl(u(\xi^*)\bigr)
\,.
\end{equation}
Take a time $\xi^*$ in $(-\infty,\xiOnce]$ and a time $\xi$ in $\rr$ and let us assume that $u(\xi)$ equals $u(\xi^*)$. If $\xi$ was larger than $\xi_0$ then it would follow from the expression \cref{Hamiltonian} of the Hamiltonian, its decrease \cref{time_derivative_Hamiltonian} and inequality \cref{minus_V_larger_than_H_at_xiExit_for_large_negative_xi} that
\[
-V\bigl(u(\xi)\bigr) \le H_V\bigl(U(\xi)\bigr) \le H_V\bigl(U(\xi_0)\bigr) < -V\bigl(u(\xi^*)\bigr)
\,,
\]
a contradiction with the equality of $u(\xi)$ and $u(\xi^*)$. Thus $\xi$ is not larger than $\xi_0$, and it follows from the one-to-one property of the function \cref{distance_to_e_minus_for_large_negative_xi} that $\xi$ must be equal to $\xi^*$. Statement \cref{item_prop:values_reached_only_once_travelling_front} of \cref{prop:values_reached_only_once} is proved. 
\end{proof}
\begin{proof}[Proof of statement \cref{item_prop:values_reached_only_once_asymmetric_standing_pulse} of \cref{prop:values_reached_only_once}]
Let $\xi\mapsto u(\xi)$ be the profile of an asymmetric standing pulse for the potential $V$, let $e$ denote the critical point approached by $u(\xi)$ as $\xi$ goes to $\pm\infty$, and let $I$ be a nonempty open interval of $\rr$. In view of the intended conclusion (statement \cref{item_prop:values_reached_only_once_asymmetric_standing_pulse}), we may assume that $I$ is bounded. According to the invariance \cref{time_derivative_Hamiltonian} of the Hamiltonian $H_V$, for every $\xi$ in $\rr$, the difference $V\bigl(u(\xi)\bigr)-V(e)$ is equal to $|\dot u(\xi)|^2/2$ and is therefore nonzero, so that the critical point $e$ is never reached by the function $\xi\mapsto u(\xi)$ on $\rr$. As a consequence there exists a (small) positive quantity $r$ such that $\abs{u(\xi)-e}$ is larger than $r$ for all $\xi$ in $I$; and since $u(\xi)$ approaches $e$ as $\xi$ goes to $\pm\infty$, there exists a (large) positive quantity $M$ such that $\abs{u(\xi)-e}$ is smaller than $r$ outside of $[-M,M]$. 

Assume that there exist two different times $\xi$ and $\xi'$ in $\rr$ such that $u(\xi)$ equals $u(\xi')$. Then, again according to the invariance \cref{time_derivative_Hamiltonian} of the Hamiltonian $H_V$, the time derivatives $\dot u(\xi)$ and $\dot u(\xi')$ must have the same norm. Besides, these two vectors cannot be equal (or else the profile $u$ would be periodic) nor opposite (or else according to \cref{lem:equivalent_def_symm_hom_orbit} the pulse would be symmetric), thus they are not proportional. Thus the couples $(\xi,\xi')$ such that $u(\xi)$ is equal to $u(\xi')$ are isolated in $\rr^2$. In addition, if $(\xi,\xi')$ is such a couple and $\xi$ is in $I$ then $\xi'$ must belong to $[-M,M]$. This shows by compactness that there exists only a finite number of couples $(\xi,\xi')$ in $I\times\rr$ such that $u(\xi)$ equals $u(\xi')$. Statement \cref{item_prop:values_reached_only_once_asymmetric_standing_pulse} of \cref{prop:values_reached_only_once} follows.
\end{proof}
\begin{proof}[Proof of statement \cref{item_prop:values_reached_only_once_symmetric_standing_pulse} of \cref{prop:values_reached_only_once}]
The arguments are the same as in the proof of statement \cref{item_prop:values_reached_only_once_asymmetric_standing_pulse} above. Let $\xi\mapsto u(\xi)$ be the profile of a symmetric pulse with turning time $\xiTurn$ for the potential $V$, let $I$ be a nonempty open interval of $(-\infty,\xiTurn]$, assumed to be bounded. If there exist two different times $\xi$ and $\xi'$ in $(-\infty,\xiTurn]$ such that $u(\xi)$ equals $u(\xi')$, again the time derivatives $\dot u(\xi)$ and $\dot u(\xi')$ have the same norm. These two vectors cannot be equal (or else the profile $u$ would be periodic) nor opposite (or else, according to statement \cref{item_lem:equivalent_def_symm_hom_orbit_symmetry_once} of \cref{lem:equivalent_def_symm_hom_orbit}, $(\xi+\xi')/2$ would be a second turning time --- smaller than $\xiTurn$ --- for $u$, a contradiction with the conclusion of \cref{lem:equivalent_def_symm_hom_orbit}). Thus again, $\dot u(\xi)$ and $\dot u(\xi')$ cannot be proportional, and the same arguments as in the proof of statement \cref{item_prop:values_reached_only_once_asymmetric_standing_pulse} above show that there exists only a finite number of couples $(\xi,\xi')$ in $I\times(-\infty,\xiTurn]$ such that $u(\xi)=u(\xi')$. 
\end{proof}
\section{Tools for genericity}
\subsection{An instance of the Sard--Smale transversality theorem}
\label{subsec:transversality_thm}
To prove that a given property generically holds, a standard method is to express this property as a transversality problem and to use one instance among the family of theorems known as Sard--Smale theorem (or Thom's theorems, or transversality theorems), see \cite{AbrahamRobbin_transversalMappingsFlows_1967,BroerHasselblattTakens_handbookDynSyst3_2010,Hirsch_differentialTopology_1976,Peixoto_onApprroximationTheoremKupkaSmale_1967}. In this paper the following instance will be used (\cref{thm:Sard_Smale} below).
Let us consider a function
\[
\Phi:\mmm\times\Lambda\to\nnn
\,,
\]
where $\mmm$ and $\nnn$ are two finite-dimensional manifolds and $\Lambda$ (``parameter space'') is a Banach manifold, together with a submanifold $\www$ of $\nnn$ (see \cref{fig:sard_smale}). Let us assume that these four manifolds and the function $\Phi$ are of class $C^k$ (as everywhere in the paper $k$ denotes an integer which is not smaller than $1$). Finally, let $\codim(\www)$ denote the codimension of $\www$ in $\nnn$. 

\begin{definition}
With the notation above, the image of $\Phi$ is said to be \emph{transverse to $\www$}, if, for every $(m,\lambda)$ in $\mmm\times\Lambda$ such that $\Phi(m,\lambda)$ is in $\www$, the following equality holds:
\[
D\Phi(T_{m}\mmm\times T_{\lambda}\Lambda)+T_{\Phi(m,\lambda)}\www = T_{\Phi(m,\lambda)}\nnn
\]
(here $D\Phi$ denotes the differential of $\Phi$ at $(m,\lambda)$).
Accordingly, for every $\lambda$ in $\Lambda$, if $\Phi_\lambda$ denotes the function:
\[
\mmm\to\nnn, \quad m\mapsto\Phi(m,\lambda)
\,,
\]
then the image of $\Phi_\lambda$ is said to be \emph{transverse to $\www$} if, for every $m$ in $\mmm$ such that $\Phi(m,\lambda)$ is in $\www$, denoting $D\Phi_\lambda$ the differential of $\Phi_\lambda$ at $m$,
\[
D\Phi_\lambda(T_m\mmm)+T_{\Phi(m,\lambda)}\www = T_{\Phi(m,\lambda)}\nnn
\,.
\]
\end{definition}
\begin{theorem}[Sard--Smale transversality theorem]
\label{thm:Sard_Smale}
With the notation above, if
\begin{enumerate}
\item $k>\dim(\mmm)-\codim(\www)$,
\label{item:thm_sard_smale_condition_regularity}
\item and the image of $\Phi$ is transverse to $\www$,
\label{item:thm_sard_smale_condition_transversality}
\end{enumerate}
then there exists a generic subset $\LambdaGen$ of $\Lambda$ such that, for every $\lambda$ in $\LambdaGen$, the image of $\Phi_\lambda$ is transverse to $\www$. 
\end{theorem}
The proof of this result can be found in \cite{AbrahamRobbin_transversalMappingsFlows_1967} or in \cite{Hirsch_differentialTopology_1976}. The key hypothesis, which is often difficult to check, is the transversality hypothesis \cref{item:thm_sard_smale_condition_transversality}. Notice that the conclusion is stronger than this hypothesis since it states that the transversality holds for a fixed generic parameter $\lambda$, whereas hypothesis \cref{item:thm_sard_smale_condition_transversality} uses the freedom of moving $\lambda$.
\begin{figure}[htbp]
\centering
\resizebox{.9\textwidth}{!}{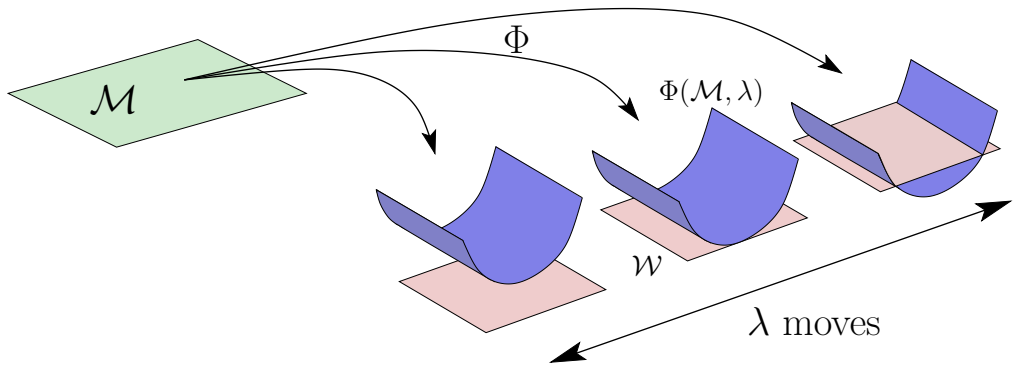}
\caption{Geometric interpretation of \cref{thm:Sard_Smale}. Assume that for a given parameter $\lambda_0$, $D\Phi_{\lambda_0}(T\mmm)+T\www$ is not the whole tangent space $T\nnn$, but that the dependence of $\Phi$ on $\lambda$ provides the missing directions. Then for almost every $\lambda$ close to $\lambda_0$, the image $\Phi(\mmm,\lambda)$ intersects $\www$ transversally.}
\label{fig:sard_smale}
\end{figure}
\subsection{Extending local genericity to global genericity}
\Cref{thm:Sard_Smale} (under the form above or another) is the standard tool to prove that a property generically holds. However, it turns out that is is often difficult, in practice, to express a property using a single function $\Phi$ as above; thus one is often led to patch together several conclusions provided by this theorem. The following lemma provides a way to carry out this patching process. This lemma is identical to Lemma 3.3 of Chapter 3 of \cite{PalisDeMelo_geometricTheoryDynamicalSystems_1982}, where a proof can be found. 
\begin{lemma}[local genericity implies global genericity in a separable Baire space]
\label{lem:local_genericity_implies_global_genericity}
Let $\vvv$ be a separable Baire space and $\vvvDense$ be a dense subset of $\vvv$. For every subset $\vvvGen$ of $\vvv$, the following two assertions are equivalent:
\begin{enumerate}
\item the subset $\vvvGen$ is generic in $\vvv$;
\label{item:lem_local_genericity_implies_global_genericity_global_genericity}
\item for every $V_0$ in $\vvvDense$, there exists a neighbourhood $\nu$ of $V_0$ in $\vvv$ such that $\vvvGen\cap\nu$ is generic in $\nu$. 
\label{item:lem_local_genericity_implies_global_genericity_local_genericity} 
\end{enumerate}
\end{lemma}
\subsection{Potentials that are quadratic past a given radius}
\label{subsec:potentials_quadratic_past_radius_R}
The whole space $\CkFull$ of potentials is somewhat difficult to handle, for various reasons: it is not separable, even locally, and the flow of system \cref{trav_wave_system_order_1} is not globally well-defined for some of the potentials $V$ in this space. To get around these difficulties, the proofs of the next \cref{sec:generic_transversality_travelling_fronts,sec:generic_elementarity_sym_stand_pulses} will be carried out on a more restricted class $\vvvQuad{R}$ of potentials, after what the results will be extended to the full set $\CkFull$ in the final \cref{sec:proof_main}. Let 
\begin{equation}
\label{notation_vvvFull}
\vvvFull = \CkFull 
\,,
\end{equation}
and, for a positive quantity $R$, let
\begin{equation}
\label{notation_vvvQuad_of_R}
\vvvQuad{R} = \left\{ V\in\vvvFull : \text{for all $u$ in $\rr^d$}, \ \abs{u}\ge R \implies V(u) = \frac{|u|^2}{2} \right\} 
\,.
\end{equation}
By contrast with $\vvvFull$, the affine subspace $\vvvQuad{R}$ of $\vvvFull$ is separable, and therefore provides a framework where \cref{lem:local_genericity_implies_global_genericity} may be applied. 
The next lemma states some (nice) properties of the flow of system \cref{trav_wave_system_order_1} for a potential $V$ in $\vvvQuad{R}$. It is followed by another one (\cref{cor:truncation_extension_of_potentials} below) providing the adequate tools to proceed with the extension mentioned above and carried out in \cref{sec:proof_main}. 
\begin{notation}
For every non negative quantity $r$, let $B_{\rr^d}(0,r)$ and $\widebar{B}_{\rr^d}(0,r)$ denote the open ball and the closed ball centred at the origin and of radius $r$ in $\rr^d$. 
\end{notation}
\begin{lemma}
\label{lem:global_flow_and_a_priori_bound_on_profiles}
For every positive quantity $R$ and for every potential $V$ in $\vvvQuad{R}$, the following conclusions hold. 
\begin{enumerate}
\item For every speed $c$, the flow defined by the differential system \cref{trav_wave_system_order_1} is global.
\label{item:lem_global_flow}
\item Every profile $\xi\mapsto u(\xi)$ of a travelling front or a standing front or a standing pulse, for this potential, satisfies the following bound: 
\begin{equation}
\label{a_priori_bound_on_profiles}
\sup_{\xi\in\rr} \abs{u(\xi)} < R
\,.
\end{equation}
\label{item:lem_a_priori_bound_on_profiles}
\end{enumerate}
\end{lemma}
\begin{proof}
Let $V$ be in $\vvvQuad{R}$ and let $c$ be a real quantity. According to the definition \cref{notation_vvvQuad_of_R} of $\vvvQuad{R}$, there exists a positive quantity $K$ such that, for every $u$ in $\rr^d$, 
\[
\abs{\nabla V(u)}\le K + \abs{u}
\,.
\]
As a consequence, it follows from the expression \cref{vector_field} of $F_{c,V}$ that, for every solution $U=(u,v)$ of \cref{trav_wave_system_order_1} in $\rr^{2d}$, 
\[
\abs{\dot U(\xi)}=\abs{F_{c,V}(u,v)} = \abs{(v,\nabla V(u)-cv)} = \mathcal{O}_{|U|\rightarrow \infty } \big( \abs{U(\xi)}\big)\,.
\]
This bound prevents solutions from blowing up in finite time, proving conclusion \cref{item:lem_global_flow}. 

Now let $\xi\mapsto u(\xi)$ denote a solution of \cref{trav_wave_system_order_2} approaching critical points of $V$ at both ends of $\rr$. Let us write $q=|u|^2/2$, so that
\begin{equation}
\label{dot_S_a_priori_bounds}
\dot q = u\cdot \dot u \quad \text{ and }\quad \ddot q = -c \dot q +  \dot u^2 + u\cdot\nabla V(u)
\,,
\end{equation}
and so that, since $V$ is in $\vvvQuad{R}$, for every real quantity $\xi$, 
\[
\abs{u(\xi)}\ge R \implies \frac{d}{{d}\xi}_{|\xi} \big( e^{c\xi}\dot q(\xi)\big) = e^{c\xi}\bigl(\dot u^2(\xi) + u^2(\xi)\bigr)>0
\,.
\]
Since $V$ is quadratic outside the ball $B_{\rr^d}(0,R)$, its critical points must belong to the interior of $B_{\rr^d}(0,R)$, and the same must be true for $u(\xi)$ when $\abs{\xi}$ is large. Now, if $\abs{u(\cdot)}$ were to reach the value $R$ at some (finite) time $\xi_0$, then (if $\xi_0$ is the first time when this happens) $\dot q(\xi_0)$ would be nonnegative; the implications above show that, from this time on, the quantity $e^{c\xi}\dot q(\xi)$ (and thus also the quantity $\dot q(\xi)$) would remain positive; so that $q(\xi)$ and $\abs{u(\xi)}$ would keep increasing with $\xi$, a contradiction with the fact that $u(\xi)$ must be back inside $B_{\rr^d}(0,R)$ for $\xi$ large. Conclusion \cref{item:lem_a_priori_bound_on_profiles} is proved. 
\end{proof}
\subsection{Topological properties of restriction maps}
\label{subsec:topological_properties_restriction_maps}
Let $R$ denote a positive quantity and let us consider the set
\begin{equation}
\label{notation_vvvRes}
\vvvRes{R} = \ccc^{k+1}\bigl(\widebar{B}_{\rr^d}(0,R),\rr\bigr)
\,.
\end{equation}
The next \cref{lem:res_R_R_prime_is_continuous_and_surjective_and_open} will be used to carry out, in \cref{sec:proof_main}, the extension mentioned at the beginning of this \namecref{subsec:potentials_quadratic_past_radius_R}. To ease its formulation, let us adopt $\vvvQuad{\infty}$ as an alternative notation for the space $\vvvFull$. Let $R'$ denote either a quantity larger than $R$ or $\infty$, and let us consider the restriction operator: 
\begin{equation}
\label{def_res_R_prime_R}
\res_{R,R'}:\vvvQuad{R'}\to\vvvRes{R}\,, \quad V\mapsto V_{|\widebar{B}_{\rr^d}(0,R)}
\,.
\end{equation}
\begin{lemma}
\label{lem:res_R_R_prime_is_continuous_and_surjective_and_open}
The restriction map $\res_{R,R'}$ is continuous, surjective and open. 
\end{lemma}
\begin{proof}
If two potentials of $\vvvQuad{R'}$ are $\ccc^k$-close, then their restrictions to the closed ball $\widebar{B}_{\rr^d}(0,R)$ are still $\ccc^k$-close on this ball, so that the map $\res_{R,R'}$ is continuous. 

To prove that the map $\res_{R,R'}$ is surjective and open, it is sufficient to construct a continuous right inverse for this map. For this purpose we may consider Seeley's extension 
\[
\ext_{\infty,R}:\vvvRes{R}\to\vvvFull
\,,
\]
 which is a right inverse for $\res_{R,\infty}$ (that is $\res_{R,\infty}\circ \ext_{\infty,R}$ is the identity map of $\vvvRes{R}$). The map defined in Seeley's original paper \cite{Seeley_extensionCinftyFunctionsHalfSpace_1964} extends to the whole space $\rr^d$ a function initially defined on a half space, but using spherical coordinates the same definition leads to this extension $\ext_{\infty,R}$. This map $\ext_{\infty,R}$ is linear and continuous for the usual topology for the departure set $\vvvRes{R}$ and the topology of uniform convergence of derivatives up to order $k$ on compact subsets of $\rr^d$ for the arrival set $\vvvFull$. Now, if $\chi:[0,+\infty)\to\rr$ denotes a smooth truncation function satisfying
\[
\chi\equiv 1 \text{ on }[0,R] 
\quad\text{and}\quad 
\chi\equiv 0 \text{ on }\bigl[\min(R+1,R'),+\infty\bigr)
\,, 
\]
then the map $\ext_{R',R}:\vvvRes{R}\to\vvvQuad{R}$ defined, for every $V$ in $\vvvRes{R}$, by 
\[
\ext_{R',R}(V)(u) = \chi(|u|) \ext_{\infty,R}(V)(u) + \bigl(1-\chi(|u|)\bigr) \frac{|u|^2}{2}
\,,
\] 
is a right inverse of $\res_{R,R'}$ and is continuous (for the topologies of uniform convergence of derivatives up to order $k$ for the departure and arrival sets). \Cref{lem:res_R_R_prime_is_continuous_and_surjective_and_open} is proved. 
\end{proof}
\begin{corollary}
\label{cor:truncation_extension_of_potentials}
For every couple $(A,B)$ of subsets of $\vvvRes{R}$, let $A'=\res_{R,R'}^{-1}(A)$ and $B'=\res_{R,R'}^{-1}(B)$ denote the sets of the potentials of $\vvvQuad{R'}$ whose restrictions to $\widebar{B}_{\rr^d}(0,R)$ belong to $A$ and $B$ respectively. Then the following equivalences hold: 
\begin{align}
\label{open_iff_open}
\text{$A$ is open in $\vvvRes{R}$}&\iff\text{$A'$ is open in $\vvvQuad{R'}$}\,,\\
\label{dense_iff_dense}
\text{$A$ is dense in $B$}&\iff\text{$A'$ is dense in $B'$}\,,\\
\label{dense_in_full_set_iff_dense_in_full_set}
\text{$A$ is dense in $\vvvRes{R}$}&\iff\text{$A'$ is dense in $\vvvQuad{R'}$}
\,.
\end{align}
\end{corollary}
\begin{proof}
Equivalence \cref{open_iff_open} follows from the continuity and the openness of $\res_{R,R'}$. 

According to the surjectivity of $\res_{R,R'}$, the set $\res_{R,R'}(A')$ is equal to $A$ and the set $\res_{R,R'}(B')$ is equal to $B$. Since the image of a dense set by a continuous map is dense in its image, if $A'$ is dense in $B'$ then $A$ is dense in $B$. Implication ``$\impliedby$'' of \cref{dense_iff_dense} is proved. 

On the other hand, if $A$ is dense in $B$, then, for every open subset $\Omega'$ of $B'$, its image $\Omega:=\res_{R,R'}(\Omega')$ is, according to \cref{lem:res_R_R_prime_is_continuous_and_surjective_and_open}, open in $B$ so that the intersection $A\cap\Omega$ is nonempty. According to the surjectivity of $\res_{R,R'}$, the set $\res_{R,R'}^{-1}(A\cap\Omega)$ is also nonempty and it is by construction included in $A'\cap\Omega'$, which is a fortiori nonempty. This proves that $A'$ is dense in $B'$ and completes the proof of equivalence \cref{dense_iff_dense}.

Finally, equivalence \cref{dense_in_full_set_iff_dense_in_full_set} follows from \cref{dense_iff_dense} by setting $B'$ equal to $\vvvQuad{R'}$ and $B$ equal to $\vvvRes{R}$. \Cref{cor:truncation_extension_of_potentials} is proved.
\end{proof}
\section{Generic transversality of travelling fronts}
\label{sec:generic_transversality_travelling_fronts}
\subsection{Notation and statement}
\label{subsec:notation_and_statement_gen_transv_trav_fronts}
\begin{notation}
Let us recall the notation $\vvvFull$ and $\vvvQuad{R}$ introduced in \cref{notation_vvvFull,notation_vvvQuad_of_R}.
For every potential function $V$ in $\vvvFull$, let $\SigmaCrit(V)$ and $\SigmaMin(V)$ denote the set of non-degenerate critical points and of non-degenerate minimum points of $V$, respectively, and let us consider the set
\begin{equation}
\label{notation_fff_V}
\begin{aligned}
\fff_V = \bigl\{ 
&(c,u) \in  (0,+\infty)\times \ccc^{k+1}(\rr,\rr^d): \xi\mapsto u(\xi) \text{ is a global solution of the} \\
&\text{system }\ddot u = - c\dot u + \nabla V(u) \text{ and there exists }(e_-,e_+)\text{ in }\\
&\SigmaCrit(V)\times\SigmaMin(V)\text{ such that }\lim_{\xi\to-\infty} u(\xi)= e_-\text{ and }\lim_{\xi\to+\infty} u(\xi)= e_+
\bigr\}
\,.
\end{aligned}
\end{equation}
In other words, $(c,u)$ is in $\fff_V$ if and only if $c$ is a positive quantity and $\xi\mapsto u(\xi)$ is the profile of a front travelling at speed $c$ and connecting a non-degenerate critical point (at the left end) to a non-degenerate minimum point (at the right end), for the potential $V$. 
\end{notation}
Let us take and fix a positive quantity $R$. The goal of this section is to prove \cref{prop:global_generic_transversality_travelling_fronts} below, which is a weaker version of statement \cref{item_thm:main_travelling_front} of \cref{thm:main} since the potentials under consideration belong to the subspace $\vvvQuad{R}$ and not to the full space $\vvvFull$. The reasons for first proving the intended genericity result in this restricted setting are explained at the beginning of \cref{subsec:potentials_quadratic_past_radius_R}, and the extension from $\vvvQuad{R}$ to $\vvvFull$ will be carried out in the last \cref{sec:proof_main}. As a reminder, the transversality of a travelling front was defined in \cref{def:transverse_travelling_front}. 
\begin{proposition}
\label{prop:global_generic_transversality_travelling_fronts}
For every positive quantity $R$, there exists a generic subset of $\vvvQuad{R}$ such that, for every potential $V$ in this subset, every travelling front $(c,u)$ in $\fff_V$ is transverse. 
\end{proposition}
\subsection{Reduction to a local statement}
\label{subsec:reduction}
Let $V_0$ denote a potential function in $\vvvQuad{R}$, and let $e_{-,0}$ and $e_{+,0}$ denote a non-degenerate critical point and a non-degenerate minimum point of $V_0$, respectively. According to \cref{prop:loc_stab_unstab_man_c_positive} (or simply to the implicit function theorem), there exists a small neighbourhood $\nuRobust(V_0,e_{-,0},e_{+,0})$ of $V_0$ in $\vvvQuad{R}$ and two $\ccc^{k+1}$-functions $e_-(\cdot)$ and $e_+(\cdot)$, defined on $\nuRobust(V_0,e_{-,0},e_{+,0})$ and with values in $\rr^d$, such that $e_-(V_0)$ equals $e_{-,0}$ and $e_+(V_0)$ equals $e_{+,0}$ and, for every $V$ in $\nuRobust(V_0,e_{-,0},e_{+,0})$, both $e_-(V)$ and $e_+(V)$ are critical point of $V$ close to $e_{+,0}$. The following local generic transversality statement, which calls upon this notation, yields \cref{prop:global_generic_transversality_travelling_fronts} (as shown below). 
\begin{proposition}
\label{prop:local_generic_transversality_travelling_fronts_given_critical_points_speed}
For every positive speed $c_0$, there exist a neighbourhood $\nu_{\VZeroeMinusZeroePlusZerocZero}$ of $V_0$ in $\vvvQuad{R}$, included in $\nuRobust(V_0,e_{-,0},e_{+,0})$, a neighbourhood $\ccc_{\VZeroeMinusZeroePlusZerocZero}$ of $c_0$ in $(0,+\infty)$, and a generic subset $\nuThingGen{\VZeroeMinusZeroePlusZerocZero}$ of $\nu_{\VZeroeMinusZeroePlusZerocZero}$ such that, for every $V$ in $\nuThingGen{\VZeroeMinusZeroePlusZerocZero}$, every front travelling at a speed $c$ in $\ccc_{\VZeroeMinusZeroePlusZerocZero}$ and connecting $e_-(V)$ to $e_+(V)$, for the potential $V$, is transverse. 
\end{proposition}
\begin{proof}[Proof that \cref{prop:local_generic_transversality_travelling_fronts_given_critical_points_speed} yields \cref{prop:global_generic_transversality_travelling_fronts}]
Let us denote by $\vvvQuadMorse{R}$ the\break
dense open subset of $\vvvQuad{R}$ defined by the Morse property (see \cite{Hirsch_differentialTopology_1976}):
\begin{equation}
\label{def_vvv_Quad_R_Morse}
\vvvQuadMorse{R} = \left\{ V\in\vvvQuad{R} : \text{all critical points of $V$ are non-degenerate} \right\}
\,.
\end{equation}
Let $V_0$ denote a potential function in $\vvvQuadMorse{R}$.  Its critical points are non-degenerate and thus isolated and, since $V_0$ is in $\vvvQuad{R}$, they belong to the open ball $B_{\rr^d}(0,R)$, so that those critical points are in finite number. Assume that \cref{prop:local_generic_transversality_travelling_fronts_given_critical_points_speed} holds. With the notation of this proposition, let us consider the following three intersections, at each time over all couples $(e_{-,0},e_{+,0})$ with $e_{-,0}$ a non-degenerate critical point and $e_{+,0}$ a non-degenerate minimum point of $V_0$:
\begin{equation}
\label{three_sets_ccc_nu_nuGen}
\begin{aligned}
\nu_{\VZerocZero} &= \nuRobust(V_0)\cap\left(\bigcap\nu_{\VZeroeMinusZeroePlusZerocZero}\right)\,, \\
\ccc_{\VZerocZero} &= \bigcap\ccc_{\VZeroeMinusZeroePlusZerocZero} \\
\text{and}\quad
\nuThingGen{\VZerocZero} &= \nuRobust(V_0)\cap\left(\bigcap\nuThingGen{\VZeroeMinusZeroePlusZerocZero}\right)
\,.
\end{aligned}
\end{equation}
Those are finite intersections, so that $\nu_{\VZerocZero}$ is still a neighbourhood of $V_0$ in $\vvvQuad{R}$, $\ccc_{\VZerocZero}$ is still a neighbourhood of $c_0$ in $(0,+\infty)$ and the set $\nuThingGen{\VZerocZero}$ is still a generic subset of $\nu_{\VZerocZero}$. Let $I$ denote a compact sub-interval of $(0,+\infty)$; the three sets defined above in \cref{three_sets_ccc_nu_nuGen} can be constructed likewise for every $c_0$ in $I$. Since $I$ is compact, it can be covered by a finite union of  sets $\ccc_{V_0,c_{0,i}}$, corresponding to a finite set $\{c_{0,1},\dots,c_{0,p}\}$ of speeds. Again the intersections
\[
\nu_{\VZeroI} = \bigcap_{1\le i\le p} \nu_{V_0,\, c_{0,i}}
\quad\text{and}\quad
\nuThingGen{\VZeroI} = \bigcap_{1\le i\le p} \nuThingGen{V_0,\, c_{0,i}}
\,.
\]
are finite and thus $\nuThingGen{\VZeroI}$ is still a generic subset of $\nu_{\VZeroI}$, which is a neighbourhood of $V_0$ in $\vvvQuadMorse{R}$. By construction, for every potential function $V$ in $\nuThingGen{\VZeroI}$, all fronts travelling at a speed belonging to $I$ and connecting a critical point of $V$ to a minimum point of $V$ are transverse. 
In other words, the set
\[
\begin{aligned}
\vvvQuadMorseTransv{R}{I} = \bigl\{&V\in \vvvQuadMorse{R}: \text{ for every travelling front $(c,u)$ in $\fff_V$,} \\
&\text{if $c$ is in $I$ then $(c,u)$ is transverse}\bigr\}
\,,
\end{aligned}
\]
is locally generic in the sense that  $\vvvQuadMorseTransv{R}{I}\cap\nu_{\VZeroI}$ is generic in $\nu_{\VZeroI}$.
Since $\vvvQuad{R}$ is separable, applying \cref{lem:local_genericity_implies_global_genericity} with $\vvv =\vvvQuad{R}$, 
$\vvvDense= \vvvQuadMorse{R}$, $\vvvGen=\vvvQuadMorseTransv{R}{I}$ and $\nu=\nu_{\VZeroI}$ shows that the set $\vvvQuadMorseTransv{R}{I}$ is generic in the whole set $\vvvQuad{R}$. 
As a consequence, the set
\[
\bigcap_{q\in\nn^*} \vvvQuadMorseTransv{R}{[-1/q,q]}
\]
is still generic in $\vvvQuad{R}$. For every potential $V$ in this set, all travelling fronts belonging to $\fff_V$ are transverse, so that this set fulfils the conclusions of \cref{prop:global_generic_transversality_travelling_fronts}.
\end{proof}
The remaining part of \cref{sec:generic_transversality_travelling_fronts} will thus be devoted to the proof of \cref{prop:local_generic_transversality_travelling_fronts_given_critical_points_speed}. 
\subsection{Proof of the local statement (\texorpdfstring{\cref{prop:local_generic_transversality_travelling_fronts_given_critical_points_speed}}{Proposition \ref{prop:local_generic_transversality_travelling_fronts_given_critical_points_speed}})}
\label{subsec:proof_local_gen}
\subsubsection{Setting}
For the remaining part of this section, let us fix a potential function $V_0$ in $\vvvQuad{R}$, a non-degenerate critical point $e_{-,0}$ of $V_0$ and a non-degenerate minimum point $e_{+,0}$ of $V_0$, differing from $e_{-,0}$. According to \cref{prop:loc_stab_unstab_man_c_positive}, there exist a neighbourhood $\nu$ of $V_0$ in $\vvvQuad{R}$, included in $\nuRobust(V_0,e_{-,0},e_{+,0})$, a neighbourhood $\ccc$ of $c_0$ in $(0,+\infty)$, and a positive quantity $r$ such that, for every $(c,V)$ in $\ccc\times\nu$, there exist $C^{k+1}$-functions
\[
\hatwuloc{c}{V}:\Bu_{E_{-,0}}(r)\to\rr^{2d}
\quad\text{and}\quad
\hatwsloc{c}{V}:\Bs_{E_{+,0}}(r)\to\rr^{2d}
\]
such that the sets
\[
\Wuloc{c}{V}\bigl(E_-(V)\bigr) = \hatwuloc{c}{V}\bigl(\Bu_{E_{-,0}}(r)\bigr) 
\quad\text{and}\quad
\Wsloc{c}{V}\bigl(E_+(V)\bigr) = \hatwsloc{c}{V}\bigl(\Bs_{E_{+,0}}(r)\bigr)
\]
define a local unstable manifold of $E_-(V)$ and a local stable manifold of $E_+(V)$, respectively (see the conclusions of \cref{prop:loc_stab_unstab_man_c_positive} and equalities \cref{def_local_stab_unstab_man_as_images_non_zero_speed}). 

Here is the setting where Sard--Smale theorem (\cref{thm:Sard_Smale}) will be applied (see \cref{fig:phi_map}). 
\begin{figure}[htbp]
\centering
\resizebox{\textwidth}{!}{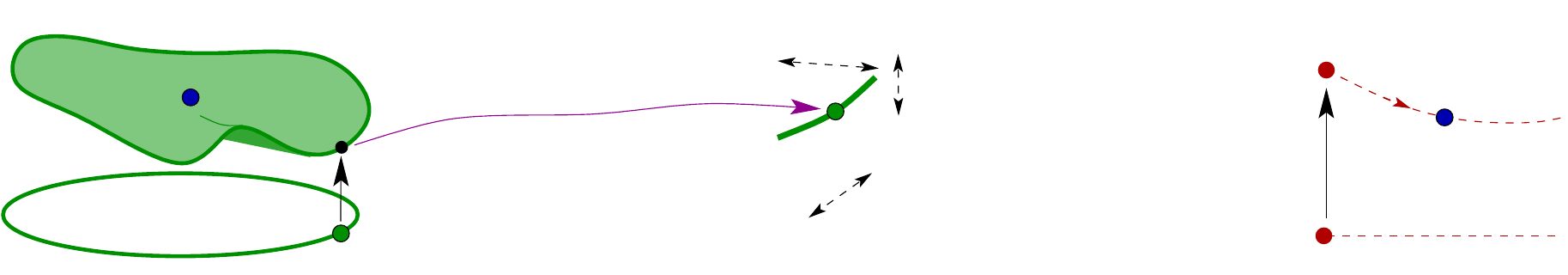}
\caption{The function $\hatwuloc{c}{V}(\cdot)$ maps $\bbbu$ onto the boundary of the local unstable manifold $\Wuloc{c}{V}\bigl(E_-(V)\bigr)$. A point $\hatwuloc{c}{V}(\bu)$ of this boundary is pushed forward during a time $\xi$ by the flow $S_{c,V}(\xi,\cdot)$ to give the image $\Phiu(\bu)$, which still belongs to the global unstable manifold of $E_-(V)$. On the other hand, $\Phis$ maps $\bbbs$ onto the boundary of the local stable manifold $\Wsloc{c}{V}(E_+)$. The dependence of $\Phiu$ on the time $\xi$ and the point $\bu$ provides a number of degrees of freedom equal to the dimension of the unstable manifold, while an additional degree of freedom is provided by the speed $c$. This additional dependence makes the difference between the transversality of a travelling front as defined in \cref{def:transverse_travelling_front} and the classical transversality of stable and unstable manifolds.}
\label{fig:phi_map}
\end{figure}
Let
\[
\begin{aligned}
\bbbu &= \partial\Bu_{E_{-,0}}(r)\,,
&\ &\text{ }\ &
\bbbs &= \partial\Bs_{E_{+,0}}(r)\,,
&\ &\text{ }\ &
\mmm &= \bbbu\times\bbbs\times\rr\times\ccc
\,, \\
\text{ }\quad
\Lambda &= \nu\,,
&\ &\text{ }\ &
\nnn &= (\rr^{2d})^2\,,
&\ &\text{and}\ &
\www &= \{(A,B)\in\nnn : A = B\}
\,.
\end{aligned}
\]
Notice that $\www$ is the diagonal of $\nnn$. Let us consider the functions
\[
\begin{aligned}
\Phiu &: \bbbu\times\rr\times\ccc \times\Lambda&\to\rr^{2d}&
\,, \quad
(\bu,\xi,c,V)&~\longmapsto~& S_{c,V}\bigl(\xi,\hatwuloc{c}{V}(\bu)\bigr) \\
\text{and}\quad 
\Phis &: 
\bbbs\times\ccc \times\Lambda&\to\rr^{2d} &
\,, \quad
(\bs,c,V)&~\longmapsto~& \hatwsloc{c}{V}(\bs)\,.
\end{aligned}
\]
For every $V$ in $\Lambda$ and $c$ in $\ccc$, the image of $\Phiu(\cdot,\cdot,c,V)$ is the global unstable manifold $\Wu_{c,V}\bigl(E_-(V)\bigr)$ (except the point $E_-(V)$ itself), whereas the image of $\Phis(\cdot,c,V)$ is the boundary of the local stable manifold $\Wsloc{c}{V}\bigl(E_+(V)\bigr)$. Finally, let
\begin{equation}
\label{def_Phi_trav_front}
\Phi : \mmm\times\Lambda \to\nnn 
\,, \quad
(m,V)=\bigl((\bu,\bs,\xi,c),V\bigr) \mapsto \bigl(\Phiu(\bu,\xi,c, V),\Phis(\bs,c,V)\bigr)
\,.
\end{equation}
\subsubsection{Additional conditions on \texorpdfstring{$\nu$}{nu} and \texorpdfstring{$r$}{r}}
The main step in the proof of \cref{prop:local_generic_transversality_travelling_fronts_given_critical_points_speed} is the construction of a suitable perturbation $W$ of $V$ (carried out in \cref{subsubsec:checking_hyp_ii_Sard_Smale_thm} below). This construction requires more accurate conditions on the setting above. 

First, since $e_{-,0}$ and $e_{+,0}$ differ, we may assume that $\nu$ and $\ccc$ and $r$ are small enough so that, for every $V$ in $\nu$,
\begin{equation}
\label{pi_pos_of_local_unstable_stable_man_do_not_intersect}
\piPos\Bigl(\Wuloc{c}{V}\bigl(E_-(V)\bigr)\Bigr)\cap\piPos\Bigl(\Wsloc{c}{V}\bigl(E_+(V)\bigr)\Bigr)=\emptyset
\,,
\end{equation}
where $\piPos$ is the projection on the first component defined in \cref{def_pi_pos}. 

Next, the following lemma is a more uniform version of assertion \cref{item_prop:values_reached_only_once_travelling_front} of \cref{prop:values_reached_only_once}, the key difference being that $r$ can be chosen small enough such that $\IOnce$ contains positive times.
\begin{lemma}
\label{lem:additional_condition_nu_and_r}
Up to replacing $\nu$ by a smaller neighbourhood of $V_0$ in $\vvvQuad{R}$, and $\ccc$ by a smaller neighbourhood of $c_0$ in $(0,+\infty)$, and $r$ by a smaller radius, we may assume that the following conclusions hold. For every $V$ in $\nu$, every $c$ in $\ccc$, and every solution $\xi\mapsto U(\xi) = \bigl(u(\xi),\dot u(\xi)\bigr)$ of system \cref{trav_wave_system_order_1} such that $U(0)$ belongs to the boundary of $\Wuloc{c}{V}\bigl(E_-(V)\bigr)$ (in other words there exists $\bu$ in $\bbbu$ such that $U(0)$ equals $\hatwuloc{c}{V}(\bu)$), there exists a a compact interval with nonempty interior $\IOnce$, included in $(0,+\infty)$, such that:  
\begin{enumerate}
\item the function $\xi\mapsto \abs{u(\xi)-e_-(V)}$ is increasing on $\IOnce$ (so that $u_{|\IOnce}$ is a diffeomorphism onto its image),
\label{item:lem_additional_condition_nu_and_r_diffeo}
\item and for all $\xi^*$ in $\IOnce$ and $\xi$ in $\rr$, $u(\xi) = u(\xi^*)$ implies $\xi = \xi^*$,
\label{item:lem_additional_condition_nu_and_r_only_once}
\item and condition \cref{pi_pos_of_local_unstable_stable_man_do_not_intersect} holds, and in addition,
\label{item:lem_additional_condition_nu_and_r_3}
\end{enumerate}
\[
u(\IOnce) \cap \piPos\Bigl[\Wuloc{c}{V}\bigl(E_-(V)\bigr) \cup \Wsloc{c}{V}\bigl(E_+(V)\bigr)\Bigr]=\emptyset
\,.
\]
\end{lemma}
\begin{proof}
Consider for now that $\nu$ and $\ccc$ and $r$ are as in the previous subsection, and, for some $\bu$ in $\bbbu$, let us consider the solution $\xi\mapsto U(\xi) = \bigl(u(\xi),\dot u(\xi)\bigr)$ of system \cref{trav_wave_system_order_1} defined as:
\[
U(\xi) = S_{c_0,V_0}\bigl(\xi,\hatwuloc{c_0}{V_0}(\bu)\bigr)
\quad\text{(so that}\quad 
U(0) = \hatwuloc{c_0}{V_0}(\bu)
\text{).}
\]
The same arguments as in the proof of statement \cref{item_prop:values_reached_only_once_travelling_front} of \cref{prop:values_reached_only_once} yield the following conclusions. First, there exists a (large, negative) time $\xi_0(\bu)$ such that the function $\xi\mapsto \abs{u(\xi)-e_{-,0}}$ is increasing on $\bigl(-\infty,\xi_0(\bu)\bigr]$. Then, there exists $\xiOnce(\bu)$ in $\bigl(-\infty,\xi_0(\bu)\bigr)$ such that, for every $\xi^*$ in $\bigl(-\infty,\xiOnce(\bu)\bigr]$, 
\[
H_V\Bigl(U\bigl(\xi_0(\bu)\bigr)\Bigr) < - V\bigl(u(\xi^*)\bigr)
\]
(which is nothing but inequality \cref{minus_V_larger_than_H_at_xiExit_for_large_negative_xi}). Then, it follows from statement \cref{item_prop:values_reached_only_once_travelling_front} of \cref{prop:values_reached_only_once} that, for the interval $\IOnce$ equal to $\bigl[\xiOnce(\bu)-2,\xiOnce(\bu)-1\bigr]$, conclusions \cref{item:lem_additional_condition_nu_and_r_diffeo,item:lem_additional_condition_nu_and_r_only_once} of \cref{lem:additional_condition_nu_and_r} hold for the solution $U$ (and they still hold if $\xiOnce(\bu)$ is replaced by a smaller quantity). 

Now, observe that, due to the smooth dependence of the map $(-\infty,0]\to\rr^{2d}$, $\xi\mapsto S_{c,V}\bigl(\xi,\hatwuloc{c}{V}(\bu)\bigr)$ on $V$ and $c$ and $\bu$, this construction can be made uniform with respect to $\bu$ in a (small) open subset $\Omega$ of $\bbbu$ and $V$ in a (small) neighbourhood $\nu_\Omega$ (included in $\nu$) of $V_0$, and to $c$ in a (small) neighbourhood $\ccc_\Omega$ (included in $\ccc$) of $c_0$; in other words, there exists a (sufficiently large negative) quantity $\xiOnce(\Omega)$ such that the conclusions above hold for all such $V$ and $c$ and $\bu$. Since $\bbbu$ is compact, it can be covered by a finite number $\Omega_1\dots\Omega_n$ of such open subsets. Thus, replacing 
\[
\nu
\text{ by }
\bigcap_{i=1}^n \nu_{\Omega_i}
\quad\text{and}\quad
\ccc
\text{ by }
\bigcap_{i=1}^n \ccc_{\Omega_i}
\,,
\]
and choosing
\[
\xiOnce = \min_{i\in\{1,\dots,n\}}\xiOnce(\Omega_i)
\quad\text{and}\quad
\IOnce = [\xiOnce-2,\xiOnce-1]
\,,
\]
conclusions \cref{item:lem_additional_condition_nu_and_r_diffeo,item:lem_additional_condition_nu_and_r_only_once} of \cref{lem:additional_condition_nu_and_r} hold. Up to replacing $r$ by a smaller positive quantity, we may assume in addition that $\IOnce$ belongs to $(0,+\infty)$. Finally, again up to replacing $r$ by a smaller positive quantity, we may assume that conclusion \cref{item:lem_additional_condition_nu_and_r_3} also holds. 
\end{proof}

\subsubsection{Equivalent characterizations of transversality}
\label{subsubsec:equivalence_transverse}
Let us consider the set
\[
\begin{aligned}
\fff_{\LambdaC} = \bigl\{ &(V,c,u): \ V\in\Lambda \text{ and } c\in\ccc \text{ and $u$ is the profile of a front travelling} \\ 
&\text{at speed $c$ and connecting $e_-(V)$ to $e_+(V)$, for the potential $V$}\bigr\}
\,,
\end{aligned}
\]
and let us denote by $\widetilde{\fff}_{\LambdaC}$ the set of equivalence classes of $\fff_{\LambdaC}$ for the equivalence relation: $(V,c,u)\sim(V^\dag,c^\dag,u^\dag)$ if and only if $V=V^\dag$ and $c=c^\dag$ and $u=u^\dag$ up to a translation of the time. The aim of this subsection is to prove \cref{prop:equivalence_transversality} below, relating the transversality of the intersection $\Phi(\mmm\times\Lambda)\cap\www$ to the transversality of travelling fronts belonging to $\fff_{\LambdaC}$. To begin with, the next \cref{prop:one_to_one_correspondence_between_fronts_and_diagonal_intersection} formalizes the correspondence between the intersection of the image of $\Phi$ with the diagonal $\www$ and the profiles of such travelling fronts.  
\begin{proposition}
\label{prop:one_to_one_correspondence_between_fronts_and_diagonal_intersection}
The map 
\begin{equation}
\label{correspondence_im_of_Phi_cap_www_travelling_front}
\Phi^{-1}(\www) \to \fff_{\LambdaC}
\,,\quad
(\bu,\bs,\xi,c,V)\mapsto \biggl(V,c,\xi'\mapsto \piPos\Bigl(S_{c,V}\bigl(\xi',\hatwuloc{c}{V}(\bu)\bigr)\Bigr)\biggr)
\end{equation}
defines a a one-to-one correspondence between $\Phi^{-1}(\www)$ and the quotient set $\widetilde{\fff}_{\LambdaC}$.
\end{proposition}
\begin{proof}
The image by $\Phi$ of a point $(\bu,\bs,\xi,c,V)$ of $\mmm\times\Lambda$ belongs to the diagonal $\www$ of $\nnn$ if and only if $\Phiu(\bu,\xi,c, V)=\Phis(\bs,c,V)$. If this last equality holds, the function $u:\xi'\mapsto \Phiu(\bu,\xi',c, V)$ is a solution belonging to the unstable manifold $\Wu_{c,V}\bigl(E_-(V)\bigr)$ such that $u(\xi)=\Phis(\bs,c,V)$ belongs to the local stable manifold of $E_+(V)$. Thus $u$ defines the profile of a front travelling at speed $c$ and connecting $e_-(V)$ to $e_+(V)$. The map \cref{correspondence_im_of_Phi_cap_www_travelling_front} is thus well defined. 

Now, if $\xi\mapsto u(\xi)$ is the profile of a front travelling at a speed $c$ in $\ccc$ for the potential $V$ and connecting $e_-(V)$ to $e_+(V)$, then, according to \cref{prop:profiles_of_fronts_pulses_as_intersections_between_Wu_Ws}, the image of $\xi\mapsto\bigl(u(\xi),\dot u(\xi)\bigr)$ belongs to the intersection $\Wu_{c,V}\bigl(E_-(V)\bigr)\cap\Ws_{c,V}\bigl(E_+(V)\bigr)$. As a consequence, this image must cross the boundary of $\Wuloc{c}{V}\bigl(E_-(V)\bigr)$ at a time $\xi_-$ and the boundary of $\Wsloc{c}{V}\bigl(E_+(V)\bigr)$ at a time $\xi_+$:  there exists $\bu$ in $\bbbu$ and $\bs$ in $\bbbs$ such that $\bigl(u(\xi_-),\dot u(\xi_-)\bigr) = \hatwuloc{c}{V}(\bu)$ and $\bigl(u(\xi_+),\dot u(\xi_+)\bigr) = \hatwsloc{c}{V}(\bs)$. By construction, $\Phiu(\bu,\xi_+ - \xi_-,c, V) = \Phis(\bs,c,V)$ and thus $\Phi(\bu,\bs,\xi_+ - \xi_-,c, V))$ is in $\www$.
In addition, according to the remark at the end of \cref{subsec:local_stable_unstable_manifolds_positive_speed}, the times $\xi_-$ and $\xi_+$ at which these intersections occur are unique (for a given profile $\xi\mapsto u(\xi)$), thus so are the points $\bu$ in $\bbbu$ and $\bs$ in $\bbbs$ and the time lag $\xi_+-\xi_-$. This completes the proof of this one-to-one correspondence. 
\end{proof}
Both corresponding notions of transversality are related as follows.
\begin{proposition}
\label{prop:equivalence_transversality}
For every potential function $V$ in $\Lambda$, the following two statements are equivalent.
\begin{enumerate}
\item The image of the function $\mmm\to\nnn$, $m\mapsto\Phi(m,V)$ is transverse to $\www$.
\label{item:lem_equivalence_transversality_Phi}
\item Every profile $\xi\mapsto u(\xi)$ of a front travelling at a speed $c$ in $\ccc$ and connecting $e_-(V)$ to $e_+(V)$, for the potential $V$, is transverse. 
\label{item:lem_equivalence_transversality_fronts}
\end{enumerate}
\end{proposition}
\begin{proof}
Let us take $(m_1,V_1)$ in $\mmm\times\Lambda$ such that $\Phi(m_1,V_1)$ is in $\www$, let $(\bu_1,\bs_1,\xi_1,c_1)$ denote the point $m_1$ and let $\xi\mapsto u_1(\xi)$ denote the profile of the corresponding travelling front. In other words, 
\[
\text{for all }\xi\text{ in }\rr\,,\quad
U_1(\xi) = \Phiu(\bu_1,\xi,c_1, V_1)\,,
\quad\text{where}\quad 
U_1(\xi) = \bigl(u_1(\xi),\dot u_1(\xi)\bigr)
\,.
\]
Let us consider the maps
\[
\begin{aligned}
{\Gamma\Phi}:(\bbbu\times\rr\times\ccc)\times(\bbbs\times\ccc)&\to\rr\times\rr^{2d} \\
\bigl((\bu,\xi,\cu),(\bs,\cs)\bigr)&\mapsto \bigl(\cu,\Phiu(\bu,\xi,\cu,V_1)\bigr)+ \bigl(\cs,\Phis(\bs,\cs,V_1)\bigr)
\,,
\end{aligned}
\]
and
\[
\Delta\Phi:\mmm\to\rr^{2d} 
\,,\quad
(\bu,\bs,\xi,c)\mapsto \Phiu(\bu,\xi,c,V_1) - \Phis(\bs,c,V_1)
\,.
\]
and let us write, only for this proof, $D\Phi$ for $D_{T_{m_1}\mmm}\Phi$, and similarly $D\Phiu$ and $D\Phis$ and $D({\Gamma\Phi})$ and $D(\Delta\Phi)$ for the differentials of $\Phiu$ and $\Phis$ and ${\Gamma\Phi}$ and $\Delta\Phi$ at $(m_1,V_1)$ and with respect to all variables but $V$. 
\renewcommand{\qedsymbol}{}
\end{proof}
\begin{lemma}
\label{lem:equivalence_between_transv_ABC}
The following three statements are equivalent. 
\begin{enumerate}[(A)]
\item The image of $D\Phi$ contains a supplementary subspace of the diagonal $\www$ of $(\rr^{2d})^2$. 
\label{item:transversality_Phi}
\item The map $D({\Gamma\Phi})$ is surjective. 
\label{item:transversality_widebar_Phi}
\item The map  $D(\Delta\Phi)$ is surjective. 
\label{item:transversality_Delta_Phi}
\end{enumerate}
\end{lemma}
\begin{proof}[Proof of \cref{lem:equivalence_between_transv_ABC}]
If statement \cref{item:transversality_Phi} holds, then, for every $(\alpha,\beta)$ in $(\rr^{2d})^2$, there exist $\gamma$ in $\rr^{2d}$ and $\delta m$ in $T_{m_1}\mmm$ such that
\begin{equation}
\label{supplementary_of_diagonal}
(\gamma,\gamma) + D\Phi\cdot\delta m = (\alpha,\beta)
\,,
\end{equation}
so that
\begin{equation}
\label{DDelta_Phi_surjective}
D(\Delta\Phi)\cdot\delta m = \alpha-\beta
\,,
\end{equation}
and statement \cref{item:transversality_Delta_Phi} holds. Conversely, if statement \cref{item:transversality_Delta_Phi} holds, then, for every $(\alpha,\beta)$ in $(\rr^{2d})^2$, there exists $\delta m$ in $T_{m_1}\mmm$ such that \cref{DDelta_Phi_surjective} holds, and as a consequence, if $(\delta\bu,\delta\bs,\delta\xi,\delta c)$ denotes the components of $\delta m$, the vector $\alpha - D\Phiu(\delta\bu,\delta\xi,\delta c)$ is equal to $\beta-D\Phis(\delta\bs,\delta c)$, and if this vector is denoted by $\gamma$, then equality \cref{supplementary_of_diagonal} holds, and this shows that statement \cref{item:transversality_Phi} holds. Thus statements \cref{item:transversality_Phi} and \cref{item:transversality_Delta_Phi} are equivalent. 

Now, if statement \cref{item:transversality_widebar_Phi} holds, then, for every $(\delta c,\delta U)$ in $\rr\times\rr^{2d}$, there exist $(\delta\bu,\delta\xi,\delta\cu)$ in $T_{\bu_1}\bbbu\times\rr^2$ and $(\delta\bs,\delta\cs)$ in $T_{\bs_1}\bbbs\times\rr$ such that 
\begin{equation}
\label{surjectivity_of_widebarPhi}
(\delta c,\delta U) = \bigl(\delta \cu,D\Phiu\cdot(\delta\bu,\delta\xi,\delta \cu)\bigr) + \bigl(\delta \cs, D\Phis\cdot(\delta \bs ,\delta \cs)\bigr)
\,,
\end{equation}
so that $\delta c$ is equal to $\delta \cu+\delta \cs$ and so that
\[
\begin{aligned}
\delta U &= D\Phiu\cdot(\delta\bu,\delta\xi,\delta \cu) + D\Phis\cdot(\delta \bs ,\delta c -\delta \cu)\\
&= D\Phiu\cdot(\delta\bu,\delta\xi,\delta \cu)  + D\Phis\cdot(0,\delta c) - D\Phis\cdot(-\delta \bs ,\delta \cu)
\,,
\end{aligned}
\]
so that finally, if $(\delta\bu,-\delta \bs,\delta\xi,\delta \cu)$ is denoted by $\delta m$, then 
\begin{equation}
\label{deltaU_related_to_DDeltaPhi}
\delta U = D(\Delta\Phi)\cdot \delta m + D\Phis\cdot(0,\delta c) 
\,.
\end{equation}
By choosing $\delta c$ equal to $0$, this shows that every $\delta U$ in $\rr^{2d}$ is in the image of $D(\Delta\Phi)$, which is statement \cref{item:transversality_Delta_Phi}. Conversely, if statement \cref{item:transversality_Delta_Phi} holds, then for every $(\delta c,\delta U)$ in $\rr\times\rr^{2d}$, there exists $\delta m$ in $T_{m_1}\mmm$ such that \cref{deltaU_related_to_DDeltaPhi} holds, and if $\delta\cu$ denotes the last component of $\delta m$ and $\delta\cs$ is the difference $\delta c-\delta\cu$, then equality \cref{surjectivity_of_widebarPhi} holds, and this shows that statement \cref{item:transversality_widebar_Phi} holds. Thus statements \cref{item:transversality_widebar_Phi} and \cref{item:transversality_Delta_Phi} are equivalent. 
\end{proof}
\begin{proof}[Continuation of the proof of \cref{prop:equivalence_transversality}]
To conclude, let us see how both transversality statements \cref{item:lem_equivalence_transversality_Phi,item:lem_equivalence_transversality_fronts} can be expressed in terms of the ingredients of \cref{lem:equivalence_between_transv_ABC}. On the one hand, according to \cref{def:transverse_travelling_front}, the travelling front with profile $u_1(\cdot)$ and speed $c_1$ is transverse if and only if the intersection
\begin{equation}
\label{transverse_gammaphi}
\left(\bigcup_{\cu>0} \{\cu\}\times \Wu_{\cu,V}\bigl(E_-(V)\bigr) \right) 
\cap 
\left(\bigcup_{\cs>0} \{\cs\}\times \Ws_{\cs,V}\bigl(E_+(V)\bigr) \right) 
\end{equation}
is transverse, in $\rr^{2d+1}$, along the set $\{c_1\}\times U_1(\rr)$. This transversality can be considered at a single point, no matter which, of the trajectory $U_0(\rr)$, thus in particular at the point $\Phiu(\bu_1,\xi_1,c_1, V_1)$, which is equal to $\Phis(\bs_1,c_1,V_1)$. By definition, the sum of the tangent spaces associated to the manifolds intersected in \cref{transverse_gammaphi} is the image of $D({\Gamma\Phi})$ and the transversality stated in statement \cref{item:lem_equivalence_transversality_fronts} is therefore equivalent to the surjectivity of the map $D({\Gamma\Phi})$ (statement \cref{item:transversality_widebar_Phi} in \cref{lem:equivalence_between_transv_ABC}). 

On the other hand, the image of the function $\mmm\to(\rr^{2d})^2$, $m\mapsto\Phi(m,V_1)$ is transverse at $\Phi(m_1,V_1)$ to the diagonal $\www$ of $(\rr^{2d})^2$ as stated in \cref{item:lem_equivalence_transversality_Phi} if and only if the image of $D\Phi$ contains a supplementary subspace of the diagonal (statement \cref{item:transversality_Phi} in \cref{lem:equivalence_between_transv_ABC}). Thus \cref{prop:equivalence_transversality} follows from \cref{lem:equivalence_between_transv_ABC}. 
\end{proof}
According to \cref{prop:equivalence_transversality}, \cref{prop:local_generic_transversality_travelling_fronts_given_critical_points_speed} follows from the conclusion of \cref{thm:Sard_Smale} applied to the function $\Phi$ (see \cref{subsubsec:conclusion_loc_gen_transv}). The next two \namecrefs{subsubsec:checking_hyp_regularity_Sard_Smale_trav_fronts} are devoted to checking that this function $\Phi$ fulfils hypotheses \cref{item:thm_sard_smale_condition_regularity,item:thm_sard_smale_condition_transversality} of this theorem.
\subsubsection{Checking hypothesis \texorpdfstring{\cref{item:thm_sard_smale_condition_regularity}}{\ref{item:thm_sard_smale_condition_regularity}} of \texorpdfstring{\cref{thm:Sard_Smale}}{Theorem \ref{thm:Sard_Smale}}}
\label{subsubsec:checking_hyp_regularity_Sard_Smale_trav_fronts}
Since the vector field \cref{vector_field} defining the flow \cref{flow} is of class $C^k$, so is the function $\Phi$. 
It follows from \cref{subsec:eigenspaces_and_dimensions} that
\[
\dim(\bbbu) = d - m(e_{-,0})-1
\text{ and }
\dim(\bbbs) = d-1
\,,\text{ thus }
\dim(\mmm) = 2d-m(e_{-,0})
\,,
\]
and since the codimension of $\www$ in $\nnn$ is equal to $2d$, 
\[
\dim(\mmm) - \codim(\www) = - m(e_{-,0})\le 0
\,,\quad\text{thus}\quad
k>\dim(\mmm) - \codim(\www)
\,;
\]
in other words, hypothesis \cref{item:thm_sard_smale_condition_regularity} of \cref{thm:Sard_Smale} is fulfilled. 
\subsubsection{Checking hypothesis \texorpdfstring{\cref{item:thm_sard_smale_condition_transversality}}{conditionTransversalitySardSmale} of \texorpdfstring{\cref{thm:Sard_Smale}}{Theorem \ref{thm:Sard_Smale}}}
\label{subsubsec:checking_hyp_ii_Sard_Smale_thm}
Take $(m_1,V_1)$ in the set $\Phi^{-1}(\www)$. Let $(\bu_1,\bs_1,\xi_1,c_1)$ denote the components of $m_1$, and, for every real quantity $\xi$, let us write
\[
U_1(\xi) = \bigl(u_1(\xi),v_1(\xi)\bigr) = S_{c_1,V_1}\bigl(\xi,\hatwuloc{c_1}{V_1}(\bu_1)\bigr)
\,.
\]
The function $\xi\mapsto u_1(\xi)$ is the profile of a front travelling at speed $c_1$ and connecting $e_-(V_1)$ to $e_+(V_1)$ for the potential $V_1$; and, according to the empty inclusion \cref{pi_pos_of_local_unstable_stable_man_do_not_intersect}, the quantity $\xi_1$ is \emph{positive}. Let us write
\[
D\Phi\,, \quad
D\Phiu
\quad\text{and}\quad
D\Phis
\]
for the \emph{full} differentials (with respect to arguments $m$ in $\mmm$ \emph{and} $V$ in $\Lambda$) of the three functions $\Phi$ and $\Phiu$ and $\Phis$ respectively at the points $\bigl(\bu_1,\bs_1,\xi_1,c_1,V_1\bigr)$, 
$\bigl(\bu_1,\xi_1,c_1,V_1\bigr)$ and $\bigl(\bs_1,c_1,V_1\bigr)$. Checking hypothesis \cref{item:thm_sard_smale_condition_transversality} of \cref{thm:Sard_Smale} amounts to prove that 
\begin{equation}
\label{transversality_of_DPhi}
\img(D\Phi) + T\www = T\nnn
\,.
\end{equation}
To this end, since the subspace $\rr^{2d}\times\{0_{\rr^{2d}}\}$ of $\nnn$ is transverse to the diagonal  $\www$, it is sufficient to prove that, for every $\gamma$ in $\rr^{2d}$, the vector $\bigl(\gamma,0_{\rr^{2d}}\bigr)$ can be reached by $D\Phi$. Thus, it is sufficient to prove that, for every $\gamma$ in $\rr^{2d}$, there exist a real quantity $\zeta$ and a function $W$ in $\CkbFull$ with a compact support $\supp(W)$ satisfying
\begin{equation}
\label{supp_W_included_in_BR}
\supp(W)\subset B_{\rr^d}(0,R)
\,,
\end{equation}
such that
\begin{align}
\label{transversality_Phi_condition_dPhiu}
D\Phiu\cdot(0,\zeta,0,W) &= \gamma \,, \\
\text{and}\quad
D\Phis\cdot(0,0,W) &= 0_{\rr^{2d}}
\,.
\label{transversality_Phi_condition_dPhis}
\end{align}
To fulfil equality \cref{transversality_Phi_condition_dPhis}, it is sufficient to assume that $W$ satisfies the following additional condition:
\begin{equation}
\label{supp_W_does_not_meet_projection_Wsloc}
\supp(W)\cap\piPos\Bigl[\Wsloc{c}{V}\bigl(E_+(V_1)\bigr)\Bigr] = \emptyset
\,,
\end{equation}
where $\piPos:\rr^{2d}\to\rr^d$ is the projection on the first component defined in \cref{def_pi_pos} (this condition ensures that the local stable manifold of $E_+(V_1)$ is not changed by a perturbation of $V_1$ in the direction of $W$, see the second remark at the end of \cref{subsec:local_stable_unstable_manifolds_positive_speed}). For convenience, we will also ensure that the same is true for the local unstable manifold of $E_-(V_1)$, that is:
\begin{equation}
\label{supp_W_does_not_meet_projection_Wuloc}
\supp(W)\cap\piPos\Bigl[\Wuloc{c}{V}\bigl(E_-(V_1)\bigr)\Bigr] = \emptyset
\,.
\end{equation}
Fulfilling equality \cref{transversality_Phi_condition_dPhiu} amounts to prove that the orthogonal complement of the subspace of the directions of $\rr^{2d}$ that can be reached by $D\Phiu\cdot(0,\zeta,0,W)$ is trivial, i.e. reduced to $\{0_{\rr^{2d}}\}$. Observe that 
\[
D\Phiu\cdot(0,\zeta,0,0) = \zeta\,\dot U_1(\xi_1)
\,.
\]
Thus the transversality statement \cref{transversality_of_DPhi} is a consequence of the following lemma. 
\begin{lemma}[perturbation of the potential reaching a given direction]
\label{lem:reach_all_directions_orthogonal_to_UzeroPrime_of_xiZero}
For every nonzero vector $(\phi_1,\psi_1)$ in $\rr^{2d}$ which is orthogonal to $\dot U_1(\xi_1)$, there exists $W$ in $\CkbFull$ satisfying conditions \cref{supp_W_included_in_BR,supp_W_does_not_meet_projection_Wsloc,supp_W_does_not_meet_projection_Wuloc} and the inequality
\begin{equation}
\label{reach_all_directions_orthogonal_to_UzeroPrime_of_xiZero}
\langle D\Phiu\cdot(0,0,0,W)\,|\,(\phi_1,\psi_1)\rangle \not= 0
\,.
\end{equation}
\end{lemma}
\begin{proof}[Proof of \cref{lem:reach_all_directions_orthogonal_to_UzeroPrime_of_xiZero}]
Let $(\phi_1,\psi_1)$ denote a nonzero vector orthogonal to $U_1(\xi_1)$ in $\rr^{2d}$, and let $W$ be a function in $\CkbFull$ satisfying the conditions \cref{supp_W_included_in_BR,supp_W_does_not_meet_projection_Wsloc,supp_W_does_not_meet_projection_Wuloc}. Let us consider the linearization of the differential system \cref{trav_wave_system_order_1}, for the potential $V_1$ and the speed $c_1$, around the solution $\xi\mapsto U_1(\xi)$:
\begin{equation}
\label{linearized_differential_system}
\frac{d}{d\xi}
\begin{pmatrix}
\delta u(\xi) \\ \delta v(\xi) 
\end{pmatrix}
= \begin{pmatrix}
0 & \id \\ D^2 V_1 \bigl(u_1(\xi)\bigr) & -c_1
\end{pmatrix}
\begin{pmatrix}
\delta u(\xi) \\ \delta v(\xi) 
\end{pmatrix}
\,,
\end{equation}
and let $T(\xi,\xi')$ denote the family of evolution operators obtained by integrating this linearized differential system between times $\xi$ and $\xi'$. It follows from condition \cref{supp_W_does_not_meet_projection_Wuloc} that $W$ only affects the part of $\Phiu$ corresponding to the flow (not on the function $\hatwuloc{c_1}{V_1}$) and the variation of constants formula yields that 
\begin{equation}
\label{expression_DPhi_trav_front}
D\Phiu\cdot(0,0,0,W) = \int_{0}^{\xi_1} T(\xi,\xi_1) \Bigl(0,\nabla W\bigl(u_1(\xi)\bigr)\Bigr)\, d\xi
\,.
\end{equation}
For every time $\xi$, let $T^*(\xi,\xi_1)$ denote the adjoint operator of $T(\xi,\xi_1)$, and let 
\begin{equation}
\label{def_phi_psi}
\bigl(\phi(\xi),\psi(\xi)\bigr) = T^*(\xi,\xi_1)\cdot(\phi_1,\psi_1)\,.
\end{equation}
According to expression \cref{expression_DPhi_trav_front}, inequality \cref{reach_all_directions_orthogonal_to_UzeroPrime_of_xiZero} reads
\[
\int_0^{\xi_1}\Bigl\langle\Bigl(0,\nabla W\bigl(u_1(\xi)\bigr)\Bigr) \,\Bigm|\, T^*(\xi,\xi_1)\cdot(\phi_1,\psi_1) \Bigr\rangle\, d\xi \not= 0
\,,
\]
or equivalently
\begin{equation}
\label{reach_all_directions_condition_with_psi}
\int_0^{\xi_1}\nabla W\bigl(u_1(\xi)\bigr)\cdot \psi(\xi) \, d\xi \not= 0
\,.
\end{equation}
Notice that, due to the expression of the linearized differential system \cref{linearized_differential_system}, $(\phi,\psi)$ is a solution of the adjoint linearized system
\begin{equation}
\label{adjoint_linearized_system}
\begin{pmatrix}
\dot\phi(\xi) \\ \dot\psi(\xi) 
\end{pmatrix}
= - \begin{pmatrix}
0 & D^2 V_1 \bigl(u_1(\xi)\bigr) \\
\id & -c_1
\end{pmatrix}  
\begin{pmatrix}
\phi(\xi) \\ \psi(\xi) 
\end{pmatrix}
\,.
\end{equation}

Our task is thus to construct a function $W$ in $\CkbFull$ satisfying \cref{supp_W_included_in_BR,supp_W_does_not_meet_projection_Wsloc,supp_W_does_not_meet_projection_Wuloc,reach_all_directions_condition_with_psi}. 
There are two difficulties to overcome. 
\begin{enumerate}
\item First, as shown by \cref{fig:one-to-one}, the function $\xi\mapsto u_1(\xi)$ may reach the same value for different values of the argument $\xi$, making it difficult to handle the interactions of the contributions to the integral \cref{reach_all_directions_condition_with_psi} of the perturbation $W\bigl(u_1(\xi)\bigr)$ at these different values of $\xi$. 
\label{item:difficulty_values_reached_only_once}
\item Second, the integral \cref{reach_all_directions_condition_with_psi} depends on the gradient $\nabla W$ of the perturbation $W$ and not on $W$ itself, and this gradient cannot be \emph{any} function. 
\end{enumerate}
These difficulties have already been tackled in several contexts, see \cite{Rifford_closingGeodesicsC1Topology_2012,RiffordRuggiero_genPropClosedOrbHamFlowsManeViewpoint_2012,Robbin_algebraicKupkaSmaleTheory_1981} (ordinary differential equations) and \cite{BrunovskyJolyRaugel_genericTransversalityHeteroclinicHomoclinicOrbitsScalarParabolicEquations_2019,BrunovskyPolacik_MorseSmaleStructureGenericRDEquHigherSpaceDim_1997,BrunovskyRaugel_genericityMorseSmalePropertyDampedWaveEqu_2003,JolyRaugel_genericHypebolicityEquPerOrbParabEquCircle_2010,JolyRaugel_genericMorseSmalePropertyParabEquCircle_2010} (partial differential equations). Each time, some specific arguments have to be found, using the peculiarities and constraints of the considered system. 

In the present case, the following trick will do the job. According to \cref{lem:additional_condition_nu_and_r}, there exists a closed interval with nonempty interior $\IOnce$, included in $(0,+\infty)$, such that
\begin{equation}
\label{u_1_of_Ionce_not_in_pPos_of_Wuloc_cup_Wsloc}
u_1(\IOnce)\cap \piPos\Bigl[\Wuloc{c}{V}\bigl(E_-(V_1)\bigr)\cup\Wsloc{c}{V}\bigl(E_+(V_1)\bigr)\Bigr] = \emptyset
\,,
\end{equation}
such that $\dot u$ does not vanish on $\IOnce$, and such that
\begin{equation}
\label{values_reached_only_once_trav}
\text{for all } \xi^*\text{ in }\IOnce\text{ and } \xi\text{ in }\rr\,, \quad
u_1(\xi) = u_1 (\xi^*) \implies \xi = \xi^*
\,.
\end{equation}
According to the empty intersection \cref{u_1_of_Ionce_not_in_pPos_of_Wuloc_cup_Wsloc} and since $u_1(\xi)$ is in $\Wsloc{c}{V}\bigl(E_+(V_1)\bigr)$ for $\xi$ larger than $\xi_1$, the interval $\IOnce$ is actually included in $(0,\xi_1)$. 
In view of \cref{values_reached_only_once_trav}, the image $u_1(\IOnce)$ of this interval provides a suitable place where the trajectory can be perturbed without the inconvenience \cref{item:difficulty_values_reached_only_once} emphasized above. Two cases have to be considered (plus a third one that will turn out to be empty). 
\paragraph*{Case 1.} There exists a time $\xi^*$ in $\IOnce$ such that $\psi(\xi^*)$ is \emph{not} collinear to $\dot u_1(\xi^*)$. 

In this case, up to an affine conformal change of coordinate system in $\rr^d$, we may assume that 
\begin{equation}
\label{loc_coordinate_system}
u_1(\xi^*)=0
\quad\text{and}\quad
\dot u_1(\xi^*) = \epsilon_1
\quad\text{and}\quad
\epsilon_2 \cdot \psi(\xi^*) \not= 0
\,,
\end{equation}
where $\epsilon_1= (1,0,\dots,0)$ and $\epsilon_2= (0,1,0,\dots,0)$ are the two first vectors of the canonical basis of $\rr^d$. Let $\rho$ denote an even function in $\ccc^{k+1}\bigl(\rr,[0,1]\bigr)$ satisfying 
\[
\rho(0) = 1
\quad\text{and $\rho$ vanishes on $\rr\setminus(-1,1)$.}
\]
Let $\varepsilon$ denote a small positive quantity to be chosen later and let us consider the bump function
\[
\rho_\varepsilon:\rr^d\to[0,1], 
\quad 
u \mapsto \rho\left(\frac{\abs{u}}{\varepsilon}\right)
\,.
\]
It follows from this definition that
\begin{equation}
\label{properties_bump_function}
\rho_\varepsilon(0_{\rr^d}) = 1 
\quad\text{and}\quad
\supp(\rho_\varepsilon)\subset \widebar{B}_{\rr^d}(0,\varepsilon)
\quad\text{and}\quad
\norm{\nabla\rho_\varepsilon}_{L^\infty(\rr^d,\rr)^d} \in \mathcal{O}_{\varepsilon\to0}(\varepsilon^{-1})
\,.
\end{equation}
\begin{minipage}{.55\textwidth}

Let us define the perturbation $W$ as follows: for every $u$ in $\rr^d$, 
\[
W(u) = \rho_\varepsilon(u) (\epsilon_2 \cdot u)
\,,
\]
see \cref{fig:perturbation_potential_case_1}, so that
\begin{equation}
\label{expression_nabla_W_case_one}
\nabla W(u) = \rho_\varepsilon(u) \epsilon_2 + (\epsilon_2 \cdot u) \nabla\rho_\varepsilon(u)
\,.
\end{equation}

\end{minipage}
\hspace{.03\textwidth}
\begin{minipage}{.4\textwidth}
\begin{figure}[H]
\centering
\resizebox{0.95\textwidth}{!}{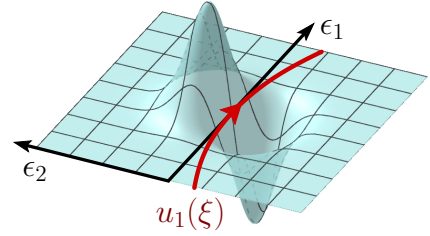}
\caption{Graph of $W$.}
\label{fig:perturbation_potential_case_1}
\end{figure}
\vspace{.001\textheight}
\end{minipage}
It follows from this definition that, if $\varepsilon$ is small enough, then, on the one hand conditions \cref{supp_W_included_in_BR} (according to inequality \cref{a_priori_bound_on_profiles} of \cref{lem:global_flow_and_a_priori_bound_on_profiles}), and \cref{supp_W_does_not_meet_projection_Wsloc,supp_W_does_not_meet_projection_Wuloc} (according to the empty intersection \cref{u_1_of_Ionce_not_in_pPos_of_Wuloc_cup_Wsloc}) are fulfilled; and, on the other hand, according to \cref{values_reached_only_once_trav} and since $\dot u_1(\xi^*)$ is nonzero, there exists an open interval $I_\varepsilon^*$ of $\rr$ satisfying
\begin{equation}
\label{open_interval_I_epsilon_star}
\xi^*\in I_\varepsilon^*
\quad\text{and, for every $\xi$ in $\rr$,}\quad
u_1(\xi)\in B_{\rr^d}(0,\varepsilon)\iff \xi\in I_\varepsilon^*
\,.
\end{equation}
Let us assume that $\varepsilon$ is chosen as such. It follows from \cref{open_interval_I_epsilon_star} that the integral in \cref{reach_all_directions_condition_with_psi} reduces to:
\begin{equation}
\label{reach_all_directions_integral_on_Istar}
\int_{I_\varepsilon^*} \nabla W\bigl(u_1(\xi)\bigr)\cdot \psi(\xi) \, d\xi
\,.
\end{equation}
As a consequence, if $u_1(\xi)$ follows a straight line in the direction of $\epsilon_1$ inside the ball $B_{\rr^d}(0,\varepsilon)$, then, for every $\xi$ in $I_\varepsilon^*$,
\[
\nabla W\bigl(u_1(\xi)\bigr) = \rho_\varepsilon\bigl(u_1(\xi)\bigr)\epsilon_2
\,,
\]
so that the integral \cref{reach_all_directions_integral_on_Istar} reduces to 
\[
\int_{I_\varepsilon^*} \rho_\varepsilon\bigl(u_1(\xi)\bigr) \epsilon_2 \cdot \psi(\xi) \, d\xi 
\,,
\]
and according to the last property of \cref{loc_coordinate_system}, if $\varepsilon$ is sufficiently small then this integral does not vanish, fulfilling inequality \cref{reach_all_directions_condition_with_psi} --- and thus also \cref{reach_all_directions_orthogonal_to_UzeroPrime_of_xiZero}.

In the general situation where $u_1(\xi)$ does not necessarily follow a straight line in the direction of $\epsilon_1$ inside the ball $B_{\rr^d}(0,\varepsilon)$, the quantity $\epsilon_2\cdot u_1(\xi)$ is in $\mathcal{O}_{\varepsilon\to0}(\varepsilon^2)$ when $\xi$ is in $I_\varepsilon^*$, thus it follows from \cref{expression_nabla_W_case_one} and from the last property of \cref{properties_bump_function} that, still for $\xi$ in $I_\varepsilon^*$,
\[
\nabla W\bigl(u_1(\xi)\bigr) = \rho_\varepsilon\bigl(u_1(\xi)\bigr) \epsilon_2 + \mathcal{O}_{\varepsilon\to0}(\varepsilon)
\,,
\]
and since $\rho_\varepsilon(0_{\rr^d})$ equals $1$, it follows from the last property of \cref{loc_coordinate_system} that, if $\varepsilon$ is sufficiently small, then inequality \cref{reach_all_directions_condition_with_psi} is fulfilled again --- thus so is inequality \cref{reach_all_directions_orthogonal_to_UzeroPrime_of_xiZero}. 

If case 1 does \emph{not} occur, then $\psi(\xi)$ is collinear to $\dot u_1(\xi)$ for every $\xi$ in $\IOnce$, and since $\dot u_1(\cdot)$ does not vanish on $\IOnce$, there exists a $\ccc^1$-function $\alpha:\IOnce \to\rr$ such that, for every $\xi$ in $\IOnce$,
\begin{equation}
\label{psi_equals_alpha_times_derivative_of_uzero}
\psi(\xi) = \alpha(\xi) \dot u_1(\xi)
\,.
\end{equation}
The next cases 2 and 3 differ according to whether the function $\alpha(\cdot)$ is constant or not. 
\paragraph*{Case 2.} For every $\xi$ in $\IOnce$, equality \cref{psi_equals_alpha_times_derivative_of_uzero} holds for some \emph{nonconstant} function $\alpha(\cdot)$. 

For every perturbation $W$ of the potential, if the support of $W$ is localized enough around some point of $u(\IOnce)$ (so that expression \cref{psi_equals_alpha_times_derivative_of_uzero} holds as soon as $\nabla W\bigl(u(\xi)\bigr)$ is nonzero), then an integration by parts shows that the integral in inequality \cref{reach_all_directions_condition_with_psi} becomes
\begin{equation}
\label{reach_all_directions_condition_with_psi_case2}
\int\nabla W\bigl(u_1(\xi)\bigr)\cdot \psi(\xi) \, d\xi 
= \int \alpha(\xi) \nabla W\bigl(u_1(\xi)\bigr)\cdot  \dot u_1(\xi) \, d\xi 
= -\int \dot\alpha(\xi) W\bigl(u_1(\xi)\bigr)\, d\xi 
\end{equation}
(with integration domain $[0,\xi_1]$ for each of these integrals). 

The expression of this last integral shows why the assumption (made in the present case 2) that $\alpha(\cdot)$ is nonconstant matters. According to this assumption, there exists $\xi^*$ in $\IOnce$ such that $\dot \alpha(\xi^*)$ is nonzero. Let us assume (up to an affine change of variable in $\rr^{2d}$) that $u_1(\xi^*)$ is equal to $0_{\rr^d}$. Let us define $B_{\rr^d}(0,\varepsilon)$ and $\rho_\varepsilon$ and $I^*_\varepsilon$ as in case 1 above, and let us simply define the perturbation $W$ as
\[
W=\rho_\varepsilon
\,.
\]
As in case 1, for $\varepsilon$ sufficiently small, conditions \cref{supp_W_included_in_BR,supp_W_does_not_meet_projection_Wsloc,supp_W_does_not_meet_projection_Wuloc} are fulfilled, and the integral in inequality \cref{reach_all_directions_condition_with_psi} reduces to the expression \cref{reach_all_directions_integral_on_Istar}. In view of \cref{reach_all_directions_condition_with_psi_case2}, \cref{reach_all_directions_condition_with_psi} thus becomes
\begin{equation}\label{eq_case_2}
\int_{I_\varepsilon^*} \dot\alpha(\xi) W\bigl(u_1(\xi)\bigr)\, d\xi \neq 0 
\,,
\end{equation}
which is fulfilled if $\varepsilon$ sufficiently small. It follows that inequality \cref{reach_all_directions_condition_with_psi} is fulfilled, and thus so is inequality \cref{reach_all_directions_orthogonal_to_UzeroPrime_of_xiZero}. 
\paragraph*{Case 3.} For every $\xi$ in $\IOnce$, $\psi(\xi)=\alpha \dot u(\xi)$, for some real (constant) quantity $\alpha$. 

In this case, expression \cref{reach_all_directions_condition_with_psi_case2} shows that inequality \cref{reach_all_directions_condition_with_psi} cannot hold if the support of $W$ is localized around some point of $u(\IOnce)$. Fortunately, this third case will lead to a contradiction (and does therefore actually not happen). Recall that $(\phi,\psi)$ is a solution of the adjoint linearized system \cref{adjoint_linearized_system}. Thus, for every $\xi$ in $\IOnce$, it follows from the assumption made in this case 3 that
\begin{equation}
\label{expression_of_phi_when_psi_equals_alpha_uzeroprime}
\phi(\xi) = c \psi(\xi) - \dot\psi(\xi) = c\alpha \dot u_1(\xi) - \alpha \ddot u_1(\xi)
\,.
\end{equation}
Besides, recall that $(\phi_1,\psi_1)$ is orthogonal to $\dot U_1(\xi_1)=T(\xi,\xi_1)\, \dot U_1(\xi)$. Thus, $\bigl(\phi(\xi),\psi(\xi)\bigr)=T^*(\xi,\xi_1)\, (\phi_1,\psi_1)$ is orthogonal to $\dot U_1(\xi)$. According to the expression of $\psi$ and  expression \cref{expression_of_phi_when_psi_equals_alpha_uzeroprime}, this last property reads 
\begin{equation}\label{dissipation_nulle}
c\alpha |\dot u_1|^2(\xi) - \alpha \ddot u_1(\xi) \cdot \dot u_1(\xi)   + \alpha \dot u_1(\xi)\cdot  \ddot u_1(\xi) = 0 
\,,\quad\text{which yields}\quad c\alpha |\dot u_1|^2(\xi) = 0
\,.
\end{equation}
Since $\dot u_1$ does not vanish on $(-\infty,\xiOnce)$, the quantity $\alpha$ must be zero. This yields $\phi\equiv \psi\equiv 0$, and contradicts the assumptions of \cref{lem:reach_all_directions_orthogonal_to_UzeroPrime_of_xiZero}.

In short, case 3 cannot happen and, in both cases 1 and 2, a suitable construction provides a function $W$ in $\CkbFull$ fulfilling the conditions \cref{supp_W_included_in_BR,supp_W_does_not_meet_projection_Wsloc,supp_W_does_not_meet_projection_Wuloc,reach_all_directions_orthogonal_to_UzeroPrime_of_xiZero}. \Cref{lem:reach_all_directions_orthogonal_to_UzeroPrime_of_xiZero} is proved.
\end{proof}
\subsubsection{Conclusion}
\label{subsubsec:conclusion_loc_gen_transv}
\begin{proof}[Proof of \cref{prop:local_generic_transversality_travelling_fronts_given_critical_points_speed}]
As seen in \cref{subsubsec:checking_hyp_regularity_Sard_Smale_trav_fronts}, hypothesis \cref{item:thm_sard_smale_condition_regularity} of \cref{thm:Sard_Smale} is fulfilled for the function $\Phi$ defined in \cref{def_Phi_trav_front}. Since the conclusion of \cref{lem:reach_all_directions_orthogonal_to_UzeroPrime_of_xiZero} yields equality \cref{transversality_of_DPhi}, hypothesis \cref{item:thm_sard_smale_condition_transversality} of this theorem is also fulfilled. The conclusion of this theorem ensures that there exists a generic subset $\LambdaGen$ of $\Lambda$ such that, for every $V$ in $\LambdaGen$, the function $\Phi(\cdot,V)$ is transverse to the diagonal $\www$ of $\nnn$. According to \cref{prop:equivalence_transversality}, it follows that, for every $V$ in $\LambdaGen$, every profile $\xi\mapsto u(\xi)$ of a front travelling at a speed $c$ in $\ccc$ and connecting $e_-(V)$ to $e_+(V)$, for the potential $V$, is transverse. In other words the conclusions of \cref{prop:local_generic_transversality_travelling_fronts_given_critical_points_speed} hold with $\ccc_{\VZeroeMinusZeroePlusZerocZero} = \ccc$, $\nu_{\VZeroeMinusZeroePlusZerocZero} = \nu = \Lambda$ and $\nuThingGen{\VZeroeMinusZeroePlusZerocZero} = \LambdaGen$.
\end{proof}
As shown in \cref{subsec:reduction}, \cref{prop:global_generic_transversality_travelling_fronts} follows from \cref{prop:local_generic_transversality_travelling_fronts_given_critical_points_speed}.

\section{Generic elementarity of symmetric standing pulses}
\label{sec:generic_elementarity_sym_stand_pulses}
This \namecref{sec:generic_elementarity_sym_stand_pulses} presents strong similarities with the previous \cref{sec:generic_transversality_travelling_fronts}. For that reason, the presentation aims at emphasizing the main differences, while some details or comments are omitted when they are identical to some already provided in \cref{sec:generic_transversality_travelling_fronts}. 
\subsection{Notation and statements}\label{sec:not_and_stat_sym}
\label{subsec:notation_and_statements_sym_stand_pulses}
\begin{notation}
For every potential function $V$ in $\vvvFull$, let us recall (\cref{subsec:notation_and_statement_gen_transv_trav_fronts}) that $\SigmaCrit(V)$ denotes the set of non-degenerate critical points of $V$, and let us consider the set
\begin{equation}
\label{notation_ppp_V}
\begin{aligned}
\ppp_V = \bigl\{ & u\in \ccc^{k+1}(\rr,\rr^d): \xi\mapsto u(\xi) \text{ is a global solution of the system } \ddot  u = \nabla V(u)\,,\\
&\text{and there exists $e$ in $\SigmaCrit(V)$ such that } u(\xi)\to e\text{ as }\xi\to\pm\infty
\bigr\}
\,.
\end{aligned}
\end{equation}
In other words, $u$ is in $\ppp_V$ if and only if $\xi\mapsto u(\xi)$ is the profile of a standing pulse connecting a non-degenerate critical point $e$ to itself, for the potential $V$. 
\end{notation}
Let us take and fix a positive quantity $R$. Let us recall that the elementarity of a symmetric standing pulse was defined in \cref{def:elementary_symm_stand_pulse}. The goal of this section is to prove the following proposition.
\begin{proposition}
\label{prop:global_generic_elementarity_sym_stand_pulses}
There exists a generic subset of $\vvvQuad{R}$ such that, for every potential $V$ in this subset, every symmetric standing pulse in $\ppp_V$ is elementary. 
\end{proposition}

Let $V_0$ denote a potential function in $\vvvQuad{R}$, and let $e_0$ denote a non-degenerate critical point of $V_0$. According to \cref{prop:loc_stab_unstab_man_c_equals_zero} (or simply to the implicit function theorem), there exists a small neighbourhood $\nuRobust(V_0,e_0)$ of $V_0$ in $\vvvQuad{R}$ and a $\ccc^{k+1}$-function $e(\cdot)$ defined on $\nuRobust(V_0,e_0)$ and with values in $\rr^d$, such that $e(V_0)$ equals $e_0$ and, for every $V$ in $\nuRobust(V_0,e_0)$, $e(V)$ is a critical point of $V_0$ close to $e_0$.

Exactly the same arguments as in \cref{subsec:reduction} show that \cref{prop:global_generic_elementarity_sym_stand_pulses} is a consequence of the following local statement.
\begin{proposition}
\label{prop:local_generic_transversality_standing_pulse}
There exists a neighbourhood $\nu_{\VZeroeZero}$ of $V_0$ in $\vvvQuad{R}$, included in $\nuRobust(V_0,e_0)$, and a generic subset $\nuThingGen{\VZeroeZero}$ of $\nu_{\VZeroeZero}$ such that, for every $V$ in $\nuThingGen{\VZeroeZero}$, every symmetric standing front connecting $e(V)$ to itself is elementary.
\end{proposition}
The remaining part of \cref{sec:generic_elementarity_sym_stand_pulses} will thus be devoted to the proof of \cref{prop:local_generic_transversality_standing_pulse}. Let us keep the notation $V_0$ and $e_0$ and $\nuRobust(V_0,e_0)$ introduced above. According to \cref{prop:loc_stab_unstab_man_c_equals_zero}, there exist a neighbourhood $\nu$ of $V_0$ in $\vvvQuad{R}$, included in $\nuRobust(V_0,e_0)$, and a positive quantity $r$ such that, for every $V$ in $\nu$, there exist $C^{k}$-functions
\[
\hatwulocZero{V}:\Bu_{E_0}(r)\to\rr^{2d}
\quad\text{and}\quad
\hatwslocZero{V}:\Bs_{E_0}(r)\to\rr^{2d}
\]
such that the sets
\[
\WulocZero{V}\bigl(E(V)\bigr) = \hatwulocZero{V}\bigl(\Bu_{E_0}(r)\bigr) 
\quad \text{and}\quad
\WslocZero{V}\bigl(E(V)\bigr) = \hatwslocZero{V}\bigl(\Bs_{E_0}(r)\bigr)
\]
define a local unstable manifold and a local stable manifold of $E(V)$, respectively (see the conclusions of \cref{prop:loc_stab_unstab_man_c_equals_zero} and equalities \cref{def_local_stab_unstab_man_as_images_zero_speed}). Observe that the departure sets $\Bu_{E_0}(r)$ of $\hatwulocZero{V}$ and $\Bs_{E_0}(r)$ of $\hatwslocZero{V}$ do not depend on $V$. Let 
\[
\bbbu = \partial\Bu_{E_0}(r)
\quad\text{and}\quad
\bbbs = \partial\Bs_{E_0}(r)
\,.
\]
According to the expression \cref{notation_eigenvectors_in_dim_2d} of the eigenvectors of the linear system \cref{trav_wave_system_order_1_linearized}, 
\[
\Eu_{V_0}(E_0) \cap \sssSym = \{0_{\rr^{2d}}\}
\quad\text{and}\quad
\Es_{V_0}(E_0) \cap \sssSym = \{0_{\rr^{2d}}\}
\,.
\]
It follows that, up to replacing $\nu$ by a smaller neighbourhood of $V_0$ in $\vvvQuad{R}$ and $r$ by a smaller positive quantity, for every $V$ in $\nu$, 
\begin{equation}
\label{stable_unstable_man_transverse_reversibility_subspace}
\WulocZero{V}\bigl(E(V)\bigr)\cap\sssSym = \{E(V)\}
\quad\text{and}\quad
\WslocZero{V}\bigl(E(V)\bigr)\cap\sssSym = \{E(V)\}
\,.
\end{equation}
\subsection{Proof of \texorpdfstring{\cref{prop:local_generic_transversality_standing_pulse}}{Proposition \ref{prop:local_generic_transversality_standing_pulse}}}
\label{subsec:proof_local_gen_elementary_sym_standing_pulses}
\subsubsection{Application of \texorpdfstring{\cref{thm:Sard_Smale}}{Theorem \ref{thm:Sard_Smale}}}
The setting to which \cref{thm:Sard_Smale} will be applied is as follows. 
Let 
\[
\mmm = \bbbu\times\rr \,,\quad
\Lambda = \nu \,,\quad
\nnn = \rr^{2d}
\quad\text{and}\quad
\www  = \sssSym
\,,
\]
and let us consider the function
\begin{equation}
\label{def_Phi_sym_stand_pulse}
\Phi:\mmm\times\Lambda\to\nnn, 
\quad
(\bu,\xi,V)\mapsto S_V\bigl(\xi,\hatwulocZero{V}(\bu)\bigr)
\,.
\end{equation}
If the conclusion of \cref{thm:Sard_Smale} holds within this setting, then there exists a generic subset $\LambdaGen$ of $\Lambda$ such that, for every $V$ in $\LambdaGen$, the image of the function $\mmm\to\nnn$, $m\mapsto\Phi(m,V)$ is transverse to $\www$.

For a given potential $V$, the image of $m\mapsto\Phi(m,V)$ is nothing but the unstable manifold of $E(V)$ (deprived of $E$), see the proof of \cref{prop:one_to_one_correspondence_between_fronts_and_diagonal_intersection}. According to the characterizations of the symmetric standing pulses stated in \cref{lem:equivalent_def_symm_hom_orbit}, the intersection of $\Phi(\mmm,V)$ with $\www=\sssSym$ actually corresponds to the set of symmetric standing pulses. Moreover, by definition (see \cref{def:elementary_symm_stand_pulse}), the elementarity of the symmetric standing pulses for $V$ is equivalent to the transversality of the intersection of $\Phi(\mmm,V)$ with $\www=\sssSym$. Thus, the conclusion of \cref{thm:Sard_Smale} directly implies \cref{prop:local_generic_transversality_standing_pulse} with $\nu_{\VZeroeZero} = \nu = \Lambda$ and $\nuThingGen{\VZeroeZero} = \LambdaGen$.

It remains to show that, in the setting above, the hypotheses of \cref{thm:Sard_Smale} are fulfilled.
\subsubsection{Checking hypothesis \texorpdfstring{\cref{item:thm_sard_smale_condition_regularity}}{\ref{item:thm_sard_smale_condition_regularity}} of \texorpdfstring{\cref{thm:Sard_Smale}}{Theorem \ref{thm:Sard_Smale}}}
\label{subsubsec:check_hyp_regularity_sym_stand_pulse}
It follows from \cref{subsec:eigenspaces_and_dimensions} that $\dim(\bbbu) = d-m(e_0)-1$. Hence, 
\[
\dim(\mmm)-\codim(\www) = \bigl(d-m(e_0)\bigr) - d = -m(e_0)
\,,
\]
which is less than the positive integer $k$ (the regularity of $\Phi$). Hypothesis \cref{item:thm_sard_smale_condition_regularity} of \cref{thm:Sard_Smale} is thus fulfilled. 
\subsubsection{Checking hypothesis \texorpdfstring{\cref{item:thm_sard_smale_condition_transversality}}{conditionTransversalitySardSmale} of \texorpdfstring{\cref{thm:Sard_Smale}}{Theorem \ref{thm:Sard_Smale}}}
\label{subsubsec:checking_hyp_ii_Sard_Smale_thm_sym_stand_pulse}
Take $(m_1,V_1)$ in the set $\Phi^{-1}(\www)$. Let $(\bu_1,\xiTurn)$ denote the components of $m_1$, and, for every real quantity $\xi$, let us write
\[
U_1(\xi) = \bigl(u_1(\xi),v_1(\xi)\bigr) = S_{V_1}\bigl(\xi,\hatwulocZero{V_1}(\bu_1)\bigr)
\,.
\]
The function $\xi\mapsto u_1(\xi)$ is the profile of a symmetric standing pulse, connecting $e(V_1)$ to itself for the potential $V_1$, and the quantity $\xiTurn$ is the turning time of this standing pulse (see \cref{def:symmetric_pulse}). Observe that, according to the first equality of \cref{stable_unstable_man_transverse_reversibility_subspace}, this turning time $\xiTurn$ must be \emph{positive}. Let $D\Phi$ denote the \emph{full} differential (with respect to $m$ and $V$) of $\Phi$ at the point $\bigl((\bu_1,\xiTurn),V_1\bigr)$. 
Hypothesis \cref{item:thm_sard_smale_condition_transversality} of \cref{thm:Sard_Smale} follows from the next \cref{lem:reach_all_directions_symm_stand_pulse}. 
\begin{lemma}[perturbation of the potential reaching a given direction]
\label{lem:reach_all_directions_symm_stand_pulse}
For every nonzero vector $\psi_1$ in $\rr^d$, there exists $W$ in $\CkbFull$ such that 
\begin{equation}
\label{reach_all_directions_symm_stand_pulse}
\langle D\Phi\cdot(0,W)\,|\,(0,\psi_1)\rangle \not = 0
\,,
\end{equation}
and 
\begin{equation}
\label{W_admissible_perturbation_sym_stand_pulses}
\supp(W)\subset B_{\rr^d}(0,R) \,.
\end{equation}
\end{lemma}
\begin{proof}
The proof is similar to that of \cref{lem:reach_all_directions_orthogonal_to_UzeroPrime_of_xiZero}. Let $\psi_1$ be a nonzero vector in $\rr^d$. , and let $W$ denote a function in $\CkbFull$ with a support satisfying the condition
\begin{equation}
\label{support_of_W_cap_Wuloc_of_V1_sym_stand_pulses_is_empty}
\supp(W)\cap\piPos\Bigl(\WulocZero{V}\bigl(E(V_1)\bigr)\Bigr) = \emptyset
\,.
\end{equation}
Let us again use the notation $T(\xi,\xi')$ to denote the family of evolution operators obtained by integrating the linearized differential system \cref{linearized_differential_system} (for $c_1$ equal to $0$) between the times $\xi$ and $\xi'$. It follows from the empty intersection \cref{support_of_W_cap_Wuloc_of_V1_sym_stand_pulses_is_empty} that
\begin{equation}
\label{expression_DPhi_sym_stand_pulse}
D\Phi\cdot(0,W) = \int_{0}^{\xiTurn} T(\xi,\xiTurn) \Bigl(0,\nabla W\bigl(u_1(\xi)\bigr)\Bigr)\, d\xi
\,.
\end{equation}
For every time $\xi$, let $T^*(\xi,\xiTurn)$ denote the adjoint operator of $T(\xi,\xiTurn)$, and let 
\[
\bigl(\phi(\xi),\psi(\xi)\bigr) = T^*(\xi,\xiTurn)\cdot(0,\psi_1)
\,.
\]
According to \cref{expression_DPhi_sym_stand_pulse}, condition \cref{reach_all_directions_symm_stand_pulse} reads
\[
\int_{0}^{\xiTurn}\Bigl\langle\Bigl(0,\nabla W\bigl(u_1(\xi)\bigr)\Bigr) \,\Bigm|\, T^*(\xi,\xiTurn)\cdot(0,\psi_1) \Bigr\rangle\, d\xi \not= 0
\,,
\]
or equivalently
\begin{equation}
\label{reach_all_directions_condition_with_psi_symm_stand_pulse}
\int_{0}^{\xiTurn}\nabla W\bigl(u_1(\xi)\bigr)\cdot \psi(\xi) \, d\xi \not= 0
\,.
\end{equation}
According to the first equality of \cref{stable_unstable_man_transverse_reversibility_subspace} and due to the Hamiltonian invariance \cref{time_derivative_Hamiltonian}, for every $(u,v)$ in $\WulocZero{V_1}\bigl(E(V_1)\bigr)$ and differing from $E(V_1)$, the quantity $V_1(u)$ is larger than $V_1\bigl(e(V_1)\bigr)$. On the other hand, since $\dot u_1(\xiTurn)$ vanishes the quantity $V_1\bigl(u_1(\xiTurn)\bigr)$ is equal to $V_1\bigl(e(V_1)\bigr)$, so that $u_1(\xiTurn)$ does not belong to the (closed) set $\piPos\Bigl(\WulocZero{V_1}\bigl(E(V_1)\bigr)\Bigr)$. As a consequence, there exists a time $\xi_-$, smaller than (and sufficiently close to) $\xiTurn$, such that 
\begin{equation}
\label{uOne_away_from_loc_unst_man_close_to_xiTurnOne}
u_1\bigl([\xi_-,\xiTurn]\bigr)\cap\piPos\Bigl(\WulocZero{V_1}\bigl(E(V_1)\bigr)\Bigr) = \emptyset
\,.
\end{equation}
Observe that, according to \cref{lem:equivalent_def_symm_hom_orbit}, the function $\xi\mapsto \dot u_1(\xi)$ does not vanish on $(-\infty,\xiTurn)$. As in \cref{subsubsec:checking_hyp_ii_Sard_Smale_thm}, three cases have to be considered for the construction of the perturbation $W$. 

\paragraph*{Case 1.} There exists a time $\xi^\dag$ in $(\xi_-,\xiTurn)$ such that $\psi(\xi^\dag)$ is not collinear to $\dot u_1(\xi^\dag)$.

In this case, conclusion \cref{item_prop:values_reached_only_once_symmetric_standing_pulse} of \cref{prop:values_reached_only_once} provides an open interval $\IOnce$ included in $(\xi_-,\xiTurn)$ and small enough so that, for every $\xi^*$ in $\IOnce$,
\begin{itemize}
\item the vector $\psi(\xi^*)$ is not collinear to $\dot u_1(\xi^*)$,
\item and for every $\xi$ in $(-\infty,\xiTurn)$, if $u_1(\xi)$ equals $u_1(\xi^*)$ then $\xi$ equals $\xi^*$.
\end{itemize}
The same construction as in case 1 of the proof of \cref{lem:reach_all_directions_orthogonal_to_UzeroPrime_of_xiZero} can then be carried out.  It leads to a perturbation $W$ such that $\supp(W)$ is localized around a point of $u(\IOnce)$ (so that, according to inequality \cref{a_priori_bound_on_profiles}, inclusion \cref{W_admissible_perturbation_sym_stand_pulses} holds and according to the empty intersection \cref{uOne_away_from_loc_unst_man_close_to_xiTurnOne} the empty intersection \cref{support_of_W_cap_Wuloc_of_V1_sym_stand_pulses_is_empty} holds) and such that inequality \cref{reach_all_directions_condition_with_psi_symm_stand_pulse} holds --- thus so does inequality \cref{reach_all_directions_symm_stand_pulse}.

\paragraph*{Case 2.} For every $\xi$ in $(\xi_-,\xiTurn)$, $\psi(\xi)=\alpha(\xi) \dot u_1(\xi)$ with $\alpha(\cdot)$ not constant.

Again, conclusion \cref{item_prop:values_reached_only_once_symmetric_standing_pulse} of \cref{prop:values_reached_only_once} provides an open interval $\IOnce$ included in $(\xi_-,\xiTurn)$, small enough so that, for every $\xi^*$ in $\IOnce$,
\begin{itemize}
\item $\psi(\xi^*)=\alpha(\xi^*)\dot u_1(\xi^*)$,
\item and $\dot\alpha(\xi^*)\neq 0$,
\item and for every $\xi$ in $(-\infty,\xiTurn)$, if $u_1(\xi)$ equals $u_1(\xi^*)$ then $\xi$ equals $\xi^*$.
\end{itemize}
The same construction as in case 2 of the proof of \cref{lem:reach_all_directions_orthogonal_to_UzeroPrime_of_xiZero} can then be carried out.

\paragraph*{Case 3.} For every $\xi$ in $(\xi_-,\xiTurn)$, $\psi(\xi)=\alpha \dot u_1(\xi)$, for some real (constant) quantity $\alpha$. 

In case 3 of the proof of \cref{lem:reach_all_directions_orthogonal_to_UzeroPrime_of_xiZero}, the non-nullity of $c$ was mandatory to take advantage of \cref{dissipation_nulle}. Thus, a new ad hoc argument is now required to preclude the possibility of the present case 3. Here it is: since $\dot u_1(\xiTurn)=0$, it follows from the assumption made in this case that $\psi(\xi)$ goes to $0$ as $\xi$ goes to $\xiTurn$, so that $\psi_1$ vanishes, contradicting the assumptions of \cref{lem:reach_all_directions_orthogonal_to_UzeroPrime_of_xiZero}.

In short, case 3 cannot occur and in both other cases, a suitable perturbation $W$ of the potential can be constructed by following the constructions introduced in the proof of \cref{lem:reach_all_directions_orthogonal_to_UzeroPrime_of_xiZero}. \Cref{lem:reach_all_directions_symm_stand_pulse} is proved.
\end{proof}

\section{Generic transversality of asymmetric standing pulses}
\label{sec:generic_tranversality_asym_stand_pulses}
As in the previous section, the proofs of this \namecref{sec:generic_tranversality_asym_stand_pulses} present strong similarities with the ones which have been already detailed and the presentation will only emphasize the main differences.
\subsection{Notation and statements}
\label{subsec:notation_and_statements_asym_stand_pulses}
The same notation as in the previous \cref{sec:generic_elementarity_sym_stand_pulses} will be used all along the present \cref{sec:generic_tranversality_asym_stand_pulses}. 
Let us take and fix a positive quantity $R$. The goal of this section is to prove the following proposition (the transversality of a standing pulse was defined in \cref{def:transverse_stand_pulse}).
\begin{proposition}
\label{prop:global_generic_transversality_asym_stand_pulses}
There exists a generic subset of $\vvvQuad{R}$ such that, for every potential $V$ in this subset, every asymmetric standing pulse in $\ppp_V$ is transverse. 
\end{proposition}
Let $V_0$ denote a potential function in $\vvvQuad{R}$, and let $e_0$ denote a non-degenerate critical point of $V_0$. As already stated in \cref{subsec:notation_and_statements_sym_stand_pulses}, there exists a small neighbourhood $\nuRobust(V_0,e_0)$ of $V_0$ in $\vvvQuad{R}$ and a $\ccc^{k+1}$-function $e(\cdot)$ defined on $\nuRobust(V_0,e_0)$ and with values in $\rr^d$, such that $e(V_0)$ equals $e_0$ and, for every $V$ in $\nuRobust(V_0,e_0)$, $e(V)$ is a critical point of $V_0$ close to $e_0$.

Exactly the same arguments as in \cref{subsec:reduction} show that \cref{prop:global_generic_transversality_asym_stand_pulses} is a consequence of the following local statement.
\begin{proposition}
\label{prop:local_generic_transversality_standing_pulse2}
There exists a neighbourhood $\nu_{\VZeroeZero}$ of $V_0$ in $\vvvQuad{R}$, included in $\nuRobust(V_0,e_0)$, and a generic subset $\nuThingGen{\VZeroeZero}$ of $\nu_{\VZeroeZero}$ such that, for every $V$ in $\nuThingGen{\VZeroeZero}$, every asymmetric standing front connecting $e(V)$ to itself is transverse.
\end{proposition}
The remaining part of \cref{sec:generic_tranversality_asym_stand_pulses} will thus be devoted to the proof of \cref{prop:local_generic_transversality_standing_pulse2}.
Let us consider the same setting as in \cref{sec:not_and_stat_sym} for local stable and unstable manifolds of $E(V)$, for $V$ in a small enough neighbourhood $\nu$ of $V_0$. In particular, let us assume that local stable and unstable manifolds are small enough so that equalities \cref{stable_unstable_man_transverse_reversibility_subspace} hold. In addition, according to the expression \cref{notation_eigenvectors_in_dim_2d} of the eigenvectors of the linear system \cref{trav_wave_system_order_1_linearized}, 
\[
\Eu_{V_0}(E_0) \cap \bigl(\{0_{\rr^d}\times\rr^d\}\bigr) = \{0_{\rr^{2d}}\}
\quad\text{and}\quad
\Es_{V_0}(E_0) \cap \bigl(\{0_{\rr^d}\times\rr^d\}\bigr) = \{0_{\rr^{2d}}\}
\,.
\]
It follows that there exists a positive quantity $\rExit$ such that, for every $U$ in $\Wu_{V_0}(E_0)$ differing from $E_0$, 
\[
\sup_{\xi\in\rr}\abs{\piPos\bigl(S_{V_0}(\xi,U)\bigr)-e_0} > \rExit
\,;
\]
in other words, if a solution $\xi\mapsto U(\xi) = \bigl(u(\xi),\dot u(\xi)\bigr)$ (for the potential $V_0$) is homoclinic to $E_0$ then $u(\xi)$ must leave the ball $\widebar B_{\rr^d}(e_0,\rExit)$ before eventually returning into it. Up to replacing $\nu$ by a smaller neighbourhood of $V_0$ in $\vvvQuad{R}$ and $\rExit$ by a smaller positive quantity, we may assume that, for every $V$ in $\nu$ and for every $U$ in $\Wu_{V}\bigl(E(V)\bigr)$ differing from $E(V)$,
\begin{equation}
\label{exit_from_e_on_unstable_manifold}
\sup_{\xi\in\rr}\abs{\piPos\bigl(S_V(\xi,U)\bigr)-e(V)} > \rExit
\,.
\end{equation}
Finally, up to replacing $\nu$ by a smaller neighbourhood of $V_0$ in $\vvvQuad{R}$ and $r$ by a smaller positive quantity, we may assume that, for every $V$ in $\nu$, 
\begin{equation}
\label{local_unstable_stable_manifold_included_in_exit_ball}
\piPos\Bigl(\WulocZero{V}\bigl(E(V)\bigr)\cup\WslocZero{V}\bigl(E(V)\bigr)\Bigr) \subset B_{\rr^d}\bigl(e(V),\rExit/4\bigr)
\,.
\end{equation}
\subsection{Asymmetric standing pulses of bounded length and away from \texorpdfstring{$\sssSym$}{Ssym}}
\label{subsec:away}
By comparison with symmetric standing pulses considered in \cref{sec:generic_elementarity_sym_stand_pulses}, dealing with asymmetric standing pulses is less straightforward for the following reasons.
\begin{enumerate}
\item Symmetric and asymmetric standing pulses connecting a given critical point to itself may coexist for some potentials, and while symmetric standing pulses will be proved to be generically \emph{elementary} (\cref{def:elementary_symm_stand_pulse}), only asymmetric standing pulses will be proved to be generically transverse, see \cref{subsec:failure_generic_transversality_symm_stand_pulses}). As a consequence, applying \cref{thm:Sard_Smale} to prove the generic transversality of asymmetric standing pulses requires to exclude, by a way or another, symmetric ones. 
\item The transversality of a standing pulse stated in \cref{def:transverse_stand_pulse} is a transversality inside the submanifold corresponding to the level set of the Hamiltonian for the energy $-V(e)$. This submanifold depends on $V$ and a direct application of \cref{thm:Sard_Smale} is not possible because its transversality is stated inside a fixed manifold $\nnn$. A simple solution to skip this dependence is to fix $V$ close to $e_0$, but with the consequence that the considered set of potentials $V$ will not be open, so that applying \cref{thm:Sard_Smale} in this framework will provide local density but not local genericity of the potentials for which asymmetric pulses are transverse. Local genericity will actually be obtained through a countable intersection of open and dense sets, with separate proofs for their openness and their density. 
\end{enumerate} 

For every $V$ in $\nu$ and for every non negative quantity $\bar{\xi}$, let us consider the set
\begin{equation}
\label{def_local_unst_man_extended_up_to_time_xi}
\begin{aligned}
\Wu_V\bigl(E(V),\bar{\xi}\bigr) &= S_V\Bigl(\bar{\xi},\WulocZero{V}\bigl(E(V)\bigr)\Bigr) = \bigcup_{U\in\WulocZero{V}\bigl(E(V)\bigr)} S_V(\bar{\xi},U) \\
&= \{E(V)\}\cup \bigcup_{\bu\in\bbbu,\,\xi\in(-\infty,\bar{\xi}]}S_V\bigl(\xi,\hatwulocZero{V}(\bu)\bigr)
\,.
\end{aligned}
\end{equation}
According to this notation, the set $\Wu_V\bigl(E(V),0\bigr)$ reduces to $\WulocZero{V}\bigl(E(V)\bigr)$ and the set $\Wu_V\bigl(E(V),\bar{\xi}\bigr)$ increases (for inclusion) with $\bar{\xi}$ and represents (in some sense) the unstable manifold of the equilibrium $E(V)$ ``until time $\bar{\xi}$''. For all positive quantities $\bar{\xi}$ and $\varepsilon$, let us consider the set 
\begin{align}
&\nuTransvAsymmStandPulses(\bar{\xi},\varepsilon) =  \Bigl\{ V\in\nu : \text{if $U_0\in\Wu_V\bigl(E(V),\bar{\xi}\bigr)\cap\partial\WslocZero{V}\bigl(E(V)\bigr)$ and if} \nonumber\\
&\qquad\dist\biggl(S_V(\rr,U_0)\setminus \Bigl[\WulocZero{V}\bigl(E(V)\bigr)\cup\WslocZero{V}\bigl(E(V)\bigr)\Bigr],\sssSym\biggr)\ge\varepsilon \,, \text{ then the}
\label{condition_staying_at_min_dist_of_sssSym} \\
&\qquad\text{corresponding standing pulse: $\rr\to\rr^d$, $\xi\mapsto \piPos\bigl(S_V(\xi,U_0)\bigr)$ is transverse} \Bigr\}
\,. 
\nonumber
\end{align}
In other words, a potential function $V$ belonging to $\nu$ is in $\nuTransvAsymmStandPulses(\bar{\xi},\varepsilon)$ if every standing pulse connecting $\WulocZero{V}\bigl(E(V)\bigr)$ to $\WslocZero{V}\bigl(E(V)\bigr)$ in a time not larger than $\bar{\xi}$ while remaining at a distance not smaller than $\varepsilon$ from $\sssSym$, is transverse. Observe that, according to equalities \cref{stable_unstable_man_transverse_reversibility_subspace}, such a standing pulse is necessarily \emph{asymmetric}. \Cref{prop:local_generic_transversality_standing_pulse2} follows from the next proposition. 
\begin{proposition}
\label{prop:transv_asym_stand_pulses_bded_length_away_symmetry_subspace}
For all positive quantities $\bar{\xi}$ and $\varepsilon$, the set $\nuTransvAsymmStandPulses(\bar{\xi},\varepsilon)$ is open and dense in $\nu$.
\end{proposition}
\begin{proof}[Proof that \cref{prop:transv_asym_stand_pulses_bded_length_away_symmetry_subspace} yields  \cref{prop:local_generic_transversality_standing_pulse2}]
It follows from \cref{prop:transv_asym_stand_pulses_bded_length_away_symmetry_subspace} that the set
\begin{equation}
\label{intersection_transverse_asym_stand_pulses}
\bigcap_{N\in\nn}\nuTransvAsymmStandPulses(N,1/N)
\end{equation}
is a generic subset of $\nu$. And, according to the definition of $\nuTransvAsymmStandPulses(\cdot,\cdot)$, for every potential $V$ in this set, every asymmetric standing pulse connecting $e(V)$ to itself is transverse.
\end{proof}%
The remaining of this section is devoted to the proof of \cref{prop:transv_asym_stand_pulses_bded_length_away_symmetry_subspace}. 
\subsection{Openness of \texorpdfstring{$\nuTransvAsymmStandPulses(\bar{\xi},\varepsilon)$}{nuTransvAsymmStandPulses(xi,epsilon)}}
\label{subsec:openness_neighb_transv_asym_stand_pulses}
For every potential $V$ in $\nu$ and for all positive quantities $\bar{\xi}$ and $\varepsilon$, the manifolds $\Wu_V\bigl(E(V),\bar{\xi}\bigr)$ and $\WslocZero{V}\bigl(E(V)\bigr)$ are compact, and those manifolds depend smoothly on $V$. Let $(V_n)_{n\in\nn}$ denotes a sequence of potentials belonging to $\nu\setminus\nuTransvAsymmStandPulses(\bar{\xi},\varepsilon)$ and converging to some potential $V_\infty$ of $\nu$, and let us prove that, in this case, $V_\infty$ is still outside of $\nuTransvAsymmStandPulses(\bar{\xi},\varepsilon)$ (this will prove that $\nuTransvAsymmStandPulses(\bar{\xi},\varepsilon)$ is open in $\nu$). For every integer $n$, there exists a non-transverse standing pulse connecting $\WulocZero{V_n}\bigl(E(V_n)\bigr)$ to $\WslocZero{V_n}\bigl(E(V_n)\bigr)$ in a time not larger than $\bar{\xi}$ while remaining at a distance not smaller than $\varepsilon$ from $\sssSym$. As emphasized in \cref{def_local_unst_man_extended_up_to_time_xi}, this pulse is characterized by a (unique) $\bu_n$ in $\bbbu$ such that its trajectory in $\rr^{2d}$ crosses the boundary of $\WulocZero{V_n}\bigl(E(V_n)\bigr)$ at the point $\hatwulocZero{V_n}(\bu_n)$, and a (unique) time $\xi_n$ in the interval $[0,\bar{\xi}]$ such that this trajectory crosses the boundary of $\WslocZero{V_n}\bigl(E(V_n)\bigr)$ at the point $S_{V_n}\bigl(\xi_n,\hatwulocZero{V_n}(\bu_n)\bigr)$. Then, 
\begin{enumerate}[(i)]
\item by compactness (up to considering a subsequence of $(V_n)_{n\in\nn}$), we may assume that $(\bu_n,\xi_n)$ converges to some couple $(\bu_\infty,\xi_\infty)$ of $\bbbu\times[0,\bar{\xi}]$, which in turn characterizes a standing pulse for $V_\infty$.  Notice here the importance of considering homoclinic orbits of bounded ``length'', otherwise the limit trajectory would not necessarily be homoclinic to $E(V_\infty)$. 
\item Moreover, both conditions in \cref{condition_staying_at_min_dist_of_sssSym} are closed conditions, so that the limit standing pulse also satisfies them.
\item Thanks to the ``margin'' $\varepsilon$ with respect to the symmetry subspace $\sssSym$, the limit standing pulse is necessarily asymmetric.
\item Last, the limit standing pulse is non-transverse since this property is closed.
\end{enumerate}
The limit potential $V_\infty$ is thus not in $\nuTransvAsymmStandPulses(\bar{\xi},\varepsilon)$, and this completes the proof that $\nuTransvAsymmStandPulses(\bar{\xi},\varepsilon)$ is open in $\nu$. 
\subsection{Density of \texorpdfstring{$\nuTransvAsymmStandPulses(\bar{\xi},\varepsilon)$}{nuTransvAsymmStandPulses(xi,epsilon)}}
\label{subsec:density_asym_stand_pulse}
\subsubsection{Application of \texorpdfstring{\cref{thm:Sard_Smale}}{Theorem \ref{thm:Sard_Smale}}}
\label{subsubsec:applying_Sard_Smale_asymm_stand_pulses_notation}
The proof of the density assertion of \cref{prop:transv_asym_stand_pulses_bded_length_away_symmetry_subspace} will again follow from applying \cref{thm:Sard_Smale} to the following appropriate setting.

Take positive quantities $\bar{\xi}$ and $\varepsilon$, and a potential $V_1$ in $\nu$. Our goal is to prove that there exist potentials in $\nuTransvAsymmStandPulses(\bar{\xi},\varepsilon)$ which are arbitrarily close to $V_1$. Let
\begin{equation}
\label{def_mmm_asym_pulse}
\begin{aligned}
\mmm &= \Bigl\{ (\bu,\xi)\in \bbbu\times(0,\bar{\xi}+1)  : \dist\Bigl(S_{V_1}\bigl([0,\xi],\hatwulocZero{V_1}(\bu)\bigr),\sssSym\Bigr) > \varepsilon/2 \\
&\text{and}\quad \piPos\Bigl(S_{V_1}\bigl(\xi,\hatwulocZero{V_1}(\bu)\bigr)\Bigr)\in B_{\rr^d}\bigl(e(V_1),\rExit/2\bigr)
\Bigr\}
\,.
\end{aligned}
\end{equation}
and let $\Lambda_1$ denote a neighbourhood of $V_1$ in the set
\begin{equation}
\label{V_equiv_Vone_around_e}
\left\{
V\in\nu : V\equiv V_1 \text{ on the closed ball } \widebar{B}_{\rr^d}\bigl(e(V_1),\rExit\bigr)
\right\}
\,.
\end{equation}
We may assume that this neighbourhood $\Lambda_1$ is small enough so that, for every $V$ in $\Lambda_1$ and $(\bu,\xi)$ in $\bbbu\times(0,\bar{\xi}+1)$, the following two conclusions hold:
\begin{enumerate}
\item if $(\bu,\xi)$ \emph{is not} in $\mmm$, then
\begin{equation}
\label{consequence_if_bu_xi_not_in_mmm}
\begin{aligned}
\text{either}\quad&
\dist\Bigl(S_V\bigl([0,\xi],\hatwulocZero{V}(\bu)\bigr),\sssSym\Bigr) < \varepsilon \\
\text{or}\quad&
\piPos\Bigl(S_V\bigl(\xi,\hatwulocZero{V}(\bu)\bigr)\Bigr)\not\in B_{\rr^d}\bigl(e(V_1),\rExit/4\bigr)
\,;
\end{aligned}
\end{equation}
\item if $(\bu,\xi)$ \emph{is} in $\mmm$, then
\begin{equation}
\label{consequence_if_bu_xi_in_mmm_asymm}
\begin{aligned}
&\dist\Bigl(S_V\bigl([0,\xi],\hatwulocZero{V}(\bu)\bigr),\sssSym\Bigr) >0 \,, \\ 
\text{and}\quad
&\piPos\Bigl(S_V\bigl(\xi,\hatwulocZero{V}(\bu)\bigr)\Bigr)\in B_{\rr^d}\bigl(e(V_1),\rExit\bigr)
\,.
\end{aligned}
\end{equation}
\end{enumerate}
For $V$ in $\Lambda_1$, let 
\begin{equation}
\label{def_nnn_and_www_asymmetric_standing_pulses}
\begin{aligned}
\nnn &= H_{V}^{-1}\Bigl(H_{V}\bigl(E(V)\bigr)\Bigr)\cap \bigl(B_{\rr^d}\bigl(e(V),\rExit\bigr)\times\rr^d\bigr) \setminus\{E(V)\}\\
\text{and}\quad
\www &= \partial\WslocZero{V}\bigl(E(V)\bigr) = \hatwslocZero{V}(\bbbs)
\,.
\end{aligned}
\end{equation}
Observe that $\mmm$, $\nnn$, and $\www$ are submanifolds of $\rr^{2d}$ and since $\Lambda_1$ is included in $\nu$, it follows from inclusion \cref{local_unstable_stable_manifold_included_in_exit_ball} that $\www$ is included in $\nnn$. 
In addition, according to the condition \cref{V_equiv_Vone_around_e} on $V$ and to the inclusion \cref{local_unstable_stable_manifold_included_in_exit_ball}, $\mmm$, $\nnn$ and $\www$ do actually \emph{not} depend on the potential $V$ in $\Lambda_1$. As already explained in the second remark of the beginning of \cref{subsec:away}, this is mandatory to provide a setting where \cref{thm:Sard_Smale} applies. It follows that, according to \cref{consequence_if_bu_xi_in_mmm_asymm}, we may consider the function
\begin{equation}
\label{def_Phi_asym_stand_pulse}
\Phi : \mmm\times\Lambda_1 \to \nnn
\,,\quad
(\bu,\xi,V)\mapsto S_V\bigl(\xi,\hatwulocZero{V}(\bu)\bigr)
\,,
\end{equation}
which is well defined. Notice that, even if $\mmm$ contains only couples $(\bu,\xi)$ for which, for $V$ in $\Lambda_1$, the position $u(\xi) = \piPos\Bigl(S_V\bigl(\xi,\hatwulocZero{V}(\bu)\bigr)$ of the corresponding solution is inside $B_{\rr^d}\bigl(e(V_1),\rExit\bigr)$ (second condition of \cref{consequence_if_bu_xi_in_mmm_asymm}), it follows from the property \cref{exit_from_e_on_unstable_manifold} defining $\rExit$ that this position $u(\cdot)$ exits $B_{\rr^d}\bigl(e(V_1),\rExit/2\bigr)$ at other times, and this will provide a suitable place to perturb the potential. In other words, it will be possible to modify $\Phi(\bu,\xi,V)$ by perturbing $V$ outside of $B_{\rr^d}\bigl(e(V_1),\rExit\bigr)$, even if the arrival set of $\Phi$ and its image are restricted to this ball. 
\begin{proposition}
\label{prop:equivalence_transversality_asym_stand_pulse}
For every potential function $V$ in $\Lambda_1$, if the image of the function $\mmm\to\nnn$, $V\mapsto \Phi(m,V)$ is transverse to $\www$, then $V$ belongs to the set $\nuTransvAsymmStandPulses(\bar{\xi},\varepsilon)$. 
\end{proposition}
\begin{proof}
Let us consider $V$ in $\Lambda_1$ and $U_0$ in $\Wu_V\bigl(E(V),\bar{\xi}\bigr)\cap\partial\WslocZero{V}\bigl(E(V)\bigr)$ satisfying inequality \cref{condition_staying_at_min_dist_of_sssSym}. According to the definition \cref{def_local_unst_man_extended_up_to_time_xi} of $\Wu_V\bigl(E(V),\bar{\xi}\bigr)$, the point $U_0$ is of the form $(u,\dot u)(\xi)$ with $u$ a standing pulse such that $(u,\dot u)(0)=\hatwulocZero{V}(\bu)$ and $\xi$ in $[0,\bar{\xi}]$. According to the inclusion \cref{local_unstable_stable_manifold_included_in_exit_ball} satisfied by the local manifolds and the definition of $\nuTransvAsymmStandPulses(\bar{\xi},\varepsilon)$, the implication \cref{consequence_if_bu_xi_not_in_mmm} shows that $(\bu,\xi)$ belongs to $\mmm$. Thus, the image $\Phi\bigl((\bu,\xi),V\bigr)$ is well defined, and it remains to notice that the transversality of $\Phi$ with $\www$ exactly corresponds to the definition \cref{def:transverse_stand_pulse} of the transversality of a standing pulse. It thus follows from the definition of the set $\nuTransvAsymmStandPulses(\bar{\xi},\varepsilon)$ that $V$ belongs to this set. 
\end{proof}
The remaining part of the proof follows exactly the same arguments as in \cref{sec:generic_transversality_travelling_fronts,sec:generic_elementarity_sym_stand_pulses}, except for the exclusion of ``case 3'', which will require a slightly different ad hoc argument. 
\subsubsection{Checking hypothesis \texorpdfstring{\cref{item:thm_sard_smale_condition_regularity}}{\ref{item:thm_sard_smale_condition_regularity}} of \texorpdfstring{\cref{thm:Sard_Smale}}{Theorem \ref{thm:Sard_Smale}}}
By contrast with the previous sections, the ambient space $\nnn$ is now a level set of dimension $2d-1$ (instead of $\rr^{2d}$); however the computation is similar. Indeed, it follows from \cref{subsec:eigenspaces_and_dimensions} that, 
on the one hand, $\dim(\mmm) = \dim\bigl(\partial\Bu_{E_0}(r)\bigr) + 1 = d - m(e_0)$ and, on the other hand, $\dim(\www) = d-m(e_0)-1$ so that $\codim(\www) =d+ m(e_0)$. Thus hypothesis \cref{item:thm_sard_smale_condition_regularity} of \cref{thm:Sard_Smale} is fulfilled. 
\subsubsection{Checking hypothesis \texorpdfstring{\cref{item:thm_sard_smale_condition_transversality}}{\ref{item:thm_sard_smale_condition_transversality}} of \texorpdfstring{\cref{thm:Sard_Smale}}{Theorem \ref{thm:Sard_Smale}}}
Take $(m_2,V_2)$ in the set $\Phi^{-1}(\www)$. Let $(\bu_2,\xi_2)$ denote the components of $m_2$, and, for every real quantity $\xi$, let us write
\[
U_2(\xi) = \bigl(u_2(\xi),v_2(\xi)\bigr) = S_{V_2}\bigl(\xi,\hatwulocZero{V_2}(\bu_2)\bigr)
\,.
\]
The function $\xi\mapsto u_2(\xi)$ is the profile of a standing pulse, connecting $e(V_2)$ to itself, for the potential $V_2$, and, according to \cref{stable_unstable_man_transverse_reversibility_subspace,consequence_if_bu_xi_in_mmm_asymm}, this standing pulse is \emph{asymmetric}. In addition, according to \cref{exit_from_e_on_unstable_manifold,local_unstable_stable_manifold_included_in_exit_ball}, the quantity $\xi_2$ is \emph{positive}. Let $D\Phi$ denote the \emph{full} differential (with respect to $m$ and $V$) of $\Phi$ at the point $(m_2,V_2\bigr)$. 
Hypothesis \cref{item:thm_sard_smale_condition_transversality} of \cref{thm:Sard_Smale} follows from the next \cref{lem:reach_all_directions_asymm_stand_pulse}. 
\begin{lemma}[perturbation of the potential reaching a given direction]
\label{lem:reach_all_directions_asymm_stand_pulse}
For every nonzero vector $(\phi_0,\psi_0)\in\rr^d\times\rr^d$ belonging to $T_{U_2(\xi_2)}\nnn$, there exists $W$ in $\CkbFull$ such that 
\begin{equation}
\label{reach_all_directions_asymm_stand_pulse}
\langle D\Phi\cdot(0,W)\,|\,(\phi_0,\psi_0)\rangle \not = 0
\,,
\end{equation}
and such that $W$ satisfies the condition
\begin{equation}
\label{support_of_W_does_not_meet_stable_unstable_manifolds_asymm_stand_pulse}
\supp(W)\cap\widebar{B}_{\rr^d}\bigr(e(V_2),\rExit\bigr) = \emptyset
\,.
\end{equation}
\end{lemma}
\begin{proof}
The proof is similar to those of \cref{lem:reach_all_directions_orthogonal_to_UzeroPrime_of_xiZero,lem:reach_all_directions_symm_stand_pulse}. Let $(\phi_2,\psi_2)$ be a nonzero vector in $\rr^d\times\rr^d$ belonging to $T_{U_2(\xi_2)}\nnn$. Let $W$ be a function in $\CkbFull$, and let us assume that condition \cref{support_of_W_does_not_meet_stable_unstable_manifolds_asymm_stand_pulse} holds. Let us again use the notation $T(\xi,\xi')$ to denote the family of evolution operators obtained by integrating the linearized differential system \cref{linearized_differential_system} (for the potential function $V_2$, and for a speed equal to $0$) between the times $\xi$ and $\xi'$. For every time $\xi$, let $T^*(\xi,\xi_2)$ denote the adjoint operator of $T(\xi,\xi_2)$, and let 
\[
\bigl(\phi(\xi),\psi(\xi)\bigr) = T^*(\xi,\xi_2)\cdot(\phi_2,\psi_2)
\,.
\]
Using the same computations as in  \cref{lem:reach_all_directions_orthogonal_to_UzeroPrime_of_xiZero,lem:reach_all_directions_symm_stand_pulse},
it follows from the inclusion \cref{local_unstable_stable_manifold_included_in_exit_ball} and the empty intersection \cref{support_of_W_does_not_meet_stable_unstable_manifolds_asymm_stand_pulse} that inequality \cref{reach_all_directions_asymm_stand_pulse} reads
\begin{equation}
\label{reach_all_directions_condition_with_psi_asymm_stand_pulse}
\int_0^{\xi_2}\nabla W\bigl(u_2(\xi)\bigr)\cdot \psi(\xi) \, d\xi \not= 0
\,.
\end{equation}
Observe that, according to inequality \cref{exit_from_e_on_unstable_manifold}, there exists a (nonempty) open interval $I$ included in $(0,\xi_2)$ and such that, for every $\xi$ in $I$, $u_2(\xi)\not\in\widebar{B}_{\rr^d}\bigl(e(V_2),\rExit\bigr)$. 
According to \cref{lem:equivalent_def_symm_hom_orbit}, the function $\xi\mapsto \dot u_2(\xi)$ does not vanish on $\rr$, thus a fortiori neither on $I$. 
As in \cref{subsubsec:checking_hyp_ii_Sard_Smale_thm,subsubsec:checking_hyp_ii_Sard_Smale_thm_sym_stand_pulse}, three cases must be considered for the construction of the perturbation $W$. 
\paragraph*{Case 1.} There exists a time $\xi^\dag$ in $I$ such that $\psi(\xi^\dag)$ is not collinear to $\dot u_2(\xi^\dag)$.

The same construction as in the first case of the proof of \cref{lem:reach_all_directions_orthogonal_to_UzeroPrime_of_xiZero} (or as in the first case of the proof of \cref{lem:reach_all_directions_symm_stand_pulse}) can then be carried out. 

\paragraph*{Case 2.} For every $\xi$ in $I$, $\psi(\xi)=\alpha(\xi) \dot u_2(\xi)$ with $\alpha(\cdot)$ not constant.

Again, the same construction as in the second case of the proof of \cref{lem:reach_all_directions_orthogonal_to_UzeroPrime_of_xiZero} (or as in the first case of the proof of \cref{lem:reach_all_directions_symm_stand_pulse}) can then be carried out. 

\paragraph*{Case 3.} For every $\xi$ in $(\xi_-,\xiTurn)$, $\psi(\xi)=\alpha \dot u_2(\xi)$ for some real (constant) quantity $\alpha$. 

As in \cref{subsubsec:checking_hyp_ii_Sard_Smale_thm,subsubsec:checking_hyp_ii_Sard_Smale_thm_sym_stand_pulse}, this third case has to be precluded by a specific argument. It follows from the adjoint linearized system \cref{adjoint_linearized_system} satisfied by $\phi$ and $\psi$ (with $c_0$ equal to zero) that, for every $\xi$ in $I$,
\begin{equation}\label{eq_case_3}
\phi(\xi) = -\dot\psi(\xi) = - \alpha\ddot u_2(\xi) = - \alpha\nabla V_2(u_2(\xi))
\,.
\end{equation}
Besides, since $(\phi_2,\psi_2)$ was assumed to belong to $T_{U_2(\xi_2)}\nnn$, it follows that $\bigl(\phi(\xi),\psi(\xi)\bigr)$ belongs to $T_{U_2(\xi)}H_{V_2}^{-1}\Bigl(H_{V_2}\bigl(E(V_2)\bigr)\Bigr)$ for all $\xi$ in $\rr$ (the level set of the energy is invariant by the flow). The orthogonal space of the tangent space to the level set $\nnn$ is a line spanned by the gradient of the Hamiltonian $\nabla H_{V_2}(U_2)=(-\nabla V_2(u_2(\xi)),\dot u_2(\xi))$. Thus, 
the condition $(\phi_2,\psi_2)\in T_{U_2(\xi_2)}\nnn$ reads 
\[
(\phi_2,\psi_2)\perp (-\nabla V_2(u_2(\xi)),\dot u_2(\xi))\,,
\quad\text{that is}\quad
\alpha\bigl( \nabla V_2(u_2)^2 + \dot u_2^2 \bigr) = 0
\,.
\]
This implies $\alpha=0$ and thus $(\phi,\psi)\equiv(0,0)$, a contradiction with the assumptions of \cref{lem:reach_all_directions_asymm_stand_pulse}. 

In summary, the third case cannot occur and, in both other cases, the same constructions as in the proofs of \cref{lem:reach_all_directions_orthogonal_to_UzeroPrime_of_xiZero,lem:reach_all_directions_symm_stand_pulse} can be carried out, leading to a perturbation $W$ satisfying the empty intersection \cref{support_of_W_does_not_meet_stable_unstable_manifolds_asymm_stand_pulse} and inequality \cref{reach_all_directions_condition_with_psi_asymm_stand_pulse} (and therefore also inequality \cref{reach_all_directions_asymm_stand_pulse}). 
\end{proof}
\subsubsection{Conclusion}
\label{subsubsec:conclusion_proof_prop_local_generic_transversality_standing_pulse}
\begin{proof}[Proof of \cref{prop:transv_asym_stand_pulses_bded_length_away_symmetry_subspace}]
To complete the proof of \cref{prop:transv_asym_stand_pulses_bded_length_away_symmetry_subspace} amounts to prove that the set $\nuTransvAsymmStandPulses(\bar{\xi},\varepsilon)$ is dense in $\nu$. 
It follows from \cref{lem:reach_all_directions_asymm_stand_pulse} that both hypotheses \cref{item:thm_sard_smale_condition_regularity,item:thm_sard_smale_condition_transversality} of \cref{thm:Sard_Smale} are fulfilled for the function $\Phi$ defined in \cref{def_Phi_asym_stand_pulse}. The conclusion of this theorem ensures that there exists a generic subset $\LambdaGen$ of $\Lambda_1$ such that, for every $V$ in $\LambdaGen$, the function $\Phi(\cdot,V)$ is transverse to $\www$. According to \cref{prop:equivalence_transversality_asym_stand_pulse}, the set $\nuTransvAsymmStandPulses(\bar{\xi},\varepsilon)$ is a superset of $\LambdaGen$; in particular, there exists potentials in $\nuTransvAsymmStandPulses(\bar{\xi},\varepsilon)$ that are arbitrarily close to $V_1$. Since $V_1$ was any potential in $\nu$, this proves the intended density. \Cref{prop:transv_asym_stand_pulses_bded_length_away_symmetry_subspace} is proved.
\end{proof}
As shown at the end of \cref{subsec:away}, \cref{prop:transv_asym_stand_pulses_bded_length_away_symmetry_subspace} implies 
\cref{prop:local_generic_transversality_standing_pulse2}, which in turn implies \cref{prop:global_generic_transversality_asym_stand_pulses}. 
\subsection{Transversality of symmetric standing pulses?}
\label{subsec:failure_generic_transversality_symm_stand_pulses}
As it stands, the proof of the generic transversality of asymmetric standing pulses provided above does not directly apply to symmetric ones. Indeed, for a symmetric standing pulse $\xi\mapsto u(\xi)$, with (say) turning time $0$, the condition corresponding to \cref{reach_all_directions_condition_with_psi} or \cref{reach_all_directions_condition_with_psi_asymm_stand_pulse} reads
\[
\int_{-\bar{\xi}}^{\bar{\xi}} \nabla W\bigl(u(\xi)\bigr)\cdot \psi(\xi)\, d\xi \not= 0
\ \ \text{or equivalently}\ \ 
\int_{-\bar{\xi}}^0 \nabla W\bigl(u(\xi)\bigr)\cdot \bigl(\psi(\xi) + \psi(-\xi)\bigr) \, d\xi \not = 0
\,,
\]
where $\bar{\xi}$ is a large enough positive quantity. This condition cannot be fulfilled if the function $\xi\mapsto \psi(\xi)$ is odd and, due to the symmetry of the adjoint linear equation
\[
\ddot \psi(\xi) = D^2 V \bigl(u(\xi)\bigr)\cdot \psi(\xi)
\,,
\]
this happens as soon as $\psi(0)$ vanishes. This case, corresponding to the degeneracy of the first order derivative with respect to perturbations of the potential, can therefore not be excluded. Possibly, the second order derivative could be investigated but the computation goes beyond the scope of this paper. For that reason, the generic transversality of symmetric standing pulses is not established here and remains, to our best knowledge, an open question.
\section{Generic non-existence of standing fronts}
\label{sec:gen_non_existence_standing_fronts}
Let us take and fix a positive quantity $R$. Due to the Hamiltonian invariance, precluding the existence of standing fronts is a simple task. 
\begin{proposition}
\label{prop:gen_non_existence_standing_fronts}
There exists a dense open subset of $\vvvQuad{R}$ such that, for every potential $V$ in this subset, there is no standing front for this potential. 
\end{proposition}
\begin{proof}
Let us consider the dense open subset $\vvvQuadMorse{R}$ of $\vvvQuad{R}$ containing the functions of $\vvvQuad{R}$ satisfying the Morse property (this notation was introduced in \cref{def_vvv_Quad_R_Morse}), and let $V$ denote a potential in $\vvvQuadMorse{R}$. The number of critical points of such a potential is finite, and, up to applying to $V$ an arbitrarily small localized perturbation around each of these critical points, it may be assumed that each of these critical points belongs to a level set of $V$ containing no other critical point. This property is open and dense in $\vvvQuadMorse{R}$, thus in $\vvvQuad{R}$, and, since the Hamiltonian $H_V$ defined in \cref{Hamiltonian} is constant along the profile of a standing front, it prevents the existence of a standing front. \Cref{prop:gen_non_existence_standing_fronts} is proved. 
\end{proof}
\section{Proof of the main results}
\label{sec:proof_main}
\Cref{prop:global_generic_transversality_travelling_fronts,prop:global_generic_elementarity_sym_stand_pulses,prop:global_generic_transversality_asym_stand_pulses,prop:gen_non_existence_standing_fronts} show the genericity of the properties considered in \cref{thm:main}, but only inside the space $\vvvQuad{R}$ of the potentials that are quadratic past some radius $R$. 
Working in this last space is easier because it is a second countable Banach space and the flows associated to its potentials are global. In this \namecref{sec:proof_main}, the arguments will be adapted to obtain the genericity of the same properties in the space $\vvvFull=\ccc^{k+1}(\rr^d,\rr)$ of all potentials, endowed with the extended topology (see \cref{subsec:space_of_potentials}).
\subsection{Proof of conclusion \texorpdfstring{\cref{item_thm:main_travelling_front} of \cref{thm:main}}{\ref{item_thm:main_travelling_front} of Theorem \ref{thm:main}}}
\label{subsec:proof_thm_main_for_nonzero_speeds}
%
%
Let us recall the notation $\fff_V$ introduced in \cref{notation_fff_V}, and, for every positive quantity $R$, let us consider the set
\begin{equation}
\label{notation_fff_V_R}
\fff_{V,R} = \left\{ (c,u)\in\fff_V : \sup_{\xi\in{\rr}}\abs{u(\xi)}\le R \right\}
\,
\end{equation}
of the travelling fronts of $\fff_V$ (invading a minimum point of $V$) with a profile contained in $\widebar{B}_{\rr^d}(0,R)$. 
As shown thereafter, the following proposition yields conclusion \cref{item_thm:main_travelling_front} of \cref{thm:main}. 
\begin{proposition}
\label{prop:genericity_transv_tf_up_to_R}
For every positive quantity $R$, there exists a generic subset $\vvvFullTransvfff{R}$ of $\vvvFull$ such that, for every potential function $V$ in this subset, $V$ is a Morse function and every travelling front $(c,u)$ in $\fff_{V,R}$ is transverse. 
\end{proposition}
\begin{proof}[Proof that \cref{prop:genericity_transv_tf_up_to_R} yields conclusion \texorpdfstring{\cref{item_thm:main_travelling_front} of \cref{thm:main}}{\ref{item_thm:main_travelling_front} of Theorem \ref{thm:main}}]
The set
\[
\bigcap_{R\in\nn^*}\vvvFullTransvfff{R}
\,,
\]
is a countable intersection of generic subsets of $\vvvFull$ and is therefore again a generic subset of $\vvvFull$. For every potential function $V$ in this set, $V$ is a Morse function and every travelling front in $\fff_V$ belongs to $\fff_{V,R}$ as soon as $R$ is large enough, and is therefore, according to the property of the set $\vvvFullTransvfff{R}$ stated in \cref{prop:genericity_transv_tf_up_to_R}, transverse. Statement \cref{item_thm:main_travelling_front} of \cref{thm:main} is proved. 
\end{proof}
The aim of \cref{subsec:proof_thm_main_for_nonzero_speeds} is thus to prove \cref{prop:genericity_transv_tf_up_to_R}. 
Before doing so, here are a few preliminary comments. Let $R$ be a positive quantity. \Cref{prop:global_generic_transversality_travelling_fronts} states that there exists a generic subset $\vvvQuadTransvfff{R}$ of $\vvvQuad{R}$ such that, for every potential $\Vquad$ in this subset, all travelling fronts in $\fff_{\Vquad}$ are transverse. However, due to the constraint at $|u|=R$, the extension to $\rr^d$ of all the truncations of these potentials in $\widebar{B}_{\rr^d}(0,R)$ is meagre. The idea is to take some margin: consider the generic subset $\vvvQuadTransvfff{(R+1)}$ of $\vvvQuad{(R+1)}$ and, using the notation introduced in definition \cref{def_res_R_prime_R}, consider the set
\begin{equation}
\label{first_attempt_application_res_res_minus_one}
\res_{R,\infty}^{-1}\circ\res_{R,(R+1)}(\vvvQuadTransvfff{(R+1)})
\,.
\end{equation}
For every potential $\Vfull$ in this set, all travelling fronts in $\fff_{\Vfull,R}$ are transverse; indeed, this property depends only on the values of $\Vfull$ inside the ball $\widebar{B}_{\rr^d}(0,R)$, where $\Vfull$ must be identically equal to some potential $\Vquad$ of $\vvvQuadTransvfff{(R+1)}$. It is tempting to look for an extension of \cref{cor:truncation_extension_of_potentials} to generic subsets, which would yield the genericity of the set \cref{first_attempt_application_res_res_minus_one}. Unfortunately, this corollary definitely applies to open dense subsets, and not to generic ones. Pursuing further in this direction, observe that, since $\vvvQuadTransvfff{(R+1)}$ is a generic subset of $\vvvQuad{(R+1)}$, there exists a countable family $(\ooo_N)_{N\in\nn}$ of dense open subsets of $\vvvQuad{(R+1)}$ such that 
\begin{equation}
\label{not_suitable_countable_intersection_of_dense_open_subsets_of_vvvQuadR}
\bigcap_{N\in\nn}\ooo_N \subset \vvvQuadTransvfff{(R+1)}
\,,
\end{equation}
leading to
\[
\res_{R,\infty}^{-1}\circ\res_{R,(R+1)}\Big(\bigcap_{N\in\nn}\ooo_N\Big) ~ \subset ~\res_{R,\infty}^{-1}\circ\res_{R,(R+1)}(\vvvQuadTransvfff{R+1})~. 
\]
According to general properties of functions, the following inclusion holds:
\begin{equation}
\label{inclusion_we_would_like_to_be_an_equality}
\res_{R,(R+1)}\Big(\bigcap_{N\in\nn}\ooo_N\Big) \subset \bigcap_{N\in\nn}\res_{R,(R+1)}(\ooo_N)
\,.
\end{equation}
If this inclusion was an equality, then, still according to general properties of functions, the following equality would hold:
\[
\res_{R,\infty}^{-1}\circ\res_{R,(R+1)}\Big(\bigcap_{N\in\nn}\ooo_N\Big) = \bigcap_{N\in\nn}\res_{R,\infty}^{-1}\circ\res_{R,(R+1)}(\ooo_N)
\,,
\]
and, since according to \cref{cor:truncation_extension_of_potentials} the right-hand side of this equality is a countable intersection of dense open subsets of $\vvvFull$, the intended conclusion that the set \cref{first_attempt_application_res_res_minus_one} is generic in $\vvvFull$ would follow. Unfortunately, \cref{prop:global_generic_transversality_travelling_fronts} provides no clue about the sets $\ooo_N$ and a strict inclusion in \cref{inclusion_we_would_like_to_be_an_equality} cannot be precluded. However, let us make the following key observation, which enlightens the remaining of the proof: if the property ``a given potential $V$ belongs to $\ooo_N$'' only depends on the values of $V$ inside the ball $\widebar{B}_{\rr^d}(0,R)$, then inclusion \cref{inclusion_we_would_like_to_be_an_equality} \emph{is} actually an equality. 

The main step in the proof is thus to construct dense subsets $\ooo_N$ of $\vvvQuad{(R+1)}$ such that: 
\begin{enumerate}
\item for every potential $\Vquad$ in $\bigcap_n \ooo_N$, every travelling front in $\fff_{V,R}$ is transverse, 
\item and the property ``a given potential $V$ belongs to $\ooo_N$'' only depends on the values of $V$ inside the ball $\widebar{B}_{\rr^d}(0,R)$.
\end{enumerate}
\begin{proof}[Proof of \cref{prop:genericity_transv_tf_up_to_R}]
\renewcommand{\qedsymbol}{}
As above, let $R$ denote a positive quantity. Let $V_0$ denote a potential function in $\vvvQuad{(R+1)}$, let $e_{-,0}$ and $e_{+,0}$ denote a non-degenerate critical point and a non-degenerate minimum point of $V_0$ and let $c_0$ denote a positive speed. Let us consider the neighbourhoods $\nu_{\VZeroeMinusZeroePlusZerocZero}$ of $V_0$ in $\vvvQuad{(R+1)}$ and $\ccc_{\VZeroeMinusZeroePlusZerocZero}$ of $c_0$ in $(0,+\infty)$ provided by \cref{prop:local_generic_transversality_travelling_fronts_given_critical_points_speed} for these objects. Recall that those neighbourhoods are the ones from which, for every $V$ in $\nu_{\VZeroeMinusZeroePlusZerocZero}$ and every $c$ in $\ccc_{\VZeroeMinusZeroePlusZerocZero}$, the functions $\hatwuloc{c}{V}$, the sets $\mmm$ and $\www$ and the functions $\Phiu$ and $\Phis$ and $\Phi$ were defined in \cref{subsec:proof_local_gen}. Up to replacing the neighbourhood $\nu_{\VZeroeMinusZeroePlusZerocZero}$ by its interior, we may assume that it is \emph{open} in $\vvvQuad{(R+1)}$. Similarly, we may assume that $\ccc_{\VZeroeMinusZeroePlusZerocZero}$ is \emph{compact} in $\rr$. Let $N$ denote a non negative integer and let us consider the set
\begin{equation}
\label{def_mmm_N}
\mmm_N = \bbbu\times\bbbs\times(-\infty,N]\times\ccc_{\VZeroeMinusZeroePlusZerocZero} = \bigl\{(\bu,\bs,\xi,c) \in \mmm : \xi\le N\bigr\}
\,.
\end{equation}
As in \cref{subsec:proof_local_gen}, let us define $\nnn$ as $(\rr^{2d})^2$, and let us consider the set
\begin{equation}
\label{def_ooo_VZeroeMinusZeroePlusZerocZeroN}
\begin{aligned}
&\ooo_{\VZeroeMinusZeroePlusZerocZeroN} = \Bigl\{
V\in\nu_{\VZeroeMinusZeroePlusZerocZero}: \Phi\bigl(\mmm_N,V\bigr) \text{ is transverse to $\www$ in $\nnn$}
\Bigr\}
\,.
\end{aligned}
\end{equation}
As shown in \cref{prop:equivalence_transversality}, this set $\ooo_{\VZeroeMinusZeroePlusZerocZeroN}$ is made of the potential functions $V$ in $\nu_{\VZeroeMinusZeroePlusZerocZero}$ such that every profile $\xi\mapsto u(\xi)$ of a front travelling at a speed $c$ in $\ccc_{\VZeroeMinusZeroePlusZerocZero}$ and connecting $e_-(V)$ to $e_+(V)$ for this potential, and connecting $\partial\Wuloc{c}{V}\bigl(E_-(V)\bigr)$ to $\partial\Wsloc{c}{V}\bigl(E_+(V)\bigr)$ in a time not larger than $N$, is transverse. 
\end{proof}
\begin{lemma}
\label{lem:ooo_open_dense}
The set $\ooo_{\VZeroeMinusZeroePlusZerocZeroN}$ is a dense open subset of $\nu_{\VZeroeMinusZeroePlusZerocZero}$. 
\end{lemma}
\begin{proof}[Proof of \cref{lem:ooo_open_dense}]
The density is a direct consequence of \cref{prop:local_generic_transversality_travelling_fronts_given_critical_points_speed} which states that, generically with respect to $V$ in $\nu_{\VZeroeMinusZeroePlusZerocZero}$, the whole image of $\mmm$ by the map $m\mapsto\Phi(m,V)$ is transverse to $\www$. To prove the openness, let us argue as in \cref{subsec:openness_neighb_transv_asym_stand_pulses}.
Let us consider a sequence $(V_n)_{n\in\nn}$ of potentials in $\nu_{\VZeroeMinusZeroePlusZerocZero}$ converging to a potential $V_\infty$ in $\nu_{\VZeroeMinusZeroePlusZerocZero}$, and such that, for every $n$ in $\nn$, there exists $m_n=(\bu_n,\bs_n,\xi_n,c_n)$ in $\mmm_N$ such that the set $\Phi(\mmm_N,V_n)$ is not transverse to $\www$ at $\Phi(m_n,V_n)$. Observe that, according to the empty intersection \cref{pi_pos_of_local_unstable_stable_man_do_not_intersect}, $\xi_n$ must be positive. As a consequence, by compactness of $\bbbu\times\bbbs\times[0,N]\times\ccc_{\VZeroeMinusZeroePlusZerocZero}$, we may assume that $m_n$ converges, as $n$ goes to $+\infty$, to a point $m_\infty$ of $\mmm_N$. Then, by continuity, the image $\Phi(\mmm_N,V_\infty)$ is not transverse to $\www$ at $\Phi(m_\infty,V_\infty)$. This proves that $\nu_{\VZeroeMinusZeroePlusZerocZero}\setminus\ooo_{\VZeroeMinusZeroePlusZerocZeroN}$ is closed in $\nu_{\VZeroeMinusZeroePlusZerocZero}$, and yields the intended conclusion.
\end{proof}
\begin{proof}[Continuation of the proof of \cref{prop:genericity_transv_tf_up_to_R}]
\renewcommand{\qedsymbol}{}
\label{page_intersection_over_all_critical_points_extension_wwwFull_trav_fronts}
Let us make the additional assumption that the potential $V_0$ is a Morse function. Then, the set of critical points of $V_0$ is finite and depends smoothly on $V$ in a neighbourhood $\nuRobust(V_0)$ of $V_0$. Intersecting the sets $\nu_{\VZeroeMinusZeroePlusZerocZero}$ and $\ccc_{\VZeroeMinusZeroePlusZerocZero}$ and $\ooo_{\VZeroeMinusZeroePlusZerocZeroN}$ above over all the possible couples $(e_{-,0},e_{+,0})$ in $\SigmaCrit(V_0)\times\SigmaMin(V_0)$ provides an open neighbourhood $\nu_{\VZerocZero}$ of $V_0$, a compact neighbourhood $\ccc_{\VZerocZero}$ of $c_0$ and an open dense subset $\ooo_{\VZerocZeroN}$ of $\nu_{\VZerocZero}$ such that, for all $V\in \ooo_{\VZerocZeroN}$, every front travelling at speed $c\in \ccc_{\VZerocZero}$ and connecting the local (un)stable manifolds of two points $(e_{-},e_{+})$ in $\SigmaCrit(V)\times\SigmaMin(V)$ within the ``time'' $N$, is transverse. 

Denoting by $\interior(A)$ the interior of a set $A$ and using the notation of definition \cref{def_res_R_prime_R}, let us introduce the sets
\begin{align}
\label{def_tilde_nu_VZerocZero}
\tilde{\nu}_{\VZerocZero} &= \res_{R,\infty}^{-1}\circ\res_{R,(R+1)}(\nu_{\VZerocZero}) \,, \\
\text{and}\quad
\label{def_tilde_ooo}
\tilde{\ooo}_{\VZerocZeroN} &= \res_{R,\infty}^{-1}\circ\res_{R,(R+1)}\bigl(\ooo_{\VZerocZeroN}\bigr) \,, \\
\text{and}\quad
\label{def_tilde_ooo_ext}\tilde{\ooo}^{\ext}_{\VZerocZeroN} &= \tilde{\ooo}_{\VZerocZeroN} \sqcup \interior\bigl(\vvvFull\setminus\tilde{\nu}_{\VZerocZero}\bigr)
\,.
\end{align}
In other words, a potential $\tilde V$ of $\vvvFull$ is in $\tilde{\nu}_{\VZerocZero}$ (in $\tilde{\ooo}_{\VZerocZeroN}$) if it coincides, inside the ball $\widebar{B}_{\rr^d}(0,R)$, with a potential $\Vquad$ quadratic past $R+1$ and belonging to $\nu_{\VZerocZero}$ (to $\ooo_{\VZerocZeroN}$). The last set $\tilde{\ooo}^{\ext}_{\VZerocZeroN}$ is an extension of the open dense subset $\tilde{\ooo}_{\VZerocZeroN}$ of $\tilde{\nu}_{\VZerocZero}$, obtained by adding all potentials outside (the closure of) $\tilde{\nu}_{\VZerocZero}$. 
\end{proof}
\begin{lemma}
\label{lem:properties_tilde_ooo_ext_VZerocZeroN}
The set $\tilde{\ooo}^{\ext}_{\VZerocZeroN}$ is a dense open subset of $\vvvFull$. 
\end{lemma}
\begin{proof}[Proof of \cref{lem:properties_tilde_ooo_ext_VZerocZeroN}]
According to \cref{cor:truncation_extension_of_potentials}, the set $\tilde{\nu}_{\VZerocZero}$ is an open subset of $\vvvFull$, and the set $\tilde{\ooo}_{\VZerocZeroN}$ is a dense open subset of $\tilde{\nu}_{\VZerocZero}$. Thus, according to its definition \cref{def_tilde_ooo_ext}, the set $\tilde{\ooo}^{\ext}_{\VZerocZeroN}$ is a dense open subset of $\vvvFull$.
\end{proof}

\begin{proof}[Continuation of the proof of \cref{prop:genericity_transv_tf_up_to_R}]
\renewcommand{\qedsymbol}{}
Since $\vvvQuad{(R+1)}$ is a separable space, it is second-countable. Thus $\vvvQuadMorse{(R+1)}\times(0,+\infty)$ is also second-countable and can be covered by a countable number of products ${\nu}_{\VZerocZero}\times\ccc_{\VZerocZero}$. With symbols, there exists a countable family $(\VZeroicZeroi)_{i\in\nn}$ of elements of $\vvvQuadMorse{(R+1)}\times(0,+\infty)$ so that 
\begin{equation}
\label{countable_cover_of_vvvQuadMorseR_times_zero_plusInfty}
\vvvQuadMorse{(R+1)}\times(0,+\infty) = \bigcup_{i\in\nn}{\nu}_{\VZeroicZeroi}\times\ccc_{\VZeroicZeroi}
\,.
\end{equation}
Notice here the importance of first working with $\vvvQuad{(R+1)}$, which is second-countable, instead of the full space $\vvvFull$, which is not. Let us consider the set
\begin{equation}
\label{intersection_of_dense_opens_subsets_of_vvvQuadR_providing_transversality_in_BRprime}
\vvvFullTransvfff{R} = \vvvFullMorse \cap \left(\bigcap_{(i,N)\in\nn^2}\tilde{\ooo}^{\ext}_{\VZeroicZeroiN}\right)
\,,
\end{equation}
where $\vvvFullMorse$ is the set of potentials in $\vvvFull$ which are Morse functions.
\end{proof}
\begin{lemma}
\label{lem:intersection_of_dense_opens_subsets_of_vvvQuadR_providing_transversality_in_BRprime}
For every potential $\tilde V$ in the set $\vvvFullTransvfff{R}$, every travelling front $(u,c)$ in $\fff_{\tilde V,R}$ is transverse. 
\end{lemma}
\begin{proof}[Proof of \cref{lem:intersection_of_dense_opens_subsets_of_vvvQuadR_providing_transversality_in_BRprime}]
Let $\tilde V$ be a potential function in the set $\vvvFullTransvfff{R}$ and $(c,u)$ be a travelling front in $\fff_{\tilde V,R}$. According to \cref{lem:res_R_R_prime_is_continuous_and_surjective_and_open}, the map $\res_{R,(R+1)}$ is surjective, thus there exists a potential function $V$ in $\vvvQuad{(R+1)}$ such that $V$ belongs to $\res_{R,(R+1)}^{-1}\circ\res_{R,\infty}(\tilde V)$ (in other words $V$ coincides with $\tilde V$ on $\widebar{B}_{\rr^d}(0,R)$). Since $\tilde V$ is a Morse function, the 
critical points of $V$ in $\widebar{B}_{\rr^d}(0,R)$ are degenerate, and up to applying to $V$ a small perturbation in $\widebar{B}_{\rr^d}(0,R+1)\setminus \widebar{B}_{\rr^d}(0,R)$, we may assume that its critical point in this set are also nondegenerate, so that $V$ is actually also a Morse function. Since $\tilde V$ coincides with $V$ inside $\widebar{B}_{\rr^d}(0,R)$ and since the travelling front $u$ is contained in this ball, it is also a travelling front of $V$ and it is sufficient to show that $(u,c)$ is a transverse travelling front for $V$. 

According to equality \cref{countable_cover_of_vvvQuadMorseR_times_zero_plusInfty}, there exists an integer $i$ such that $V$ belongs to ${\nu}_{V_{0,i},\,c_{0,i}}$ and $c$ belongs to $\ccc_{V_{0,i},\,c_{0,i}}$. Then, since $V$ and $\tilde V$ coincide on $\widebar{B}_{\rr^d}(0,R)$, $\tilde V$ belongs to $\tilde{\nu}_{V_{0,i},\,c_{0,i}}$ (definition \cref{def_tilde_nu_VZerocZero}). Besides, it follows from definition \cref{intersection_of_dense_opens_subsets_of_vvvQuadR_providing_transversality_in_BRprime} that, for every integer $N$, $\tilde V$ belongs to $\tilde{\ooo}^{\ext}_{\VZeroicZeroiN}$; and since $V$ is also in $\tilde{\nu}_{V_{0,i},\,c_{0,i}}$, it follows from definition \cref{def_tilde_ooo_ext} that $\tilde V$ actually belongs to $\tilde{\ooo}_{\VZeroicZeroiN}$. 

Let us denote by $e_-$ and $e_+$ the critical points of $V$ (and $\tilde{V}$) approached by $u(\xi)$ as $\xi$ goes to $-\infty$ and $+\infty$ respectively. According to the definition of the neighbourhood $\nu_{V_{0,i},\,c_{0,i}}$ of $V_{0,i}$, there exists a (unique) critical point $e_{-,0,i}$ and a (unique) minimum point $e_{+,0,i}$ of $V_{0,i}$ such that, if $W\mapsto e_{-,i}(W)$ and $W\mapsto e_{+,i}(W)$ denote the functions which ``follow'' these critical points for $W$ in $\nuRobust(V_{0,i})$, then $e_-$ equals $e_{-,i}(V)$ and $e_+$ equals $e_{+,i}(V)$. Let us keep the notation $\mmm$ and $\Phi$ to denote the objects defined as in \cref{subsec:proof_local_gen} for the neighbourhoods $\nu_{\VZeroieMinusZeroiePlusZeroicZeroi}$ of $V_{0,i}$ and $\ccc_{\VZeroieMinusZeroiePlusZeroicZeroi}$ of $c_{0,i}$. The travelling front $(c,u)$ therefore corresponds to an intersection between $\Phi(\mmm,V)$ and $\www$, which occurs at a certain point $m$ of $\mmm$ and thus for a certain (positive) time $\xi$ which is the time that the profile of this travelling front takes to go from the border of the local unstable manifold of $e_-$ to the border of the local stable manifold of $e_+$. 

Let $N$ denote an integer not smaller than $\xi$. Since $\tilde V$ belongs to $\tilde{\ooo}_{\VZeroicZeroiN}$, there must exist (according to definition \cref{def_tilde_ooo}) a potential $V_N$ in ${\nu}_{V_{0,i},\,c_{0,i}}$ identically equal to $\tilde V$ (and $V$) on the ball $\widebar{B}_{\rr^d}(0,R)$ and belonging to ${\ooo}_{\VZeroicZeroiN}$. Again, $(c,u)$ is a travelling front for $V_N$ and the previous correspondence between this front and an intersection between $\Phi(\mmm,V_N)$ and $\www$ still holds. Since $V_N$ belongs to ${\ooo}_{\VZeroicZeroiN}$, the aforementioned intersection must be transverse, leading to the transversality of the front $(u,c)$ for $V_N$. 

Again, the three potentials $\tilde{V}$ and $V$ and $V_N$ considered here have the same values along the profile of the travelling front $(u,c)$. Thus, this front is also transverse for $\tilde V$.
\end{proof}
\begin{proof}[End of the proof of \cref{prop:genericity_transv_tf_up_to_R}]
The set $\vvvFullTransvfff{R}$ defined in \cref{intersection_of_dense_opens_subsets_of_vvvQuadR_providing_transversality_in_BRprime} is a countable intersection of dense open subsets of $\vvvFull$, and is therefore a generic subset of $\vvvFull$. In view of \cref{lem:intersection_of_dense_opens_subsets_of_vvvQuadR_providing_transversality_in_BRprime}, \cref{prop:genericity_transv_tf_up_to_R} is proved. 
\end{proof}
\subsection{Proof of conclusions \texorpdfstring{\cref{item_thm:main_symmetric_standing_pulse,item_thm:main_asymmetric_standing_pulse} of \cref{thm:main}}{\ref{item_thm:main_symmetric_standing_pulse,item_thm:main_asymmetric_standing_pulse} of Theorem \ref{thm:main}}}
\label{subsec:proof_thm_main_zero_speed}
The proof of conclusions \cref{item_thm:main_symmetric_standing_pulse,item_thm:main_asymmetric_standing_pulse} of \cref{thm:main} is similar to the proof of conclusion \cref{item_thm:main_travelling_front} provided in the previous \namecref{subsec:proof_thm_travelling}. As a consequence, only the core arguments will be reproduced here. 
Let us recall the notation $\ppp_V$ introduced in \cref{notation_ppp_V}, and, for every positive quantity $R$, let us consider the set
\[
\ppp_{V,R} = \left\{ u\in\ppp_V : \sup_{\xi\in{\rr}}\abs{u(\xi)}\le R \right\}
\,.
\]
As shown in the previous \namecref{subsec:proof_thm_main_for_nonzero_speeds} for \cref{prop:genericity_transv_tf_up_to_R} and conclusion \cref{item_thm:main_travelling_front} of \cref{thm:main}, the following proposition yields conclusions \cref{item_thm:main_symmetric_standing_pulse,item_thm:main_asymmetric_standing_pulse} of of \cref{thm:main}. 
\begin{proposition}
\label{prop:genericity_transv_sp_up_to_R}
For every positive quantity $R$, there exists a generic subset $\vvvFullTransvppp{R}$ of $\vvvFull$, included in $\vvvFullMorse$, such that, for every potential function $V$ in $\vvvFullTransvppp{R}$, every standing pulse $u$ in $\ppp_{V,R}$ is: elementary if this standing pulse is symmetric, and transverse if this standing pulse is asymmetric. 
\end{proposition}
\begin{proof}
Let $R$ denote a positive quantity and let $V_0$ denote a Morse potential function in $\vvvQuad{(R+1)}$. Let $e_0$ denote a non-degenerate critical point of $V_0$ and let us consider an open neighbourhood $\nu_{\VZeroeZero}$ of $V_0$ in $\vvvQuad{(R+1)}$ included in both neighbourhoods provided by \cref{prop:local_generic_transversality_standing_pulse,prop:local_generic_transversality_standing_pulse2}. For every $N$ in $\nn^*$ and for every $V$ in $\nu_{\VZeroeZero}$, let us consider the subset $\ooo_{\VZeroeZeroN}$ of $\nu_{\VZeroeZero}$ defined as the set of potentials $V$ in $\nu_{\VZeroeZero}$ satisfying the following two conditions: 
\begin{enumerate}
\item every symmetric standing pulse of $V$, connecting $\partial\WulocZero{V}\bigl(E(V)\bigr)$ to the symmetric subspace $\sssSym$ in a time not larger than $N$, is elementary;
\item and every asymmetric standing pulse of $V$, connecting $\partial\WulocZero{V}\bigl(E(V)\bigr)$ to\\ 
$\partial\WslocZero{V}\bigl(E(V)\bigr)$ in a time not larger than $N$ while remaining at a distance not smaller than $1/N$ of $\sssSym$, is transverse.
\end{enumerate}
The same arguments as in the proof of \cref{lem:ooo_open_dense} show that the set $\ooo_{\VZeroeZeroN}$ is a dense open subset of $\nu_{\VZeroeZero}$: the density follows from \cref{prop:local_generic_transversality_standing_pulse,prop:local_generic_transversality_standing_pulse2} and, regarding the openness, the key new ingredient is the condition that every asymmetric standing pulse remains at a distance at least $1/N$ of $\sssSym$. Indeed, a sequence of asymmetric standing pulses (as considered in the proof) may (generally speaking) approach a symmetric standing pulse which may be non-transverse even if it is elementary. Staying away from $\sssSym$ precludes this possibility. 

As on page \pageref{page_intersection_over_all_critical_points_extension_wwwFull_trav_fronts}, let us consider the intersections of the previous sets over all the critical points of $V_0$:
\begin{equation*}
\label{two_sets_nu_ooo_intersections_over_eZero}
\nu_{V_0} = \bigcap_{e_0\in\SigmaCrit(V_0)} \nu_{\VZeroeZero}
\quad\text{and}\quad
\ooo_{\VZeroN} = \bigcap_{e_0\in\SigmaCrit(V_0)} \ooo_{\VZeroeZeroN}
\,.
\end{equation*}
The set $\nu_{V_0}$ is still open in $\vvvQuad{(R+1)}$ and the set $\ooo_{\VZeroN}$ is still a dense open subset of $\nu_{V_0}$. As in definitions \cref{def_tilde_nu_VZerocZero,def_tilde_ooo,def_tilde_ooo_ext}, these sets can be extended as follows:
\[
\begin{aligned}
\tilde{\nu}_{V_0} &= \res_{R,\infty}^{-1}\circ\res_{R,(R+1)}(\nu_{V_0}) \,, \\
\tilde{\ooo}_{\VZeroN} &= \res_{R,\infty}^{-1}\circ\res_{R,(R+1)}\bigl(\ooo_{\VZeroN}\bigr)\,,  \\
\text{and}\quad
\tilde{\ooo}^{\ext}_{\VZeroN} &= \tilde{\ooo}_{\VZeroN} \sqcup \interior\bigl(\vvvQuad{(R+1)}\setminus\tilde{\nu}_{V_0}\bigr)
\,.
\end{aligned}
\]
The end of the proof follows the same arguments as the ones of \cref{subsec:proof_thm_main_for_nonzero_speeds}. The set $\vvvQuadMorse{(R+1)}$ can be covered by a countable number of subsets $\tilde{\nu}_{V_{0,i}}$ and the set 
\[
\vvvFullTransvppp{R}= \vvvFullMorse \cap \left(\bigcap_{(i,N)\in\nn^2}\tilde{\ooo}^{\ext}_{\VZeroiN}\right)
\]
is the generic subset the existence of which was stated in \cref{prop:genericity_transv_sp_up_to_R}. 
\end{proof}
\subsection{Proof of conclusion \texorpdfstring{\cref{item_thm:main_standing_front} of \cref{thm:main}}{\ref{item_thm:main_standing_front} of Theorem \ref{thm:main}}}
\label{subsec:proof_thm_main_no_standing_front}
%
Let us consider the set $\ooo_R$ of potentials $V$ of $\vvvFull$ such that all the critical points of $V$ in $\widebar{B}_{\rr^d}(0,R)$ are non-degenerate and have different values. The same arguments as in \Cref{prop:gen_non_existence_standing_fronts} show that this set $\ooo_R$ is an open dense subset of $\vvvFull$, so that the intersection $\cap_{R\in\nn^*} \ooo_R$ is generic in $\vvvFull$. Since the critical points connected by a standing front must belong to the same level set of the potential, no standing front can exist for a potential in this intersection. 
\subsection{Proof of conclusions \texorpdfstring{\cref{item:cor_main_bistable_front,item:cor_main_bistable_pulse,item:cor_main_countable,item:cor_main_robust}}{\ref{item:cor_main_bistable_front} to \ref{item:cor_main_robust}} of \texorpdfstring{\cref{cor:main}}{Corollary \ref{cor:main}}}
\label{subsec:proof_thm_travelling}
Let $V$ be a potential function belonging to the generic subset provided by \cref{thm:main}, let $(c,u)$ be a travelling front in $\fff_V$, and let $e_-$ and $e_+$ denote the critical point and the minimum point of $V$ connected by this travelling front. According to \cref{table:dim_stable_unstable_centre},
\[
\begin{aligned}
\dim \left(\bigcup_{c'>0} \{c'\}\times \Wu_{c',V}(E_-) \right) 
 &= d-m(e_-)+1 \,,\\
\text{and}\quad
\dim\left(\bigcup_{c'>0} \{c'\}\times \Ws_{c',V}(E_+) \right) &= d+1
\,.
\end{aligned}
\]
The intersection between these two manifolds contains at least the curve $\{c\}\times U(\rr)$ corresponding to the travelling front. Thus, the dimension of the sum of the tangent spaces to these two manifolds is not larger than the quantity
\[
\bigl(d-m(e_-)+1\bigr)+(d+1)-1 = 2d + 1 - m(e_-)
\,.
\]
Since according to \cref{def:transverse_travelling_front,thm:main} the intersection between these two manifolds is transverse in $\rr^{2d+1}$, along the set $\{c\}\times U(\rr)$, this quantity is not smaller than $2d+1$, so that the Morse index $m(e_-)$ must be zero. This proves conclusion \cref{item:cor_main_bistable_front} of \cref{cor:main}. 

Now let us assume that $u$ is the profile of a standing pulse and let $e$ denote the critical point of $V$ such that this pulse connects $e$ to itself. According to \cref{table:dim_stable_unstable_centre},
\[
\dim\left(\Wu_{V}(E)\right) = d-m(e)
\quad\text{and}\quad
\dim\left(\Ws_{V}(E)\right) = d-m(e)
\,.
\]
According to \cref{def:elementary_symm_stand_pulse,thm:main}, if $u$ is symmetric then the intersection between $\Wu_{V}(E)$ and the $d-$dimensional manifold $\sssSym$ is transverse in $\rr^{2d}$, at the point $U(\xiTurn)$ and this can happen only if $m(e)=0$. If $u$ is asymmetric then the intersection between $\Wu_{V}(E)$ and $\Ws_{V}(E)$ is transverse, in $H_V^{-1}\bigl(V(E)\bigr)$, along the trajectory $U(\rr)$. The intersection of $\Wu_{V}(E)$ and $\Ws_{V}(E)$ is at least one-dimensional and the dimension of $H_V^{-1}\bigl(V(E)\bigr)$ is equal to $2d-1$. Again, the transversality can happen only if $m(e)=0$. This proves conclusion \cref{item:cor_main_bistable_pulse} of \cref{cor:main}. 

In all the cases considered above, the counting of the dimensions and the transversality imply that the intersections of the stable and unstable manifolds reduce to the smallest possible set, that is: the one-dimensional curve drawn by the trajectory $U$ for travelling fronts or asymmetric pulses, and the singleton $\bigl\{U(\xiTurn)\bigr\}$ defined by the turning point for symmetric pulses. By local compactness of the unstable manifolds, this implies that the trajectories of a given class are isolated from each other (even if a family of asymmetric standing pulses may accumulate on a non-degenerate --- and in this case non-transverse --- symmetric pulse). In particular, there is only a countable number of such trajectories. Conclusion \cref{item:cor_main_countable} of \cref{cor:main} is proved.

Finally, conclusion \cref{item:cor_main_robust} about the robustness of travelling fronts and standing pulses (the fact that they persist under small perturbations of the potential) follows from their transversality (that, is, the transversality of the intersections considered above). 
\section{Generic asymptotic behaviour for the profiles of bistable travelling fronts and of standing pulses stable at infinity}
\label{sec:generic_asympt_behaviour}
The goal of this \namecref{sec:generic_asympt_behaviour} is to prove \cref{thm:generic_asympt_behaviour} (and thus also conclusion \cref{item:cor_main_direction_of_approach_of_limits_at_ends_of_R} of \cref{cor:main}). 
\subsection{Asymptotic behaviour of profiles}
\label{subsec:asympt_behav_profiles}
Let $V_0$ denote a potential in $\vvvFull$, let $e_0$ denote a nondegenerate \emph{minimum} point of $V$, and let $c$ denote a nonnegative quantity (speed). As in \cref{subsec:eigenspaces_and_dimensions}, let $(u_1,\dots,u_d)$ denote an orthonormal basis of $\rr^d$ made of eigenvectors of $D^2V(e_0)$, and let $\mu_1,\dots,\mu_d$ denote the corresponding (positive) eigenvalues, with $\mu_1\le\dots\le\mu_d$. The statement ``the smallest eigenvalue of $D^2V(e_0)$ is simple'', in conclusion \cref{item:thm_generic_asympt_behaviour_simple_smallest_eig} of \cref{thm:generic_asympt_behaviour}, just mean that $\mu_1$ is smaller than $\mu_2$ (and thus also than all the other eigenvalues of $D^2V(e_0)$). Let us make this assumption. With the notation of \cref{subsec:eigenspaces_and_dimensions}, it follows that, for every $j$ in $\{2,\dots,d\}$, 
\[
\lambda_{j,-}<\lambda_{1,-}<0<\lambda_{1,+}<\lambda_{j,+}
\,;
\]
in other words, $\lambda_{1,-}$ and $\lambda_{1,+}$ are, among all the eigenvalues of $DF_{c,V}(E_0)$ (which are real), the closest ones to $0$ (here $E_0=(e_0,0_{\rr^d})$ is the equilibrium point of the flow $S_{c,V}$ corresponding to $e_0$). If a solution $\xi\mapsto u(\xi)$ of the differential system \cref{trav_wave_system_order_1} goes to $e_0$ as $\xi$ goes to $-\infty$ ($+\infty$), then one among the following two possible cases occurs (see \cref{prop:local_strong_stable_unstable_manifolds} below for a more precise statement): 
\begin{enumerate}
\item there exists a real quantity $K$ such that 
\[
\begin{aligned}
u(\xi) - e_0 &= K e^{\lambda_{1,+}\xi} u_1 + o_{\xi\to-\infty}(e^{\lambda_{1,+}\xi}) \\
\quad\text{(and }
u(\xi) - e_0 &= K e^{\lambda_{1,-}\xi} u_1 + o_{\xi\to+\infty}(e^{\lambda_{1,-}\xi})\,,\text{ respectively)};
\end{aligned}
\]
\label{item:case_slow_approach}
\item $u(\xi) - e_0 = o_{\xi\to-\infty}(e^{\lambda_{1,+}\xi})$ (and $u(\xi) - e_0 = o_{\xi\to+\infty}(e^{\lambda_{1,-}\xi})$, respectively).
\label{item:case_fast_approach}
\end{enumerate}
The words ``$u(\xi)$ approaches its limit (at $\pm\infty$) tangentially to the eigenspace corresponding to the smallest eigenvalue of $D^2V$ at this point'', used in conclusion \cref{item:cor_main_direction_of_approach_of_limits_at_ends_of_R} of \cref{cor:main} and in conclusion \cref{item:thm_generic_asympt_behaviour_approach_limits} of \cref{thm:generic_asympt_behaviour}, mean that case \cref{item:case_slow_approach} above occurs. As illustrated on \cref{fig:attractive_node} (see also \cref{fig:transverse_intersection}), approach of equilibria ``at the slowest possible rate'' (case \cref{item:case_slow_approach} above) is a generic feature among solutions of differential systems. The main goal of this section is thus to provide a formal proof that this feature is indeed generic (with respect to the potential $V$) for bistable travelling fronts and standing pulses stable at infinity of the parabolic system \cref{partial_differential_system}. 
\begin{figure}[htbp]
\centering
\includegraphics[width=.3\textwidth]{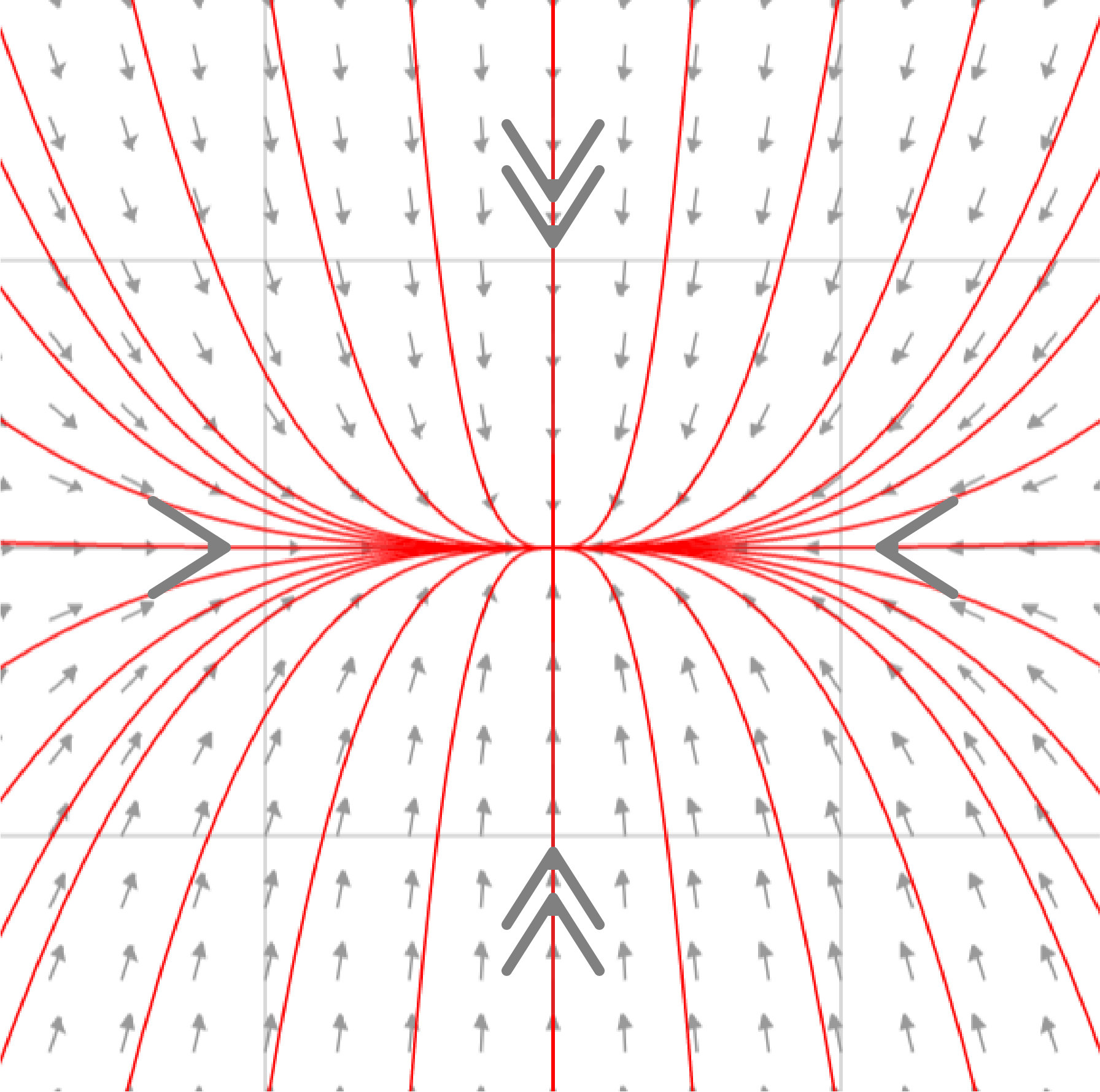}
\caption{Attractive node of a two-dimensional vector field. In the language of \cref{subsec:local_strongly_stable_unstable_manifolds}, the vertical axis is the ``strongly stable subspace'' of the equilibrium.}
\label{fig:attractive_node}
\end{figure}
\subsection{Local strongly stable and unstable manifolds when the speed \texorpdfstring{$c$}{c} is positive}
\label{subsec:local_strongly_stable_unstable_manifolds}
Let us keep the notation and assumptions of the previous \namecref{subsec:asympt_behav_profiles} and let us assume that $c$ is positive. The aim of this \namecref{subsec:local_strongly_stable_unstable_manifolds} is to provide a variant of \cref{prop:loc_stab_unstab_man_c_positive} devoted to the ``strongly'' local stable and unstable manifolds, which are characterized by a ``fast'' convergence (case \cref{item:case_fast_approach} above). Concerning the references, the same comments as in \cref{subsec:local_stable_unstable_manifolds_positive_speed} apply. 

Calling upon the notation of \cref{subsec:eigenspaces_and_dimensions}, let 
\[
\begin{aligned}
\Esu(E_0) &= \spanset\bigl(\{U_{2,+},\dots,U_{d,+}\}\bigr)
\quad\text{and}\quad
\Ess(E_0) = \spanset\bigl(\{U_{2,-},\dots,U_{d,-}\}\bigr)\,,\\
\text{and}\quad
\Em(E_0) &= \spanset\bigl(\{U_{1,-},U_{1,+}\}\bigr)
\end{aligned}
\]
(the superscripts ``su'', ``ss'', and ``m'' stand for ``strongly unstable'', ``strongly stable'', and ``mild'', respectively), and 
\[
\betasu = \lambda_{2,+}
\quad\text{and}\quad
\betass = \lambda_{2,-}
\,.
\]
As in \cref{subsec:local_stable_unstable_manifolds_positive_speed}, there exist norms $\normsu{\cdot}$ on $\Esu(E_0)$ and $\normss{\cdot}$ on $\Ess(E_0)$ such that inequalities \cref{normu_norms_characterization} (with ``su'' instead of ``u'' and ``ss'' instead of ``s'' everywhere) hold. For every positive quantity $r$, let us define the balls $\Bsu_{E_0}(r)$ and $\Bss_{E_0}(r)$ as in \cref{def_Bu_Bs_B_around_E0} (with the same substitutions ``u''$\leftarrow$``su'' and ``s''$\leftarrow$``ss''), let $\Bm_{E_0}(r)$ denote the closed ball centred at $E_0$ and with radius $r$, in the subspace $\Em(E_0)$, for the usual euclidean norm on these subspace, and let 
\[
\widebar{B}_{E_0}(r) = \bigl\{\Usu+\Uss+\Um : \Usu\in\Bsu_{E_0}(r) \text{ and } \Uss\in\Bss_{E_0}(r) \text{ and } \Um\in\Bm_{E_0}(r)\bigr\}
\,.
\] 
Let $\lambda_{3/2,-}$ and $\lambda_{3/2,+}$ denote real quantities satisfying
\[
\lambda_{2,-}<\lambda_{3/2,-}<\lambda_{1,-}
\quad\text{and}\quad
\lambda_{1,+}<\lambda_{3/2,+}<\lambda_{2,+}
\,.
\]
\begin{proposition}[local strong stable and unstable manifolds]
\label{prop:local_strong_stable_unstable_manifolds}
There exist a neighbourhood $\nu$ of $V_0$ in $\vvvFull$, a neighbourhood $\ccc$ of $c_0$ in $(0,+\infty)$ and a positive quantity $r$ such that, for every $(c,V)$ in $\ccc\times\nu$, in addition to the conclusions of Proposition 2.2, the following statements hold.
\item There exist $\ccc^k$-functions
\[
\wsuloc{c}{V}:\Bsu_{E_0}(r)\to \Bm_{E_0}(r)+\Bss_{E_0}(r)
\quad\text{and}\quad
\wssloc{c}{V}:\Bss_{E_0}(r)\to \Bm_{E_0}(r)+\Bsu_{E_0}(r)
\]
such that, if we consider the sets
\[
\begin{aligned}
\Wsuloc{c}{V}\bigl(E(V)\bigr) &= \left\{E(V) + \Usu + \wsuloc{c}{V}(\Usu) : \Usu\in \Bsu_{E_0}(r)\right\} \\ 
\text{and}\quad
\Wssloc{c}{V}\bigl(E(V)\bigr) &= \left\{E(V) + \Uss + \wssloc{c}{V}(\Uss) : \Uss\in \Bss_{E_0}(r)\right\}
\,,
\end{aligned}
\]
then, for every $U$ in $\widebar{B}_{E_0}(r)$ the following two assertions are equivalent:
\begin{enumerate}
\item $U$ is in $\Wsuloc{c}{V}\bigl(E(V)\bigr)$;
\item $S_{c,V}(\xi,U)-E(V)$ remains in $\widebar{B}_{E_0}(r)$ for all $\xi$ in $(-\infty,0]$ and 
\[
\abs{S_{c,V}(\xi,U)-E(V)} = o_{\xi\to -\infty}(e^{\lambda_{3/2,+}\xi})
\,;
\]
\end{enumerate}
and for every $U$ in $\widebar{B}_{E_0}(r)$ the following two assertions are equivalent:
\begin{enumerate}
\setcounter{enumii}{2}
\item $U\in\Wssloc{c}{V}\bigl(E(V)\bigr)$;
\item $S_{c,V}(\xi,U)-E(V)$ remains in $\widebar{B}_{E_0}(r)$ for all $\xi$ in $[0,+\infty)$ and 
\[
\abs{S_{c,V}(\xi,U)-E(V)} = o_{\xi\to +\infty}(e^{\lambda_{3/2,-}\xi})
\,.
\]
\end{enumerate}
\item Both differentials $D\wsuloc{c_0}{V_0}(0)$ and $D\wssloc{c_0}{V_0}(0)$ vanish, and both maps 
\[
\begin{aligned}
\ccc\times\nu\times \Bsu_{E_0}(r)\to \Bm_{E_0}(r)+\Bss_{E_0}(r), &\quad (c,V,\Usu)\mapsto \wsuloc{c}{V}(\Usu) \\
\text{and}\quad
\ccc\times\nu\times \Bss_{E_0}(r)\to \Bm_{E_0}(r)+\Bsu_{E_0}(r), &\quad (c,V,\Uss)\mapsto \wssloc{c}{V}(\Uss)
\end{aligned}
\]
are of class $\ccc^k$. 
\end{proposition}
\subsection{Idea of the proof of conclusion \texorpdfstring{\cref{item:thm_generic_asympt_behaviour_approach_limits} of \cref{thm:generic_asympt_behaviour}}{\ref{item:thm_generic_asympt_behaviour_approach_limits} of Theorem \ref{thm:generic_asympt_behaviour}}} 
\label{subsec:idea_proof_generic_asympt_behaviour}
The goal of this \namecref{subsec:idea_proof_generic_asympt_behaviour} is to provide a rough idea of the proof of \cref{thm:generic_asympt_behaviour}, more precisely of the main conclusion of this theorem which is conclusion \cref{item:thm_generic_asympt_behaviour_approach_limits} (the proof of conclusion \cref{item:thm_generic_asympt_behaviour_simple_smallest_eig}, carried out in the next \namecref{subsec:potentials_for_which_smallest_eigenvalue_Hessian_simple}, is straightforward). 

The proof of conclusion \cref{item:thm_generic_asympt_behaviour_approach_limits} is actually almost identical to the proof of \cref{thm:main}. Observe that, by contrast with the proof of \cref{thm:main}, only \emph{bistable} travelling fronts and standing pulses that are \emph{stable at infinity} need to be considered. In each case (bistable travelling fronts, symmetric and asymmetric standing pulses stable at infinity), the proof relies on applying Sard--Smale \cref{thm:Sard_Smale} to the same settings as in the proof of \cref{thm:main}, both for potentials that are quadratic past a certain radius and for the extension to general potentials, except for the following change:
\begin{enumerate}
\item either the unstable manifold of the left end equilibrium $E_-(V)$ is replaced by its strongly unstable manifold,
\label{item:replacement_Wu_by_Wsu}
\item or the stable manifold of the right end equilibrium $E_+(V)$ is replaced by its strongly stable manifold. 
\label{item:replacement_Ws_by_Wss}
\end{enumerate}
More precisely, both replacements have to be (separately) considered both for travelling fronts and asymmetric standing pulses, while only the first replacement is relevant for symmetric standing pulses. 

Let us see why such change (replacement) in the setting does not affect the validity of the two assumptions of \cref{thm:Sard_Smale}, and how its conclusions can be interpreted. Concerning assumption \cref{item:thm_sard_smale_condition_regularity} of \cref{thm:Sard_Smale}, this replacement leads to the following consequences:
\begin{enumerate}
\item either the dimension of the manifold denoted by $\mmm$ is decreased by $1$ (this is what happens for travelling fronts, be it with replacement \cref{item:replacement_Wu_by_Wsu} or \cref{item:replacement_Ws_by_Wss}, for symmetric standing pulses with replacement \cref{item:replacement_Wu_by_Wsu}, and for asymmetric standing pulses with replacement~\cref{item:replacement_Wu_by_Wsu}), 
\item or the dimension of the manifold denoted by $\www$ is decreased by $1$ (this is what happens for asymmetric standing pulses with replacement \cref{item:replacement_Ws_by_Wss}). 
\end{enumerate}
In each of these cases, the dimension of the arrival manifold $\nnn$ is unchanged, and as a consequence, the difference $\dim(\mmm)-\codim(\www)$ is exactly decreased by $1$. More precisely, since only bistable travelling fronts and standing pulses stable at infinity are considered, this difference is actually exactly equal to $-1$. Assumption \cref{item:thm_sard_smale_condition_regularity} of \cref{thm:Sard_Smale} is therefore still satisfied. 

Concerning assumption \cref{item:thm_sard_smale_condition_transversality} of \cref{thm:Sard_Smale}, it is also fulfilled in each of these cases, due to the key following observation: in the proof of each of the three lemmas proving that this assumption holds (\cref{lem:reach_all_directions_orthogonal_to_UzeroPrime_of_xiZero,lem:reach_all_directions_symm_stand_pulse,lem:reach_all_directions_asymm_stand_pulse}), the freedom provided by the variables $\bu$ and $\bs$ is not used --- only the freedom provided by the time variable $\xi$ and by the potential $V$ are. As a consequence, the fact that the unstable manifold of $E_-(V)$ is replaced by its strongly unstable manifold does not affect the validity of the conclusion of the lemma, and neither does the fact that the stable manifold of $E_+(V)$ is replaced by its strongly stable manifold. In other words, the key assumption \cref{item:thm_sard_smale_condition_transversality} of \cref{thm:Sard_Smale} still holds. 

In each case and for each of the two replacements \cref{item:replacement_Wu_by_Wsu,item:replacement_Ws_by_Wss}, the conclusions of \cref{thm:Sard_Smale} thus still hold, and ensure that, locally generically with respect to $V$, the profiles of travelling fronts or of (a)symmetric standing pulses locally correspond to transverse intersections between the image of $m\mapsto\Phi(m,V)$ and $\www$ in $\nnn$. But the fact that $\dim(\mmm)-\codim(\www)$ is now equal to $-1$ actually precludes the very existence of such transverse intersections. In other words, locally generically with respect to $V$, profiles of bistable travelling fronts or of (a)symmetric pulses stable at infinity approaching their limit at $-\infty$ through its strongly stable manifold or their limit at $+\infty$ through its strongly stable manifold do simply (locally) not exist, which is the intended conclusion. The emptiness of such a transverse intersection due to the value $-1$ of the difference $\dim(\mmm)-\codim(\www)$ is illustrated by \cref{fig:transverse_intersection}. 
\begin{figure}[htbp]
\centering
\includegraphics[width=.8\textwidth]{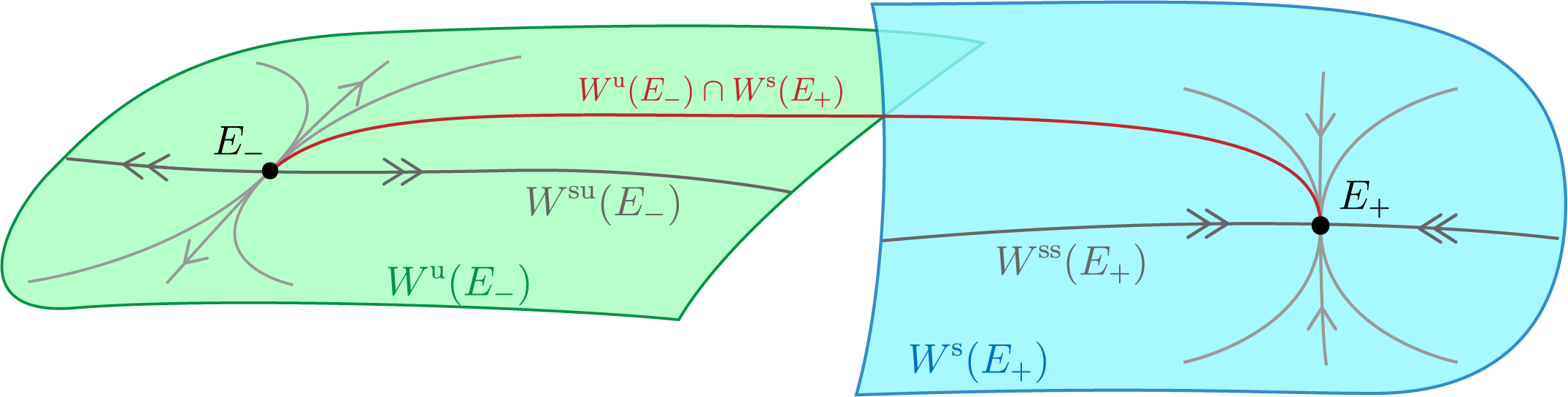}
\caption{Whereas the sum of the dimensions of $\Wu(E_-)$ and $\Ws(E_+)$ has the minimal value for a nonempty transverse intersection between these two manifolds to exist, for $\Wsu(E_-)$ and $\Ws(E_+)$ (or for $\Wu(E_-)$ and $\Wss(E_+)$) this sum is smaller, so that a transverse intersection between such manifolds must be empty. This figure actually depicts the intersection defining a transverse asymmetric bistable standing front, but the same principle applies for bistable travelling fronts, and  elementary symmetric (or transverse asymmetric) standing pulses that are stable at infinity.}
\label{fig:transverse_intersection}
\end{figure}

The remaining arguments, ensuring the first extension to global statements for potentials quadratic past a certain radius (\cref{subsec:reduction,subsec:notation_and_statements_sym_stand_pulses,subsec:notation_and_statements_asym_stand_pulses}), and then the second extension to general potentials (\cref{subsec:proof_thm_main_for_nonzero_speeds,subsec:proof_thm_main_zero_speed}), are unchanged.

To complete these arguments, a few milestones of the proof for travelling fronts are detailed in \cref{subsec:generic_asympt_behaviour_bistable_trav_fronts} below.
\subsection{Proof of conclusion \texorpdfstring{\cref{item:thm_generic_asympt_behaviour_simple_smallest_eig} of \cref{thm:generic_asympt_behaviour}}{\ref{item:thm_generic_asympt_behaviour_simple_smallest_eig} of Theorem \ref{thm:generic_asympt_behaviour}}}
\label{subsec:potentials_for_which_smallest_eigenvalue_Hessian_simple}
Let $R$ denote a positive quantity, let us recall the notation $\vvvQuad{R}$ introduced in \cref{notation_vvvQuad_of_R} and $\vvvQuadMorse{R}$ introduced in \cref{def_vvv_Quad_R_Morse}, and let us consider the set
\[
\begin{aligned}
\vvvQuadMorseSsEig{R} = \bigl\{&V\in\vvvQuadMorse{R}: \text{ at every minimum point of $V$,}\\
&\text{the smallest eigenvalue of $D^2V$ is simple}\bigr\}
\end{aligned}
\]
(the subscript ``ss-eig'' stands for ``simple smallest eigenvalue''). 
\begin{proposition}
\label{prop:vvvQuadMorseSsEig_of_R_dense_open_in_vvvQuadMorse_of_R}
The set $\vvvQuadMorseSsEig{R}$ is a dense open subset of $\vvvQuadMorse{R}$ (and thus of $\vvvQuad{R}$). 
\end{proposition}
\begin{proof}
Openness follows from the continuity of the roots (eigenvalues of $D^2V$ at a minimum point) of a polynomial with respect to its coefficients. To prove the density, let $V$ be in $\vvvQuadMorse{R}$, and let us assume that there exists a minimum point $e$ of $V$ such that the smallest eigenvalue $\mu_1$ of $D^2V(e)$ is \emph{not} simple. Let $\delta$ denote a positive quantity, small enough so that the closed ball $\widebar{B}_{\rr^d}(e,\delta)$ contains no critical point of $V$ but $e$. Let $\rho$ denote a smooth function $[0,+\infty)\to\rr$ satisfying
\[
\rho(r)=1
\quad\text{for}\quad
r\text{ in }[0,1/2]
\quad\text{and}\quad
\rho(r)=0
\quad\text{for}\quad
r\text{ in }[1,+\infty)
\,,
\]
and let $\varepsilon$ denote a small positive quantity to be chosen below. Let $u_1$ denote an unit eigenvector of $D^2V(e)$ associated to $\mu_1$, and let us consider the perturbed potential $\Vpert$ defined as:
\[
\Vpert(u) = V(u) - \frac{\varepsilon}{2}\bigl((u-e)\cdot u_1\bigr)^2 \rho\bigl(\abs{u-e}/\delta\bigr)
\,.
\]
Then, $e$ is still a critical point of $\Vpert$ and, for every $v$ in $\rr^d$, 
\[
D^2\Vpert(e)(v,v) = D^2 V(e)(v,v) - \varepsilon (v\cdot u_1)^2
\,.
\]
As a consequence, $u_1$ is still an eigenvector of $D^2\Vpert(e)$, the corresponding eigenvalue $\mu_1-\varepsilon$ is simple, and the other eigenvalues of $D^2\Vpert(e)$ are the same as those of $D^2 V(e)$ (the difference $D^2\Vpert(e)-D^2 V(e)$ vanishes on the orthogonal subspace to $u_1$ in $\rr^d$), these other eigenvalues are therefore larger than $\mu_1-\varepsilon$. In addition, if $\varepsilon$ is small enough, then $\mu_1-\varepsilon$ is positive (so that $e$ is still a minimum point of $\Vpert$) and the closed ball $\widebar{B}_{\rr^d}(e,\delta)$ contains no critical point of $\Vpert$ but $e$. The same procedure, repeated for each minimum point of $V$ such that the smallest eigenvalue of $D^2V$ at this minimum point is not simple, provides an arbitrarily small perturbation of $V$ belonging to $\vvvQuadMorseSsEig{R}$, and therefore proves the intended density. 
\end{proof}
Let $\vvvMorseSEig$ denote the subset of $\vvvFull$ containing Morse potentials $V$ such that, at every minimum point point of $V$, the smallest eigenvalue of the Hessian $D^2V$ at this minimum point is simple. 
Proceeding as in \cref{subsec:proof_thm_main_no_standing_front}, the same arguments as in the proof of \cref{prop:vvvQuadMorseSsEig_of_R_dense_open_in_vvvQuadMorse_of_R} above show that this set $\vvvMorseSEig$ is a generic subset of $\vvvFull$, which proves conclusion \cref{item:thm_generic_asympt_behaviour_simple_smallest_eig} of \cref{thm:generic_asympt_behaviour}. 
\subsection{Proof of conclusion \texorpdfstring{\cref{item:thm_generic_asympt_behaviour_approach_limits} of \cref{thm:generic_asympt_behaviour}}{\ref{item:thm_generic_asympt_behaviour_approach_limits} of Theorem \ref{thm:generic_asympt_behaviour}} for bistable travelling fronts}
\label{subsec:generic_asympt_behaviour_bistable_trav_fronts}
The aim of this \namecref{subsec:generic_asympt_behaviour_bistable_trav_fronts} is to complete the idea of the proof of conclusion \cref{item:thm_generic_asympt_behaviour_approach_limits} of \cref{thm:generic_asympt_behaviour} provided in \cref{subsec:idea_proof_generic_asympt_behaviour} with a few milestones of this proof, in the case of travelling fronts (only). 

As for conclusion \cref{item_thm:main_travelling_front} of \cref{thm:main}, the first goal is to prove the intended conclusion among potentials that are quadratic past a certain (positive) radius $R$. This is stated by the following proposition, which is an extension of \cref{prop:global_generic_transversality_travelling_fronts}. It calls upon the notation $\fff_V$ introduced in \cref{notation_fff_V}. 
\begin{proposition}
\label{prop:global_generic_asympt_behav_trav_fronts}
There exists a generic subset of $\vvvQuad{R}$, included in $\vvvQuadMorseSsEig{R}$ such that, for every potential $V$ in this subset, every travelling front $(c,u)$ in $\fff_V$ is transverse, bistable, and its profile $u$ approaches its limit at $+\infty$ ($-\infty$) tangentially to the eigenspace corresponding to the smallest eigenvalue of $D^2V$ at this point. 
\end{proposition}
\subsubsection{Reduction to a local statement}
Let $V_0$ denote a potential function in $\vvvQuadMorseSsEig{R}$, and let $e_{-,0}$ and $e_{+,0}$ denote non-degenerate minimum points of $V_0$. Let us consider the neighbourhood $\nuRobust(V_0,e_{-,0},e_{+,0})$ of $V_0$ introduced in \cref{subsec:reduction}, and let us denote by $\tildeNuRobust(V_0,e_{-,0},e_{+,0})$ the intersection $\nuRobust(V_0,e_{-,0},e_{+,0})\cap\vvvQuadMorseSsEig{R}$. The following proposition is a variant (extension in the case of bistable travelling fronts) of \cref{prop:local_generic_transversality_travelling_fronts_given_critical_points_speed}. The notation is similar, except for the ``tilde'' added to the symbols of the various sets, in order to differentiate them for the corresponding sets introduced in \cref{prop:local_generic_transversality_travelling_fronts_given_critical_points_speed}.
\begin{proposition}
\label{prop:local_generic_asympt_behav_trav_fronts}
For every positive speed $c_0$, there exist a neighbourhood $\tilde{\nu}_{\VZeroeMinusZeroePlusZerocZero}$ of $V_0$ in $\vvvQuad{R}$, included in $\tildeNuRobust(V_0,e_{-,0},e_{+,0})$, a neighbourhood $\tilde{\ccc}_{\VZeroeMinusZeroePlusZerocZero}$ of $c_0$ in $(0,+\infty)$, and a generic subset $\tildeNuThingGen{\VZeroeMinusZeroePlusZerocZero}$ of $\tilde{\nu}_{\VZeroeMinusZeroePlusZerocZero}$ such that, for every $V$ in $\tildeNuThingGen{\VZeroeMinusZeroePlusZerocZero}$, every front travelling at a speed $c$ in $\tilde{\ccc}_{\VZeroeMinusZeroePlusZerocZero}$ and connecting $e_-(V)$ to $e_+(V)$, for the potential $V$, is transverse and its profile $u$ approaches its limit at $+\infty$ ($-\infty$) tangentially to the eigenspace corresponding to the smallest eigenvalue of $D^2V$ at this point. 
\end{proposition}
\begin{proof}[Proof that \cref{prop:local_generic_asympt_behav_trav_fronts} yields \cref{prop:global_generic_asympt_behav_trav_fronts}]
\Cref{prop:global_generic_transversality_travelling_fronts} already ensures the existence of a generic subset $\vvvQuadTransvfff{R}$ of $\vvvQuad{R}$ such that, for every potential function $V$ in this subset, every travelling front $(c,u)$ in $\fff_V$ is transverse. According to the arguments of \cref{subsec:proof_thm_travelling}, such a front is necessarily bistable. Thus, only the conclusion of \cref{prop:local_generic_asympt_behav_trav_fronts} relative to the asymptotic behaviour of the profile remains to be proved. 

To this end, the arguments are the same as in \cref{subsec:reduction}. We may introduce the sets $\tilde{\nu}_{\VZerocZero}$ and $\ccc_{\VZerocZero}$ and $\nuThingGen{\VZerocZero}$, defined exactly as the corresponding sets (without tilde) in \cref{three_sets_ccc_nu_nuGen} (with $\nuRobust(V_0,e_{-,0},e_{+,0})$ replaced with $\tildeNuRobust(V_0,e_{-,0},e_{+,0})$), and the same remaining arguments (replacing $\vvvQuadMorse{R}$ with $\vvvQuadMorseSsEig{R}$) show the existence of a generic subset of $\vvvQuad{R}$, included in $\vvvQuadMorseSsEig{R}$, such that, for every potential $V$ in this subset, every \emph{bistable} travelling front $(c,u)$ in $\fff_V$ is transverse and its profile $u$ approaches its limit at both ends of $\rr$ according to the intended conclusion. Intersecting this generic subset with the one provided by \cref{prop:global_generic_transversality_travelling_fronts} provides a generic subset of $\vvvQuad{R}$ for which all conclusions of \cref{prop:global_generic_asympt_behav_trav_fronts} hold.
\end{proof}
\subsubsection{Proof of the local statement}
The proof of \cref{prop:local_generic_asympt_behav_trav_fronts} may be derived from the proof of \cref{prop:local_generic_transversality_travelling_fronts_given_critical_points_speed}, up to a few changes and thanks to some key arguments, all of which are exposed in \cref{subsec:idea_proof_generic_asympt_behaviour} above. 
\subsubsection{Extension to all potentials}
The extension to all potentials is obtained by applying the same strategy as in \cref{subsec:proof_thm_main_for_nonzero_speeds}. Let us recall the notation $\fff_{V,R}$ introduced in \cref{notation_fff_V_R}. The same arguments as in \cref{subsec:proof_thm_main_for_nonzero_speeds} show that the intended extension is a consequence of the following extension of \cref{prop:genericity_transv_tf_up_to_R}. 
\begin{proposition}
\label{prop:genericity_transv_tf_min_rate_up_to_R}
For every positive quantity $R$, there exists a generic subset\\
$\vvvFullTransvfffMinRate{R}$ of $\vvvFull$, included in $\vvvMorseSEig$, such that, for every potential $V$ in this subset, every travelling front $(c,u)$ in $\fff_{V,R}$ is transverse, bistable, and approaches its limit at $+\infty$ ($-\infty$) tangentially to the eigenspace corresponding to the smallest eigenvalue of $D^2V$ at this point. 
\end{proposition} 
\begin{proof}
\Cref{prop:genericity_transv_tf_up_to_R} already provides a generic subset $\vvvFullTransvfff{R}$ of $\vvvFull$ such that, for every potential $V$ in this subset, every travelling front $(c,u)$ in $\fff_{V,R}$ is transverse, and therefore bistable (\cref{subsec:proof_thm_travelling}). Therefore, only the conclusion relative to the asymptotic behaviour of the profiles remains to be proved. 

The proof of this conclusion is a variation of the proof of \cref{prop:genericity_transv_tf_up_to_R} and follows the ideas exposed in \cref{subsec:idea_proof_generic_asympt_behaviour}: for some potential $V_0$ in $\vvvQuad{(R+1)}$ and for some non-degenerate minimum points $e_{-,0}$ and $e_{+,0}$ of $V$, and for every nonnegative integer $N$, two variants of the set $\mmm_N$ defined in \cref{def_mmm_N} (and of the open subset $\ooo_{\VZeroeMinusZeroePlusZerocZeroN}$ defined in \cref{def_ooo_VZeroeMinusZeroePlusZerocZeroN}) can be introduced: one where $\bbbu$ is replaced by $\bbbsu$, and one where $\bbbs$ is replaced by $\bbbss$. In each of theses two cases, the condition ``$\Phi\bigl(\mmm_N,V\bigr)$ is transverse to $\www$ in $\nnn$'' can be read as ``the intersection between $\Phi\bigl(\mmm_N,V\bigr)$ and $\www$ is empty'', due to the missing dimension induced by the change in each of theses variants. Then, replacing the open subset $\ooo_{\VZeroeMinusZeroePlusZerocZeroN}$ by the intersection of its two variants, the remaining arguments are exactly the same. This proves \cref{prop:genericity_transv_tf_min_rate_up_to_R} (and therefore also completes the proof of conclusion \cref{item:thm_generic_asympt_behaviour_approach_limits} of \cref{thm:generic_asympt_behaviour} for bistable travelling fronts). 
\end{proof}
\printbibliography
\bigskip
\RomainsSignature
\bigskip

\mySignature
\end{document}